\documentclass[11pt]{amsart}
\usepackage[T1]{fontenc}
\usepackage[utf8]{inputenc}

\usepackage[margin=3.5cm]{geometry}

\usepackage{lmodern, amsfonts,amsmath,amstext,amsbsy,amssymb,
amsopn,amsthm,upref,eucal,mathptmx,mathtools,url,thmtools}

\DeclareMathAlphabet{\matholdcal}{OMS}{cmsy}{m}{n}   
\usepackage{dutchcal}
\usepackage{mathrsfs}
\usepackage{tcolorbox}
\usepackage{import}
\usepackage{float}
\usepackage{enumitem}
\usepackage[hidelinks]{hyperref}
\usepackage{caption}

\usepackage[nameinlink]{cleveref}
\crefalias{bigthm}{theorem}

\RequirePackage{xcolor} 
\definecolor{halfgray}
{gray}{0.55}
\definecolor{webgreen}
{rgb}{0,0.4,0}
\definecolor{webbrown}
{rgb}{.8,0.1,0.1}
\definecolor{red}
{rgb}{1,0,0}

\definecolor{black}{RGB}{0,0,0}
\definecolor{orange}{RGB}{230,159,0}
\definecolor{skyblue}{RGB}{86,180,233}
\definecolor{bluishgreen}{RGB}{0,158,115}
\definecolor{yellow}{RGB}{240,228,66}
\definecolor{blue}{RGB}{0,114,178}
\definecolor{vermillion}{RGB}{213,94,0}
\definecolor{reddishpurple}{RGB}{204,121,167}

\usepackage{microtype}
\usepackage{tikz}
\usetikzlibrary{shapes.geometric}
\usetikzlibrary{calc,math}
\usetikzlibrary{decorations.markings}
\usepackage{soul} 

\newcommand{\cF}{\mathcal{F}}

\newcommand{\cL}{\mathcal{L}}

\newcommand{\cA}{\mathcal{A}}

\newcommand{\cB}{\mathcal{B}} 

\newcommand{\N}{\mathbb{N}}
\newcommand{\Z}{\mathbb{Z}}

\newcommand{\R}{\mathbb{R}}
\newcommand{\C}{\mathbb{C}}

\newcommand{\tS}{\widetilde{S}}
\newcommand{\tSz} {\widetilde{S_0}}

\newcommand{\tT}{\widetilde{T}}
\newcommand{\tphi}{\widetilde{\phi}}
\newcommand{\tx}{\tilde{x}}
\newcommand{\ty}{\tilde{y}}
\newcommand{\cC}{\mathscr{C}}
\newcommand{\sC}{C_\sharp}

\newcommand*{\diff}{\mathop{}\!\mathrm{d}}
\newcommand{\one}{{\rm 1\mskip-4mu l}}

\DeclareMathOperator{\vol}{vol}
\DeclareMathOperator{\Deck}{Deck}
\DeclareMathOperator{\abel}{ab}
\DeclareMathOperator{\dist}{d}
\DeclareMathOperator{\supp}{supp}
\DeclareMathOperator{\Hom}{Hom}
\DeclareMathOperator{\cl}{cl}
\DeclareMathOperator{\SL}{SL}
\DeclareMathOperator{\range}{range}
\DeclareMathOperator{\ess}{ess}
\DeclareMathOperator{\Aut}{Aut}
\DeclareMathOperator{\id}{id}
\DeclareMathOperator{\reg}{reg}
\DeclareMathOperator{\bdd}{bdd}

\DeclareMathOperator{\Realpart}{Re}
\renewcommand{\Re}{\Realpart}
\DeclareMathOperator{\Imaginarypart}{Im}
\renewcommand{\Im}{\Imaginarypart}

\newtheorem{theorem}{Theorem}[section]
\newtheorem {lemma}[theorem]{Lemma}

\newtheorem {proposition}[theorem]{Proposition}
\newtheorem{corollary}[theorem]{Corollary}

\newtheorem{definition}[theorem]{Definition}
\newtheorem{bigthm}{Theorem}

\theoremstyle{definition}
\newtheorem{remark}[theorem]{Remark}
\newtheorem{notation}[theorem]{Notation}

\providecommand{\given}{}
\newcommand{\SetSymbol}[1][]{%
	\nonscript\:\mathord{:}
	\allowbreak
	\nonscript\:
	\mathopen{}}
\DeclarePairedDelimiterX{\Set}[1]\{\}{%
	\renewcommand{\given}{\SetSymbol[\delimsize]}
	#1
}

\makeatletter
\newcount\repeats
\newcommand{\rep}[2]{
  \begingroup
  \repeats=\z@
  \@whilenum\repeats<#1\do{#2 \hspace{-5pt} \advance\repeats 1}%
  \endgroup
}
\makeatother

\usepackage[normalem]{ulem}
\usepackage{cancel}

\begin{document}

\date{\today}

\author[M. Artigiani]{Mauro Artigiani}
\address{Departamento de Matemáticas\\
	Universidad Nacional de Colombia\\
	Carrera 30 No.\ 45-03\\
	111321\\
	Bogotá\\
	Colombia}
\email{martigiani@unal.edu.co}

\author[R. Castorrini]{Roberto Castorrini}
\address{Scuola Normale Superiore\\
	Piazza dei Cavalieri 7\\
	Pisa \\
	Italy\\
	56126
}
\email{roberto.castorrini@gmail.com}

\author[D. Ravotti]{Davide Ravotti}
\address{Université de Lille\\
	CNRS\\
	UMR 8524 - Laboratoire Paul Painlevé\\
	F-59000 Lille\\
	France
}
\email{davide.ravotti@gmail.com}

\author[Y. Tumarkin]{Yuriy Tumarkin}
\address{Institut für Mathematik\\
	Universität Zürich\\
	Winterthurerstrasse 190\\
	CH-8057 Zürich \\
	Switzerland
}
\email{yuriy.tumarkin@math.uzh.ch}

\title[Maharam-Pollicott-Ruelle resonances and self-similar translation flows]
{ Maharam-Pollicott-Ruelle resonances and self-similar translation flows on Abelian covers}

\begin{abstract}
	We study self-similar translation flows on $\mathbb Z^d$-covers of compact translation surfaces. 
	Our main goal is to investigate their ergodic properties with respect to general Maharam measures. To this end, we develop a renormalization approach based on a family of twisted transfer operators associated with the renormalizing pseudo-Anosov map acting on anisotropic spaces of distributions. 
	We describe the discrete spectrum of these operators in terms of the action of the pseudo-Anosov on suitable twisted cohomology groups.
	We further show that the resonant states of the dual operator corresponding to peripheral eigenvalues give rise to Maharam distributions which are invariant under the translation flow. Motivated by this correspondence, we refer to these eigenvalues as \emph{Maharam-Pollicott-Ruelle resonances}.
	As applications, we derive asymptotic formulas for ergodic integrals of smooth observables at Maharam-generic points, prove a central limit theorem for the associated Frobenius cocycle, and compute the Hausdorff dimension of Maharam measures.

\end{abstract}

\maketitle

\tableofcontents

\section{Introduction}

\subsection{Motivation}
Translation surfaces arise naturally at the intersection of several areas of mathematics, including dynamical systems, geometry, and complex analysis. Informally, a translation surface is obtained by gluing polygons in the plane along parallel edges via translations, producing a surface with a metric that is everywhere flat apart from a finite number of points, where the total angle is a positive integer multiple of $2\pi$ (conical singularities). 
Straight-line (or directional) flows on these surfaces, called translation flows, provide a natural class of zero-entropy dynamical systems which exhibit remarkably rich and subtle behavior. 

Compact translation surfaces have been extensively studied in the last fifty
years, using techniques from ergodic theory, algebraic and hyperbolic geometry
and low-dimensional topology. The dynamical picture of the straight-line flow in
a given direction is quite well understood: generically the flow is uniquely
ergodic~\cite{Masur, Veech}, weakly mixing~\cite{AF}, and cannot be strongly
mixing~\cite{Katok}. Informally, these results are obtained using
renormalization to speed up the trajectories, since the straight-line flow on
a translation surface is usually classified as a \emph{renormalizable} flow.
In this paper, we will consider the special case of \emph{self-similar} translation flows, where the renormalization dynamics consists of the action of an affine automorphism of the surface; namely, a \emph{linear pseudo-Anosov} automorphism.
More precisely, we study $\Z^d$-covers $\tS \to S$ of a translation surface $S$, assuming that the base surface $S$ has a pseudo-Anosov  
transformation which lifts to the cover $\tS$. 

In
the case of abelian covers, unique ergodicity no longer holds~\cite{Schmidt,
ST}. The relevant measures, apart from Lebesgue, are \emph{Maharam}
measures~\cite{ANSS,HHW,Hooper,Tum}. In this paper, we study Maharam measures
and ergodic integrals for their generic points. We remark that the non self-similar
case seems to be fundamentally different, and that there are indications that
not even ergodicity should be generically expected, see~\cite{FU}.

We introduce Maharam \emph{distributions}, which generalize Maharam measures. We
give a spectral description of these distributions in
\Cref{bigthm:spectral_Maharam} and relate them to the geometry of the surfaces in
\Cref{bigthm:geometric_Maharam}. The spectral description comes from the
quasi-compactness of a weighted transfer operator, which we prove in
\Cref{bigthm:quasi_compactness}. The link with geometry is given by the fact
that eigenvalues in the spectrum come from the action of the
pseudo-Anosov on a \emph{twisted} cohomology space on the base surface.\\

Since the cover $\tS$ has infinite Lebesgue measure, the
Birkhoff ergodic theorem no longer holds~\cite{Aaronson}. However, one can
hope to recover a description of an ergodic, conservative, measure preserving
flow $\phi_t$ on the $\sigma$-finite space $\tS$ of infinite measure $\mu$ for an observable $f$ by
\begin{equation}\label{eq:asymptotic_ergodic_integral}
	\int_0^t f\circ\phi_s(x)\, \diff s \sim  a(t)\Phi_t(x)\mu(f),
\end{equation}
where $a(t)$ is the ``correct'' asymptotic growth, $\Phi_t$ is an oscillating
term, which does not depend on the function $f$, and that converges in some
weaker sense.  We obtain such a description for the Maharam
measures in \Cref{bigthm:ergodic_integrals}.
Analogous results were previously obtained for several parabolic
systems in~\cite{LS2,LS3,ADDS,BFRT,CR}; in those works, however, the reference measure is given by the pullback of the invariant measure on the base surface. In contrast, \Cref{bigthm:ergodic_integrals} is concerned with Maharam measures that are distinct from Lebesgue measure.
A key difference in our setting is that these measures are \emph{squashable}. In particular, this property rules out the existence of a generalized strong law of large numbers (GLLN), see~\cite{Aaronson}. The absence of a GLLN is reflected in the form of the oscillatory factor $\Phi_t$ appearing in the asymptotic description \eqref{eq:asymptotic_ergodic_integral} of ergodic integrals, which differs significantly from that found in the aforementioned works. We also note that the asymptotic growth rate $a(t)$ occurs at a different order. \\

We were recently informed of the work in progress \cite{DS_new}, in which Dolgopyat and Sarig have also obtained results for ergodic integrals with respect to Maharam measures, specifically on the infinite staircase surface, as well as describing their generic points in this case. 

\subsection{Some previous results}
The transfer operator is a fundamental tool in hyperbolic dynamical systems. Its
study, on appropriate spaces of observables, allows one to prove mixing, decay
of correlations and many other, finer, statistical properties of the dynamical
system. 
Since for a self-similar translation flow the renormalization scheme consists in the iterations of a hyperbolic map $T$ (the linear pseudo-Anosov automorphism), one can use the transfer operator associated to $T$ in order to
study the renormalization and try to obtain insights on the translation flow as
a byproduct. 
This was first carried out in a proof of concept by Giulietti and
Liverani in~\cite{GL} and then by Faure, Gouëzel and Lanneau in~\cite{FaGoLa}.
Very recently, the second and third authors implemented this scheme in the setting of
horocycle flows on $\Z^d$-covers of negatively curved surfaces in~\cite{CR},
and the third author with Bruin, Fougeron and Terhesiu \cite{BFRT} in the same context as for the present paper, but limited to the study of points generic for the Lebesgue measure.
This latter work generalizes and strengthens the results in~\cite{ADDS}, which apply to a $\Z$-cover of a punctured torus (a staircase surface). \\

In the compact case, the
transfer operator associated to $T$, acting on appropriate Banach spaces, is
quasi-compact~\cite{FaGoLa}. It is a remarkable phenomenon that, in this case,
the transfer operator governs the dynamics of the straight-line
flow along the leaves of the stable foliation. In particular, since the flow
in this direction is uniquely ergodic, i.e., the only invariant measure for
the flow is the one induced by Lebesgue measure in coordinates, one knows that
the ergodic integrals grow linearly in time for all (nice) observables. The
deviations of these averages were studied by Forni in his seminal
paper~\cite{Forni:deviations} and are governed by a set of distributions
invariant with respect to the vector field tangent to stable leaves. These
distributions are related to eigenvalues of the transfer operator associated to
$T$, see~\cite{FaGoLa}.\\

As we mentioned above, in the case of abelian covers of compact translation surfaces,
we have infinitely many invariant measures, which are the
\emph{Maharam} measures. These are measures that are scaled under the action of Deck transformations, 
first studied by Maharam in~\cite{Maharam}. More precisely, given a homomorphism $\psi \in \Hom(\Deck,\R)$, 
a $\psi$-Maharam measure is a measure $\mu_\psi$ on $\tS$ such that for any $D\in \Deck$
\begin{equation}
\label{eq:Maharam_def}
\mu_\psi \circ D = e^{\psi(D)} \mu_\psi.
\end{equation}
We will identify $\Hom(\Deck,\R) = \Hom(\Z^d,\R)$ with $\R^d$ and denote $\mu_\psi$ by $\mu_r$ for $r\in \R^d$.

Maharam measures for $\Z^d$-covers have been studied in the following settings:
\begin{enumerate}[label=(\roman*)]
\item In the setting of the horocycle flow on hyperbolic surfaces, the work of
Babillot, Ledrappier, Sarig and Schapira
(\cite{BL,Sarig04,LS,LS2,SarigSchapira}) showed that the ergodic invariant
measures are precisely the Maharam measures (also known as
\emph{Babillot-Ledrappier measures} in this context) and classified their
generic points.
\item Aaronson, Nakada, Sarig and Solomyak studied some classes of $\Z^d$-skew-products over symbolic systems in~\cite{ANSS}.
\item Pollicott and Sharp studied the invariant measures for the stable foliation of a pseudo-Anosov diffeomorphism satisfying some conditions in~\cite{PS}.
\item For $\Z^d$-covers of translation surfaces, Maharam measures have been studied by Hooper, Hubert and Weiss in~\cite{HHW} and by Hooper in~\cite{Hooper}. 
In our setting of self-similar covers, the results of~\cite{Tum} show that the Maharam measures are precisely the ergodic invariant Radon measures for the flow $\{\phi_t\}_{t\in \R}$.
\end{enumerate}

\subsection{Setting}
We now briefly describe our setting and we present our main results. We refer to~\cite{AM} for a general introduction to
compact translation surfaces and to~\cite{DHV} for infinite ones.

Let $(S, \alpha)$ be a compact translation surface, that is $S$ is a compact
Riemann surface and $\alpha$ is a non identically zero holomorphic $1$-form on
$S$. Let $\Sigma$ denote the set of zeros of $\alpha$. We let $\vol =
\frac{\imath}{2}\alpha \wedge \overline{\alpha}$ denote the standard area form
on $S_0 \coloneqq S \setminus \Sigma$.

The 1-form $\alpha$ defines a flat metric on $S_0$, as well as two orthonormal
vector fields $X$ and $Y$, called the horizontal and vertical vector fields, and
the associated pair of measured foliations. We assume that these foliations are
precisely the stable and unstable foliations for a linear pseudo-Anosov
diffeomorphism $T \colon S \to S$. We denote by $\lambda > 1$ the stretch factor
of $T$. Finally, we let $\{\phi_t\}_{t\in \R}$ be the translation flow in the
horizontal direction. 

We are interested in abelian covers of $S$. Let $\mathcal{p}\colon \tS \to S$ be
a locally compact cover of $S$ (we refer to \Cref{sec:twisted_hilbert_spaces}
for precise definitions) with group of deck transformations $\Deck$ isomorphic
to $\Z^d$. Let $\{\widetilde{\phi_t}\}_{t\in \R}$ be the translation flow on
$\tS$ defined by $\mathcal{p}\circ \widetilde{\phi_t} = \phi_t \circ
\mathcal{p}$. We assume that we can lift $T$ to $\tS$; more precisely, we assume
that there exists a linear pseudo-Anosov diffeomorphism $\widetilde{T}\colon \tS
\to \tS$ so that $T \circ \mathcal{p} = \mathcal{p}\circ \widetilde{T}$ and
$\widetilde{T}$ commutes with all elements of $\Deck$. We refer to
\Cref{sec:examples} for explicit examples, which include some staircase
surfaces~\cite{HWS, HHW} and some examples of the Ehrenfest wind-tree surface,
see~\cite{Ehrenfest,FU,Hooper} and the references therein.

\subsection{Main results I : renormalization for Maharam-generic points}

One of the main contributions of this paper is the extension of the transfer operator formalism to the study of the translation flow dynamics with respect to general Maharam measures. We expect that this framework provides a flexible blueprint for further generalizations beyond the present setting.
The technical core of our work is to study a family of \emph{transfer operators} associated to the action of $\widetilde{T}$ on $\tS$. 
Concretely, in \Cref{sec:weighted_transfer_operators}, we study a family of bounded linear operators 
\[
\cL_{z} \colon \cB_{p,q} \to \cB_{p,q}, \qquad \text{for $z\in \C^d$.}
\]
The spaces $\cB_{p,q}$, indexed by a pair $(p,q) \in \Z_{\geq 0}^2$, are
\emph{anisotropic Banach spaces} of distributions, which are constructed
explicitly in \Cref{sec:Sobolev_spaces_and_distributions}. The operators
$\cL_{z} = \cL_{z,F}$ are defined in terms of the pseudo-Anosov map $T$ on the
compact surface $S$ and of the so-called \emph{Frobenius function} $F$ (see
\Cref{sec:twisted_hilbert_spaces}), a smooth function naturally associated to
the cover $\mathcal{p}\colon \tS \to S$ and the map~$\widetilde{T}$. Let $\cL_{z,F}'$ be the
dual operator of $\cL_{z,F}$. Our first main result is a spectral
description of Maharam distributions.

\begin{bigthm}[Spectral description of Maharam distributions
(\Cref{prop:Maharam_distribution})]\label{bigthm:spectral_Maharam} 

Let $r\in \R^d$ and let $\eta$ be an eigenvalue of $\cL'_{r,F}$, with $|\eta|$ sufficiently large.
We can associate to $\eta$ a Maharam distribution which is invariant under the
horizontal flow on $\tS$. Moreover, if $\eta$ is the spectral radius of
$\cL_{r,F}$, this construction gives the Maharam \emph{measure} $\mu_r$.  
\end{bigthm}

Let us remark that transfer operators twisted by a unitary complex number were
studied by a subset of the authors in~\cite{CR}. They also appeared, in a slightly different form, in the work of Forni on effective weak mixing properties~\cite{Forni:twisted}, as well as in \cite{FR}. We use a complex weight $z$ because,
morally, its real part $r$ will govern the Maharam scaling, while the imaginary
part manages the Fourier modes of the cover (see \Cref{sec:road_map} for a
brief description of the Fourier decomposition and \Cref{sec:twisted_hilbert_spaces} for the precise definitions).

The existence of eigenvalues comes from \emph{quasi-compactness} of the
operators on the spaces~$\cB_{p,q}$.

\begin{bigthm}[$\cL_{z,F}$ is quasi-compact
(\Cref{cor:spectral_decomposition})]\label{bigthm:quasi_compactness}

Let $z\in \C^d$ with real part $r=\Re z \in \R^d$, and let $p,q \geq 1$. The
operator $\cL_{z} \colon \cB_{p,q} \to \cB_{p,q}$ is \emph{quasi-compact}, with
spectral radius $\rho(z) \leq \rho(r)$ and essential spectral radius
$\rho_{\ess}(z) \leq \lambda^{-\min\{p,q\}}\, \rho(r)$.
\end{bigthm}

A central objective in the spectral theory of transfer operators is not only to establish quasi-compactness, but also to identify the resulting discrete spectrum.
A precise description of the Pollicott-Ruelle resonances for the standard (untwisted) transfer operator was obtained, for example, by Butterley, Kiamari, and Liverani in \cite{BuKiLi} for linear hyperbolic toral automorphisms, and by Faure, Gouëzel and Lanneau~\cite{FaGoLa} in our setting of linear pseudo-Anosov maps on translation surfaces. Recently, Galli \cite{Gal} extended these results to nonlinear Anosov diffeomorphisms. The core of the aforedmentioned works is to establish a correspondence between the Pollicott-Ruelle resonances and the eigenvalues of the induced action on cohomology.

In our setting, the presence of a twist in the transfer operator brings our work closer to that of Forni~\cite{Forni:twisted_cohomological,Forni:twisted}, and, more notably, to that of Dang and Rivière~\cite{DaRi}. Following this perspective, we describe the discrete spectrum of the operators $\cL_{z,F}$  in terms of the action of $T$ on a twisted cohomology group. 
We are indebted to Gabriel Rivière for explaining to us their work and for suggesting that the methods could be adapted to our setting.

We denote the twisted cohomology group by $H^1_z(S_0,\C)$ for $z\in\C^d$, and by $\cL_z^\#$ the induced action of $T$ on $H^1_z(S_0,\C)$.

\begin{bigthm}[Connection of Maharam distributions to geometry (\Cref{thm:cohomological_Maharam})]\label{bigthm:geometric_Maharam}
Let $p,q \geq 1$ and $z \in \R^d \oplus i(-\pi,\pi)^d \setminus \{0\}$. Denote by $\sigma_{p,q}(\cL_{z,F})$ the spectrum of $\cL_{z,F}$ on $\cB_{p,q}$, and let
\[\sigma^+_{p,q}(\cL_{z,F}) = \{\alpha \in \sigma_{p,q}: |\alpha| > \lambda^{-p} \rho(z)\}.\]
Let $\Theta(z) \subset \Theta_0(z)$ denote the spectrum of the action of $\cL_z^\#$ on $H^1_z(S,\C)$ and on $H^1_z(S_0,\C)$ respectively.
Then 
\[
\{\lambda^{-n}\mu: n \geq 1, \, \mu \in \Theta(z), \, |\mu|> \lambda^{-\min\{p,q\}+1} \rho(z)\} \subseteq \sigma_{p,q}(\cL_{z,F}),
\]
and 
\[ 
\sigma^+_{p,q}(\cL_{z,F}) \subseteq \{ \lambda^{-n}\mu: 1 \leq n \leq p, \mu \in \Theta_0(z) \}.
\]
\end{bigthm}

We call the discrete spectrum of $\cL_z$ the \emph{Maharam-Ruelle-Pollicott spectrum of $\widetilde{T}$} associated to the parameter $z \in \C^d$.

\Cref{bigthm:quasi_compactness,bigthm:geometric_Maharam} above are a
generalization of~\cite{FaGoLa}, which
corresponds to the case $z=0$. The twisted cohomology was introduced
in~\cite{Forni:twisted} to study effective weak mixing of the straight-line flow
on compact translation surfaces.

In the proof of \Cref{bigthm:quasi_compactness}, we show that the maximal
eigenvalue $\rho$ for the transfer operator $\cL_{z}$ is simple, and it is the
only one on the circle of radius $|\rho|$. We call $\ell_+$ the corresponding
eigenfunction. Then, the dual transfer operator $\cL'_z$ also has a simple maximal
eigenvalue. Following the approach of~\cite{GoLi}, we show that this corresponds
to a \emph{probability measure} $\nu$ in \Cref{lem:nu_is_measure_2}. Moreover,
we also define the measure $\nu_T = \nu(\cdot\ell_+)$, which is invariant under
$T$ and exponentially mixing for $T$, see \Cref{lem:nu_is_measure_2} and
\Cref{prop:nu_T_is_exponentially_mixing}. We show that these two measures share
a common transverse measure $\vartheta$, when decomposed along flow lines, see \Cref{cor:transverse_measure}. 

The measure $\nu$ is quasi-invariant for $\{\phi_t\}_{t\in\R}$ on $S$ and its appropriately scaled lift to $\tS$ gives the invariant Maharam measure $\mu_r$, see \Cref{cor:maharam_measures}.
We study the Hausdorff dimension of the measure $\nu$ and prove the following result.

\begin{bigthm}[Hausdorff dimension of $\nu$ (\Cref{thm:Hausdorff_dim_nu})]
\label{bigthm:Hausdorff}
	Let $r\in\R^d$ and let $\rho(r)$ be the maximal eigenvalue of
	$\cL_r$ and $\nu$ be the probability measure corresponding to the eigenvalue
	$\rho(r)$ for the dual transfer operator $\cL'_r$. Then, the Hausdorff
	dimension of $\nu$ is given by
	\[
		\dim_H \nu = 1 + \frac{\log{\lambda} + \log{\rho(r)}-r\cdot\nu_T(F)}{\log{\lambda}},
	\]
	where $F$ is the Frobenius function.
\end{bigthm}

The above result was obtained in a somewhat more general setting by Berk, Fr\k{a}czek, Kotlewski and Trujillo in~\cite{BFKT}, using techniques from skew-products of IETs. 

We believe that $\nu_T$ is the Gibbs measure for $T$ with
potential $r\cdot F - \log{\lambda}$,  see
\Cref{rmk:nu_T_is_Gibbs}. If this is true, then the previous result says that
\[
	\dim_H \nu = 1 +\frac{h_{\nu_T}(T)}{\log \lambda}.
\]

\subsection{Main results II : asymptotics of ergodic integrals}

As a main dynamical application of the results presented so far, we study ergodic integrals of sufficiently smooth functions with compact support.
A remarkable novelty is that we obtain an asymptotic description for ergodic integrals not only for Lebesgue-generic points (as in \cite{BFRT}), but also for points which are generic 
with respect to other Maharam measures.

We define a continuous positive function $\Delta_r(x,s)$ for all points $x \in S$ with an infinite forward orbit and for all positive $s$ which is a \lq\lq self-similar\rq\rq\ parametrization of orbits of the translation flow, see \Cref{lem:self_similarity_of_Delta}. With the aid of this \lq\lq normalized length\rq\rq, we can state our result.


\begin{bigthm}[Asymptotics of ergodic integrals (\Cref{thm:ergodic_integrals_asymptotics})]\label{bigthm:ergodic_integrals}
	There exists a constant $c >0$, which only depends on the cover, such that for any $r\in \R^d$ with $\|r\|_{\infty} < c$ the following holds.
	
	There exists a symmetric negative definite $d\times d$ matrix $\Sigma = \Sigma_r$ such that  for any $f\in \mathscr{C}^1_c(\tS)$ with $\mu_r(f) \neq 0$, for $\mu_r$-almost every $x\in \tS$, and for all $t\geq 1$, letting $n = \lceil \log_\lambda t \rceil$, we have
    \[
    \int_0^{t} f\circ \phi_s(x)\diff s = \frac{(\lambda \rho(r))^n e^{-r\cdot \xi(\widetilde{T}^nx)}}{(2\pi n)^{d/2} \sqrt{\det(-\Sigma)}} \mu_{r}(f) \Bigg( \Delta_{r}\Big(T^n \mathcal{p}(x),\frac{t}{\lambda^n}\Big) e^{Z_n(x)} + o(1) \Bigg),
    \]
    where $o(1)$ denotes a term which tends to zero as $t\to \infty$, $Z_n(x)$ satisfies
    \[
    Z_n \to -\frac{1}{2} (Z+ \xi(x)) \cdot \Sigma^{-1} (Z+\xi(x)) \text{ \ in distribution (w.r.t.~$\mu_r$)}, \qquad \text{ where }Z \sim \matholdcal{N}(0,-\Sigma),
    \]
    and $\xi: \tS \to \R^d$ is a continuous analogue of the $\Z^d$-coordinate on the cover (see \Cref{sec:twisted_hilbert_spaces} for definition).
\end{bigthm}

The function $\xi$ appearing in \Cref{bigthm:ergodic_integrals} above is related to the \emph{Frobenius function} $F$ associated to the lift $\widetilde{T}$ of $T$ to $\tS$, which describes the displacement in the cover of points $x$ under $\widetilde{T}$. 
A key ingredient to prove the above result is a Central Limit Theorem for the Birkhoff sums of $F$ with respect to $\nu_T$, see \Cref{thm:CLT}.

\subsection{A road map to the proof of \Cref{bigthm:ergodic_integrals}}\label{sec:road_map}
In this section we give a brief and very informal sketch of the proof of \Cref{bigthm:ergodic_integrals}, which uses many of the other results proved in the paper.

Let $x\in \tS$ and let $f$ be a compactly supported smooth observable on $\tS$. We are aiming to estimate $\int_0^t f\circ \phi_s(x) \diff s$.
\medskip

\underline{Step 1: Fourier decomposition.} Fixing an $r\in \R^d$, we decompose $f$ as 
\begin{equation}
\label{eq:roadmap_fourier}
f(x) = \int_{\theta \in \left(-\frac12,\frac12\right)^d} f_{r+2\pi i \theta}(x) \diff \theta,
\end{equation}
where each $f_z$ is rescaled under the action of $\Deck$, depending on $z$. More precisely, given $\xi: \tS \to \R^d$ a continuous analogue of the $\Z^d$ coordinate on $\tS$, $f_z$ has the form 
\[f_z(x) = e^{-z\cdot \xi (x)} g_z(\mathcal{p}(x)),\]
for a function $g_z$ on $S$.
\medskip

\underline{Step 2: renormalization.} We let $n$ be such that $\lambda^{n-1} < t \leq \lambda^n$ and renormalize the interval $\phi_{[0,t]}(x)$ to approximately unit length by applying $\widetilde T^n$. 
We have 
\begin{equation}
\label{eq:roadmap_renormalization}
\int_0^t f\circ \phi_s(x) \diff s = \lambda^n \int_0^{\,t/\lambda^n} f\circ \widetilde{T}^{-n} \circ \phi_s (\widetilde{T}^n(x)) \diff s.
\end{equation}
\smallskip

\underline{Step 3: the weighted transfer operator appears.} Computing 
\[f_z \circ \widetilde{T}^{-1} (x) = e^{-z\cdot \xi(x)} e^{z \cdot(\xi(x)-\xi(\widetilde{T}^{-1}(x)) )} g_z\circ T^{-1} (\mathcal{p}(x)),\]
we see that for $F(x) = \xi(\widetilde T(x)) - \xi(x)$, which is a $\Deck$-invariant function and hence projects to a function on $S$, defining $\cL_{z,F} (g) = (e^{z\cdot F} g)\circ T^{-1}$, we have
\begin{equation}
\label{eq:roadmap_wto}
f_z\circ \widetilde{T}^{-n} (x) = e^{-z\cdot \xi(x)} \cL_{z,F}^n(g_z) (\mathcal{p}(x)).
\end{equation}
Hence after the Fourier decomposition, the renormalization dynamics is governed by $\cL_z = \cL_{z,F}$.

Combining \eqref{eq:roadmap_fourier}-\eqref{eq:roadmap_wto}, we can write
\begin{equation}
\label{eq:roadmap_combined_1}
\int_0^t f\circ \phi_s(x) \diff s = \lambda^n \int_{\theta \in \left(-\frac12,\frac12\right)^d} \int_0^{t/\lambda^n} \big(e^{-z\cdot\xi} \cL_z^n(g_z) \big) \circ \phi_s (\widetilde{T}^n (x)) \diff s \diff \theta.
\end{equation}
\medskip

\underline{Step 4: projection onto the maximal eigenvalue.} By the quasi-compactness of $\cL_z$ (\Cref{bigthm:quasi_compactness}), the simplicity of the maximal eigenvalue of $\cL_r$ (\Cref{cor:spectral_decomposition})
and the perturbation theory of bounded operators (\Cref{prop:pert}), for $\theta$ close to $0$ one can write 
\[\cL_z^n = \rho(z)^n \Pi_z + Q_z^n,\] where $\rho(z)$ is the spectral radius of $\cL_z$, $\Pi_z$ a projection onto the eigenspace corresponding to the simple maximal eigenvalue and $Q_z$ a slower-growing error term such that $\Pi_z Q_z = Q_z\Pi_z = 0$. On the other hand for $\theta$ not close to 0, the spectral radius of $\cL_z$ is less than some $\sigma < \rho(r)$, so, assuming that $n$ is large, we can ignore these $\theta.$
\medskip

\underline{Step 5: the Maharam measure.} We show in \Cref{bigthm:spectral_Maharam} that $\mu_r$ is defined by $\mu_r(f) = \nu_r (g_r)$ for compactly supported $f$, where $g_r$ comes from the Fourier decomposition of $f$ and $\nu_r$ is a measure on $S$ arising as the eigenvector of the dual transfer operator $\cL'_{r,F}$, which is dual to $\ell_+$.

We approximate $\Pi_z(g_z)$ by $\Pi_r(g_r)$, and observing that $\Pi_r$ has the form $\Pi_r(g) = \nu_r(g) \ell_+$, being the projection onto the one-dimensional subspace generated by $\ell_+$, we have the estimate
\begin{equation}
\label{eq:roadmap_Maharam}
\cL_{z}^n (g_z) \approx \rho(z)^n \mu_r(f) \ell_+.
\end{equation}

Plugging \eqref{eq:roadmap_Maharam} into \eqref{eq:roadmap_combined_1}, we write 
\[\int_0^t f\circ \phi_s(x) \diff s \approx \lambda^n \rho(r)^n \mu_r(f) \int_{\theta \in B(0,\epsilon)} \frac{\rho(z)^n}{\rho(r)^n} \int_0^{t/\lambda^n} e^{-z\cdot\xi} \circ\phi_s(\widetilde{T}^n x) \diff \ell_+(s) \diff \theta  \]
\smallskip

\underline{Step 6: the $\Z^d$ coordinate.} Approximating $e^{-z\cdot \xi}$ by $e^{-r\cdot \xi}$ for small $\theta$, and using that this is a smooth function which does not vary too much along intervals of length 1, we can take out the factor $e^{-r\cdot \xi(\widetilde{T}^n (x))}$ to write 
\[ \int_0^{t/\lambda^n} e^{-z\cdot\xi} \circ\phi_s(\widetilde{T}^n x) \diff \ell_+(s)  \approx e^{-r\cdot \xi(\widetilde{T}^n (x))} \Delta_r\big(T^n(\mathcal{p}(x)),\frac{t}{\lambda^n} \big), \]
where $\Delta_r$ is a bounded positive function.
\medskip

\underline{Step 7: stationary phase estimate.}
It remains to estimate $\int_{\theta \in B(0,\epsilon)} \frac{\rho(z)^n}{\rho(r)^n} \diff \theta$.

In \Cref{prop:pert} we compute the first two terms of the Taylor expansion at $\theta=0$ of $\log \rho(z)$, and write $\frac{\rho(z)^n}{\rho(r)^n}$ as an exponential of this Taylor expansion.

For $n$ large we apply a stationary phase estimate to the integral of this exponential, which gives the term 
\[\frac{e^{Z_n(x)}}{(2\pi n)^{d/2} \sqrt{\det(-\Sigma)}}.\]

Using the CLT for $T$ proved in \Cref{thm:CLT}, we can say that $Z_n$ converges in distribution. 




\subsection{Future directions}

We conclude with some open questions and future directions of this work. 

We investigate only the case of self-similar covers of translation surfaces, but it is natural to ask what can be said in the more general case of abelian covers. Some strong results in this direction, in the specific case of the infinite staircase surface, have recently been obtained by Dolgopyat and Sarig in \cite{DS_new}. It would be interesting to see whether the methods of the present work can be extended to cover non self-similar cases and what can be said there. 

In the compact case, the interest of Pollicott-Ruelle resonances lies in their connection to the correlations of observables under the hyperbolic dynamics. It would be interesting to see if also in our case the Maharam-Pollicott-Ruelle resonances play a role in the mixing properties of the pseudo-Anosov $\tT$. As the standard notion of mixing is not relevant to infinite measure systems, one would need to consider a different notion of mixing, such as a variant of global-local mixing introduced in \cite{Lenci}.

In \Cref{thm:CLT}, we obtain a Central Limit Theorem for $\nu_T$. It is natural
to ask whether statements like a Local Limit Theorem or even a Large Deviations
Principle hold in this case too.


\subsection{Organization of the paper}
Motivated by the Fourier decomposition, we introduce anisotropic Banach spaces
for functions on compact translation surfaces in
\Cref{sec:Sobolev_spaces_and_distributions} and study the weighted transfer
operator on them in \Cref{sec:weighted_transfer_operators}, where we prove
\Cref{bigthm:quasi_compactness}. We describe in detail the Fourier decomposition
in \Cref{sec:twisted_hilbert_spaces} and specialize the results of the previous
sections to our context. We prove
\Cref{bigthm:spectral_Maharam,bigthm:geometric_Maharam} in
\Cref{sec:Maharam_stuff}, introducing the twisted cohomology. In
\Cref{sec:ergodic_integrals} we study ergodic integrals and prove
\Cref{bigthm:ergodic_integrals}. We study the Hausdorff dimension of Maharam
measures in \Cref{sec:Hausdorff} and prove \Cref{bigthm:Hausdorff}. We give some detailed examples to which our
results apply, together with explicit computations of the twisted cohomology in
Appendix \ref{sec:examples}. Finally, in Appendix \ref{sec:leafwise_measures} we prove a
technical result, \Cref{prop:leafwise_measures}, which we need to adapt
from~\cite{GoLi2}.

\tikzset{every picture/.style={line width=0.75pt}} 

\section{Anisotropic Banach spaces}\label{sec:Sobolev_spaces_and_distributions}
In this section we define the Banach spaces we will use throughout the paper. 
\subsection{Open rectangles and a partition of unity}

Let $\{\phi_t\}_{t\in \R}$ be our translation flow on $S_0$ which, without loss
of generality, we can assume to be a horizontal flow. We denote by
$\{\phi^{\perp}_t\}_{t\in \R}$ the vertical flow. A point $x \in S_0$ is called
\emph{regular} if the horizontal segment $\{\phi_t(x) \ :\ t\in [-1,1]\}$ does
not contain any singularity. Let us call $S_{\reg}$ the open set of regular
points. 

We
begin by clarifying what we mean by ``smooth functions'' on $S$ and on $\tS$.
Let $\tS_{\reg} = \mathcal{p}^{-1}(S_{\reg})$.
\begin{definition}
	Let $k \in \N \cup \{\infty\}$, and let $Z\in \{S, \tS\}$. The space $\mathscr{C}^{k,\bdd}(Z)$ consists of functions $f \colon Z_{\reg} \to \C$ so that, for every differential operator $V$ of order $\leq k$, the function $Vf$ is well-defined and uniformly bounded.
\end{definition}
For a function $f \in \mathscr{C}^{k,\bdd}(S)$, we define the $\mathscr{C}^{k}$-norm, for $k \in \N$, in the usual way as
\[
\|f\|_{\mathscr{C}^k} \coloneqq \sup \{ |X^{n} Y^m f(x)| \ :\ n,m\geq 0,\ n+m\leq k,\ x\in S_{\reg} \};
\]
an analogous definition holds for elements of $\mathscr{C}^{k,\bdd}(\tS)$.

An open rectangle $\Omega \subset S_0$ is the image of an embedding
\[
\begin{split}
I_v \times I_h & \to S_0\\
(s,t) &\mapsto  \phi_s^{\perp} \circ \phi_t(x),
\end{split}
\]
for some $x\in S_0$ and some open interval $I_v$ and $I_h$. 
Clearly, the same rectangle $\Omega$ can be the image of different embeddings; however, the lengths $|I_h|$ and $|I_v|$ of $I_h$ and $I_v$ are uniquely determined by $\Omega$. We call them the horizontal and vertical length of $\Omega$ respectively.
In the following, the rectangles we consider will have lengths $\leq 1$. We will write
\[
\Omega = \{ \phi_t(y)\ :\ t\in I_h,\ y \in \Omega_v\},
\]
where $I_h$ is a subinterval of $I\coloneqq (0,1)$, $\Omega_v = \{\phi_s^{\perp}(x)\ :\ s\in I_v\}$ is a vertical segment, and $x \in S_{\reg}$. Similarly, we write $\Omega_h = \{\phi_t(x)\ :\ t\in I_h\}$.

The following lemma follows from elementary geometric considerations and
will be used later.

\begin{lemma}\label{lem:rectangle} There exists $\delta >0$ such that, for all
	$x \in S_0$, there exists an open rectangle $\Omega = \Omega(x) \subset
	S_0$ containing $x$ with horizontal length $>1$ and with vertical length
	$\geq \delta$.
\end{lemma}

\begin{proof}
We recall that holonomy vectors of saddle connections on a compact translation surface are discrete. Then, we define
\[
	2m_0 = \min\Set*{\Big\|\int_\gamma \alpha\Big\|_{\infty}\given \gamma \text{ saddle connection in } S} >0.
\]
Let $n_0 \geq 0$ be minimal so that $\lambda^{n_0}m_0 > 1$. We claim that it is enough to choose $\delta = m_0\lambda^{-n_0}$. 
Indeed, let $x \in S_0$, and let $y = T^{n_0}x \in S_0$. By definition, there exists an open rectangle $\Omega \subset S_0$ containing $y$ of horizontal and vertical sides $m_0$. Then, the rectangle $   T^{-n_0}\Omega \subset S_0$ contains $x$, its horizontal side is greater than $1$ and its vertical side is $ m_0\lambda^{-n_0}$.
\end{proof}

We will need to ``localize'' elements of $\mathscr{C}^q(S)$ via an appropriately
constructed partition of unity. The properties of this localization procedure
are the content of \Cref{prop:localize} below. Recall that $I$ denotes the unit
interval $(0,1)$. For $q\geq 0$, we let $\mathcal{C}^q = \mathscr{C}_c^q(I)$ and
$\mathcal{C}^{-q} = (\mathscr{C}_c^q(I))^{\ast}$.

By taking a smaller $\delta$ if necessary, we can assume that the $\delta$ we found in \Cref{lem:rectangle} is less than $\frac{1}{6}$.

\begin{proposition}\label{prop:localize}
There exists a finite cover $\{\Omega^{a}\}_{a \in \cA}$ of $S_{\reg}$
(up to a zero measure set) by open rectangles $\Omega^{a} = \{\phi_t(y) \ :\ t
\in I_h^a, \ y\in \Omega^a_v\}$, with horizontal and vertical lengths between
$\delta$ and $1$ so that the following holds. Let $q\geq 0$ and $u \in
\cC^q(S)$; then, there exists a family of functions $\{u_a\}_{a \in \cA}
\subset \cC^q(S)$ so that 
\begin{enumerate}
	\item $u = \sum_a u_a$,
	\item $u_a(x)\neq 0$ only if $x\in \Omega^a$ and $u_a|_{\Omega^a} \in
	\mathscr{C}^q(\overline{\Omega^a})$,
	\item for each $y \in \Omega^a_v$, the function $u_{a,y} \colon t \mapsto
	u_a \circ \phi_t(y)$ can be seen as an element of $\mathcal{C}^q$ and
	$\|u_{a,y}\|_{\mathscr{C}^q} \leq C \|u\|_{\cC^q}$, for a uniform constant
	$C$.
\end{enumerate}
\end{proposition}
\begin{proof}
Let us fix some notation: we use the coordinate $w = a+ib \in \C$, with $a,b \in \R$, and we let $W_r = \{w \in \C : |a|, |b| < r\}$ for any $r>0$.
Let $x \in \Sigma$ be a singular point; by definition, there exist a positive integer $k_x$, an open set $U_x \subset S$ such that $U_x \cap \Sigma = \{x\}$, and a (branched) cover $\pi_x \colon U_x \to W_{1/2}$ so that, in canonical coordinates $z = s+it$ on $U_x$, we have
\[
\pi_x(z) = \frac{z^{k_x+1}}{k_x+1}.
\]
Note that the vector fields $X$ and $Y$ on $U_x$ are given by
\[
\begin{split}
	&X = \pi_x^{\ast}(\partial_a) = \frac{1}{|z|^{2k_x}} \left(\Re (z^{k_x}) \partial_s - \Im (z^{k_x} )\partial_ t\right),\qquad \text{and} \\ &Y=\pi_x^{\ast}(\partial_b)=\frac{1}{|z|^{2k_x}} \left(\Im (z^{k_x}) \partial_s + \Re (z^{k_x}) \partial_ t\right)
\end{split}
\]
respectively; in particular, $\|X_z\|, \|Y_z\| \leq |z|^{-k_x}$. 

Let $V_x = \pi_x^{-1}(W_{1/3}) \subset U_x$, and define $V_\Sigma = \cup_{x\in
\Sigma} V_x$. Fix a finite cover $\{\widetilde{\Omega_i}\}_{i \in \mathcal{I}}$
of $S \setminus V_\Sigma$ by open rectangles $\widetilde{\Omega_i}$ of
horizontal and vertical lengths between $\delta$ and $1/6$; in particular
$\dist(\widetilde{\Omega_i},\Sigma) > 1/6$. By compactness, there exists a
\emph{finite} partition of unity $\{\psi_j\}_{j \in \mathcal{J}}$ of $S
\setminus V_\Sigma$ subordinate to the cover $\{\widetilde{\Omega_i}\}_{i \in
\mathcal{I}}$ (and hence we will replace the index set $\mathcal{I}$ by
$\mathcal{J}$); in other words, there exists a finite family $\{\psi_j\}_{j \in
\mathcal{J}}$ of smooth functions $\psi_j \colon S_0 \to [0,1]$ such that 
\[
\sum_{j\in \mathcal{J}} \psi_j(x) = 1, \qquad \text{for all $x\in S \setminus V_\Sigma$, and }\qquad \dist(\supp \psi_j, \Sigma) \geq 1/6.
\] 
We complete the family $\{\psi_j\}_{j \in \mathcal{J}}$ to a partition of unity
on $S$ as follows: for every $x \in \Sigma$, we define $\psi_x \colon U_x \to
[0,1]$ by $\psi_x = 1- \sum_{j\in \mathcal{J}}\psi_j$. Since
$\psi_x|_{\pi_x^{-1}(W_{1/6})} \equiv 1$, we have that $\psi_x \in
\cC^{\infty}(U_x)$; moreover, $\psi_x|_{U_x \setminus V_x} \equiv 0$.

\begin{figure}[t]
	\centering
	\def\svgwidth{0.9\textwidth}
\begingroup%
  \makeatletter%
  \providecommand\color[2][]{%
    \errmessage{(Inkscape) Color is used for the text in Inkscape, but the package 'color.sty' is not loaded}%
    \renewcommand\color[2][]{}%
  }%
  \providecommand\transparent[1]{%
    \errmessage{(Inkscape) Transparency is used (non-zero) for the text in Inkscape, but the package 'transparent.sty' is not loaded}%
    \renewcommand\transparent[1]{}%
  }%
  \providecommand\rotatebox[2]{#2}%
  \newcommand*\fsize{\dimexpr\f@size pt\relax}%
  \newcommand*\lineheight[1]{\fontsize{\fsize}{#1\fsize}\selectfont}%
  \ifx\svgwidth\undefined%
    \setlength{\unitlength}{407.01420593bp}%
    \ifx\svgscale\undefined%
      \relax%
    \else%
      \setlength{\unitlength}{\unitlength * \real{\svgscale}}%
    \fi%
  \else%
    \setlength{\unitlength}{\svgwidth}%
  \fi%
  \global\let\svgwidth\undefined%
  \global\let\svgscale\undefined%
  \makeatother%
  \begin{picture}(1,0.39739802)%
    \lineheight{1}%
    \setlength\tabcolsep{0pt}%
    \put(0,0){\includegraphics[width=\unitlength,page=1]{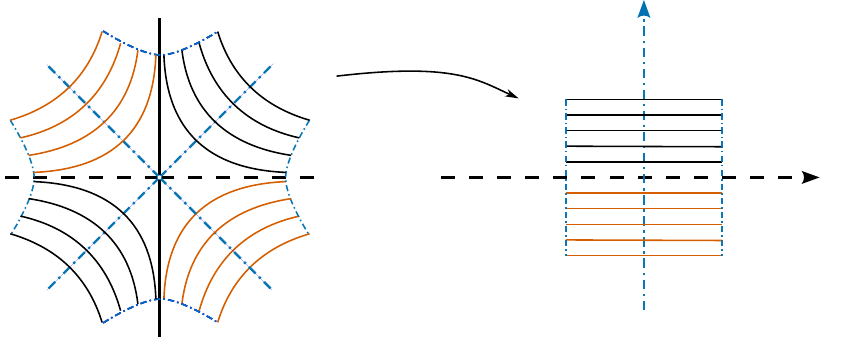}}%
    \put(0.50471354,0.32211065){\makebox(0,0)[t]{\lineheight{1.25}\smash{\begin{tabular}[t]{c}$\pi_x(z) = \frac{z^2}{2}$\end{tabular}}}}%
    \put(0.28993808,0.3408729){\makebox(0,0)[t]{\lineheight{1.25}\smash{\begin{tabular}[t]{c}$\Omega_x^1$\end{tabular}}}}%
    \put(0.08724246,0.3408729){\color[rgb]{0.83529412,0.36862745,0}\makebox(0,0)[t]{\lineheight{1.25}\smash{\begin{tabular}[t]{c}$\Omega_x^2$\end{tabular}}}}%
    \put(0.08632111,0.02485199){\makebox(0,0)[t]{\lineheight{1.25}\smash{\begin{tabular}[t]{c}$\Omega_x^3$\end{tabular}}}}%
    \put(0.28901674,0.02485199){\color[rgb]{0.83529412,0.36862745,0}\makebox(0,0)[t]{\lineheight{1.25}\smash{\begin{tabular}[t]{c}$\Omega_x^4$\end{tabular}}}}%
    \put(0.73974965,0.09216773){\makebox(0,0)[t]{\lineheight{1.25}\smash{\begin{tabular}[t]{c}$-\frac{1}{2}$\end{tabular}}}}%
    \put(0.75053551,0.27587149){\makebox(0,0)[t]{\lineheight{1.25}\smash{\begin{tabular}[t]{c}$\frac{1}{2}$\end{tabular}}}}%
    \put(0.86553309,0.15581428){\makebox(0,0)[t]{\lineheight{1.25}\smash{\begin{tabular}[t]{c}$\frac{1}{2}$\end{tabular}}}}%
    \put(0.64139025,0.15489294){\makebox(0,0)[t]{\lineheight{1.25}\smash{\begin{tabular}[t]{c}$-\frac{1}{2}$\end{tabular}}}}%
    \put(0.88573956,0.25334162){\makebox(0,0)[t]{\lineheight{1.25}\smash{\begin{tabular}[t]{c}$R^+$\end{tabular}}}}%
    \put(0.88573956,0.10674964){\color[rgb]{0.83529412,0.36862745,0}\makebox(0,0)[t]{\lineheight{1.25}\smash{\begin{tabular}[t]{c}$R^-$\end{tabular}}}}%
  \end{picture}%
\endgroup%
	\caption{The rectangles $R^\pm$ and their preimages, in an example where $k_x=1$.}
	\label{fig:partitionof1}
\end{figure}
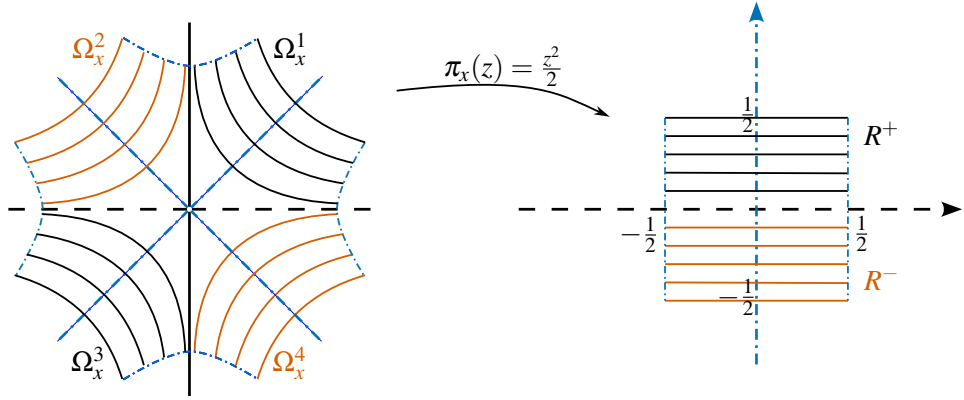

We now slightly modify the family of rectangles $\{\widetilde{\Omega_j}\}_{j \in
\mathcal{J}}$ and obtain a new family of rectangles $\{\Omega_j\}_{j \in
\mathcal{J}}$ which form a finite cover of $S_{\reg}$ up to a zero measure set
(by a slight abuse of notation, we still denote by $\mathcal{J}$ the finite
index set). If $\widetilde{\Omega_j}\cap S_{\reg} \neq \widetilde{\Omega_j}$, we
cut horizontally $\widetilde{\Omega_j}$ and we replace it with two disjoint open
rectangles with the same horizontal length and with vertical length still $\geq
\delta$. We now add the following open rectangles: let $R^{\pm} = W_{1/2} \cap
\{ \text{sign} (b) = \pm\}$, and let $\Omega_x^1, \dotsc, \Omega_x^{2(k_x+1)}$
be the open rectangles given by the connected components of
$\pi_x^{-1}(R^{\pm})$, see \Cref{fig:partitionof1}. We call $\{\Omega_j\}_{j \in
\mathcal{J}}$ the resulting family; by construction, up to a finite union of
horizontal segments, they form an open cover of $S_{\reg}$.

We are ready to conclude. Define the index set $\cA = \mathcal{J} \cup
\{(x,l) \ :\ x \in \Sigma,\ l\in \{1,\dotsc,2(k_x+1)\}\}$. Let $u \in \cC^q(S)$,
and let $u_a = u \psi_j$ for $a=j\in \mathcal{J}$, and $u_{a} =
u\psi_x|_{\Omega_x^l}$ for $a = (x,l)$. Since the functions $\psi_j$ and
$\psi_x$ form a partition of unity, we have that $\sum_a u_a = u$ on $S_{\reg}$.
Moreover, by construction, each $u_{a}$ as above is not identically zero in at
most one rectangle $\Omega^a = \Omega_j$ or $\Omega^a = \Omega_x^l$. Then,
$u_a\in \mathscr{C}^q(\overline{\Omega_a})$ with $\|u_a\|_{\cC^q} \leq
C\|u\|_{\cC^q} $ for some uniform constant $C$ (depending on $\delta$ and on
the finite partition of unity we constructed). Finally, let us verify (3). If
$a=j \in \mathcal{J}$, then for every $y \in \Omega^a_v$, the function $t\mapsto
\psi_j \circ \phi_t(y)$ is compactly supported in $I^a_h \subset I$; therefore,
the function $u_{a,y}(t) = u_a \circ \phi_t(y)$ can be seen as an element of
$\mathcal{C}^q = \mathscr{C}_c^q(I)$. If $ a=(x,l)$, then $\psi_x|_{\Omega_x^l}
\equiv 0$ in $\Omega_x^l \cap (U_x\setminus V_x)$; thus, similarly as above, for
every $y \in \Omega^a_v$, the function $t\mapsto u_a\circ \phi_t(y)$ is
compactly supported in $I_h = I$. Hence (3) follows in this case as well, and
the proof is complete.
\end{proof}

Let us state a few properties of the space $\cC^k(S)$ which will be useful
later.

\begin{proposition}\label{prop:facts_about_Ck}
	For each $k, k'\in \Z_{\geq 0}$, with $k>k'$, the following properties hold.
	\begin{enumerate}[label=(\roman*)]
		\item $\cC^k(S) \hookrightarrow (\cC^\infty(S))^*$.
		\item \label{item:compact_inclusion} The immersion
		$\cC^{k,\bdd}(S)\hookrightarrow\cC^{k',\bdd}(S)$ is compact.
		\item \label{item:continuity_on_hor_int} For any open rectangle $\Omega
		\subset S_0$, there is a continuous embedding of $\cC^k(\Omega)$ into
		$\cC^k(\Omega_h)$.
	\end{enumerate}
\end{proposition}

\begin{proof}
	The only part which needs to be proven is the second one. 
	
	Let $\{f_n\}_{n\in\N}$ be a sequence of $\cC^{k,\bdd}$ functions. By
	compactness of $S$, we can cover it with finitely many open rectangles
	$\{\Omega^a\}_{a\in\cA}$. By definition, a $\cC^{k,\bdd}$ function extends
	continuously to the boundary of each rectangle (which might include
	singularities). Since each rectangle is bounded and convex, the inclusion
	$\cC^{k,\bdd}\bigl(\overline{\Omega^a}\bigr) \hookrightarrow
	\cC^{k'}\bigl(\overline{\Omega^a}\bigr)$ is compact for each $a\in\cA$, see
	e.g.,~\cite[Theorem~1.34]{AdamsFournier}. Passing to a common refinement, we
	obtain a converging subsequence of the $f_n$'s inside $\cC^{k'}(S)$, as we
	wanted.
\end{proof}

\subsection{Anisotropic Banach spaces: definitions}
\label{sec:def_spaces}

We are going to define the elements of our anisotropic Banach spaces as
distributions generated by functionals in the dual space $\mathcal{C}^{-q}$
which are essentially integration along the flow $\phi_t$ away from the singular
set. 

In view of \Cref{prop:facts_about_Ck}, we can define the following family of
distributions.
\begin{definition}\label{def:Gamma}
	Let $p,q \in \Z_{\geq 0}$ and let $v \in \cC^{p,\bdd}(S)$. Then, for every
	$x\in S_{\reg}$, we define the distribution $\Gamma v(x) \in
	\mathcal{C}^{-q}$ by
	\[
	\Gamma v(x) \colon u \mapsto \langle \Gamma v(x), u \rangle=\int_0^1 v\circ\phi_s(x) \cdot u (s) \diff s, \qquad \text{ for }u \in \mathcal{C}^{q}.
	\] 
\end{definition}
By \Cref{lem:rectangle} and \Cref{prop:facts_about_Ck}, there exists a constant
$C$ depending only on the geometry of $S$ such that, for every $v \in
\cC^{p,\bdd}(S)$ and $u \in \cC^q$, we have
\begin{equation}\label{eq:ineq_H2}
|\langle \Gamma v(x), u \rangle | \leq C \, \|v\|_{\cC^0}\, \|u\|_{\cC^0(I)}, \qquad \text{so that} \qquad \|\Gamma v(x)\|_{\mathcal{C}^{-q}}  \leq C \,   \|v\|_{\cC^0}.
\end{equation}
Given $v \in \cC^{p,\bdd}(S)$, we have a map
\[
\begin{split}
	\Gamma v \colon S_{\reg}& \to \mathcal{C}^{-q}\\
	x& \mapsto \Gamma v(x),
\end{split}
\]
which, by the inequality above, is uniformly bounded. As a matter of facts, by
\labelcref{item:continuity_on_hor_int} of \Cref{prop:facts_about_Ck}, it is
continuous. We now study the differentiability properties of this map, in the
following sense. Let $V \in \{X,Y\}$, and let $\phi^V_s$ be the unit speed
motion for time $s$ in direction $V$. Given $f \colon S_{\reg} \to
\mathcal{C}^{-q}$, the derivative $Vf$, if it exists, is the map $Vf\colon
S_{\reg} \to \mathcal{C}^{-q-1}$ satisfying
\[
\lim_{\varepsilon \to 0} \left\| \frac{f\circ \phi^V_{\varepsilon}(x)-f(x)}{\varepsilon} - Vf(x) \right\|_{\mathcal{C}^{-q-1}} = 0.
\]
Note that the range of $Vf$ is not $\mathcal{C}^{-q}$ but the larger space $\mathcal{C}^{-q-1}$.
We say that $f$ is of class $\mathscr{C}^r$ if all derivatives $X^mY^nf$, for $m+n\leq r$, exist and are continuous maps from $S_{\reg}$ to $\mathcal{C}^{-q-m-n}$. 
We equip the space $\mathscr{C}^r(S_{\reg},\mathcal{C}^{-q})$ of maps $f \colon S_{\reg} \to \mathcal{C}^{-q}$ of class $\mathscr{C}^r$ with the norm
\[
 \sup_{x \in S_{\reg}} \sum_{|\alpha| \leq r} \|\partial^\alpha f(x)\|_{\mathcal{C}^{-q-|\alpha|}}<\infty,
\]
where, for a multi-index $\alpha = (m,n) \in \Z_{\geq 0}^2$, we denoted
\[
\partial^\alpha = X^mY^n = Y^nX^m \qquad \text{and} \qquad |\alpha| = m+n.
\]

\begin{lemma}\label{lem:derivative_Gamma}
	Let $p,q \in \Z_{\geq 0}$. Then, for any $v \in \cC^{p,\bdd}(S)$, the
	function $\Gamma v \colon S_{\reg} \to \mathcal{C}^{-q}$ is of class
	$\cC^p$. More precisely, for any $n \leq p$ and $x \in S_{\reg}$, we have 
	\[
		Y^{n}(\Gamma v)(x) = \Gamma (Y^{n}v)(x), \qquad \text{and} \qquad X^{n}(\Gamma v)(x) = \Gamma (X^{n}v)(x) = (-1)^n \, \Gamma v (x) \circ \partial^n,
	\]
	and there exists a constant $C_{n}$ depending on $n$  such that
	\[
		\sup_{x \in S_{\reg}} \|Y^{n}(\Gamma v)(x)\|_{\mathcal{C}^{-q-n}} \leq C_{n} \|v\|_{\cC^n}, \qquad \text{and} \qquad \sup_{x \in S_{\reg}} \|X^{n}(\Gamma v)(x)\|_{\mathcal{C}^{-q-n}} \leq C_{n} \|v\|_{\cC^0}.
	\]
\end{lemma}

\begin{proof}
	We assume $p\ge 1$, otherwise there is nothing to prove. Let $V \in
	\{X,Y\}$, and, as above, let $\phi^V_s$ be the unit speed motion for time
	$s$ in direction $V$. Let $v \in \cC^{p,\bdd}(S)$; we shall prove that, for
	all $x \in S_{\reg}$, we have 
	\[
		V(\Gamma v)(x) =  \Gamma (Vv)(x) \qquad \text{in }\mathcal{C}^{-q-1};
	\]
	the general claim follows by induction.
	By definition, we have to show that, for all $x \in S_{\reg}$ we have 
	\begin{equation}\label{eq:deriv_x_gamma}
		\lim_{\varepsilon \to 0} \left\| \frac{\Gamma v(\phi^V_{\varepsilon}x)-\Gamma v(x)}{\varepsilon} - \langle \Gamma (Vv)(x), \cdot \rangle  \right\|_{\mathcal{C}^{-q-1}} = 0.
	\end{equation}
	Let $x\in S_{\reg}$, $u \in \mathcal{C}^{q+1}$, and let $\Omega$ be an open rectangle as in \Cref{lem:rectangle}. For $\varepsilon$ sufficiently small (so that a $\varepsilon$-neighbourhood of the horizontal segment of length one at $x$ is all contained in $\Omega$), we have
	\[
	\begin{split}
		&\left\lvert \frac{\langle \Gamma v(\phi^V_{\varepsilon}x), u \rangle - \langle \Gamma v(x), u \rangle }{\varepsilon} - \langle\Gamma (Vv)(x), u \rangle   \right\rvert \\
		&\qquad = \left\lvert \frac{1}{\varepsilon} \int_{0}^{1} (v\circ \phi_{\varepsilon}^V -v) \circ \phi_s(x) \cdot u(s)\diff s- \int_{0}^{1} Vv\circ \phi_s(x) \cdot u(s)\diff s \right\rvert \\
		&\qquad \leq  \|u\|_{\cC^0(I)}   \int_{0}^{1} |Vv(y_s) - Vv(\phi_s(x))| \diff s,
	\end{split}
	\]
	where $y_{s}$ is a point in the segment joining $\phi_s(x)$ and $ \phi_{\varepsilon}^V(\phi_s(x))$.  
	Continuity of $Vv$ proves \cref{eq:deriv_x_gamma}. 
	In the case $V=X$, we can further write
	\[
	\langle\Gamma (Xv)(x), u \rangle = \int_{0}^{1} \partial_s v\circ \phi_s(x) \cdot u(s)\diff s = - \int_{0}^{1} v\circ \phi_s(x) \cdot u'(s)\diff s = \langle\Gamma v(x), u' \rangle,
	\]
	since $u$ is compactly supported on $I$.
	The conclusion follows from the bounds~\eqref{eq:ineq_H2}.
\end{proof}

We are now ready to give the definition of our anisotropic Banach spaces.

\begin{definition}
Let $p,q \in \Z_{\geq 0}$. We define the \emph{anisotropic Banach space} $\cB_{p,q}$ as 
\[
\cB_{p,q} = \cl_{\mathscr{C}^p}(\langle \Gamma v \ : v \in \cC^{p,\bdd}(S)\rangle),
\]
where $\cl_{\mathscr{C}^p}$ denotes the closure in $\mathscr{C}^p(S_{\reg}, \mathcal{C}^{-q})$. 
\end{definition}

Naturally, the norm of an element $\ell$ in $\cB_{p,q}$ is given by
\[
\|\ell\|_{p,q}= \sup_{x\in S_{\reg}} \sup_{|\alpha| \le p} \|\partial^\alpha \ell(x)\|_{\mathcal{C}^{-q-|\alpha|}} = 
\sup_{x\in S_{\reg}} \sup_{|\alpha| \le p} \sup_{\substack{u \in \mathcal{C}^{q+|\alpha|} \\ \|u\|_{\mathcal{C}^{q+|\alpha|}}\le 1}} \left\lvert \langle \partial^\alpha \ell(x), u\rangle \right\rvert.
\]
For elements of the form $\ell= \Gamma v$, with $v \in \cC^{p,\bdd}(S)$, the
norm above, by \Cref{lem:derivative_Gamma}, could be rewritten as
\[
\|\Gamma v\|_{p,q} = \sup_{x\in S_{\reg}} \, \sup_{|\alpha| \le p} \, \sup_{\substack{u \in \mathcal{C}^{q+|\alpha|} \\ \|u\|_{\mathcal{C}^{q+|\alpha|}}\le 1}} \left\lvert \langle  \Gamma (\partial^{\alpha}v)(x), u\rangle \right\rvert.
\]

\subsection{Anisotropic Banach spaces: basic properties}
The following Lemma describes the spaces $\cB_{p,q}$ as Banach spaces of
distributions. We recall that we use the notation $\cC^{-q} = (\cC^q)^*$.

\begin{lemma}\label{lem:inclusion}
For any $p,q \in \Z_{\ge 0}$ we have the following continuous injections:
\[
\cC^{p,\bdd}(S) \hookrightarrow \cB_{p,q} \hookrightarrow \cC^{-q}(S). 
\]
Moreover, the first inclusion is dense.
\end{lemma}
\begin{proof}
By the definition of the spaces $\cB_{p,q}$ and by \Cref{lem:derivative_Gamma}
we know that $\cC^{p,\bdd}(S)\hookrightarrow \cB_{p,q}$ is continuous and dense.
Let us show it is injective. Let $v\in \cC^{p,\bdd}(S)$, with $v \neq 0$; then,
there exists a smooth function $u$ compactly supported in an open rectangle
$\Omega= \{\phi_t (y)\ :\ y \in \Omega_v,\ t\in (b,c) \}$, where $\Omega_v =
\{\phi_s^{\perp}(x)\ :\ s \in I_v\}$ is a vertical segment, with horizontal and
vertical lengths $< 1$ so that
\[
0 \neq \int_{\Omega} v \cdot u \, \diff\!\vol = \int_{\Omega_v} \int_b^{c}(v \cdot u) \circ \phi_t(y) \, \diff t \diff y.
\]
As usual, we assume $I_h=(b,c) \subset I$; then, we can extend the horizontal side to have length exactly $1$, so that we can write
\[
0 \neq \int_{\Omega} v \cdot u \, \diff\!\vol = \int_{\Omega_v} \langle \Gamma v(y),u_y \rangle \diff y,
\]
where $u_y \in \mathcal{C}^q$ is defined as $u_y(t) = u\circ \phi_t(y)$.
Thus, there exists $y \in \Omega_v$ so that $\Gamma v(y) \in \mathcal{C}^{-q}\setminus\{0\}$, which proves the injectivity claim.

To prove the existence of the second inclusion, let us consider the functional
$i\colon\cC^{\infty,\bdd}(S)\to \cC^{-q}(S)$ defined by 
\[
\langle i(v), u \rangle=\int_S v \cdot u \, \diff\!\vol, \qquad \text{for } 
 u\in \cC^q.
\]
It is sufficient to prove that the map $i$ can be extended by continuity to $\cB_{p,q}$ and that it is injective.
Let us fix $u\in \cC^q$; by \Cref{prop:localize}, we have
\[
\begin{split}
	\left|\int_S v \cdot u \, \diff\!\vol \right|& = \left|\sum_{a\in \cA} \int_{\Omega^a} v \cdot u_a \, \diff\!\vol\right|	
	= \left|\sum_{a\in \cA} \int_{\Omega^a_v} \int_{I^a_h} v \circ \phi_t(y) \cdot \widetilde{u_{a,y}}(t) \, \diff t \diff y\right|,
\end{split}
\]
where, $\widetilde{u_{a,y}}(t) = u_a\circ \phi_t(y)$ is a $\mathscr{C}^q$
function with compact support in the interval $I_h^a \subset I$. Thus, we can
write $\widetilde{u_{a,y}} \in \mathcal{C}^q$, with
$\|\widetilde{u_{a,y}}\|_{\mathcal{C}^q} \leq C \|u\|_{\cC^q}$. We deduce that
\[
	\left|\int_S v \cdot u \, \diff\!\vol \right| = \left|\sum_{a\in \cA} \int_{\Omega^a_v} \langle \Gamma v(y), \widetilde{u_{a,y}}\rangle \, \diff y\right| \leq C \sup_{y \in S_{\reg}} \|\Gamma v(y)\|_{\mathcal{C}^{-q}}\|u\|_{\cC^q}.
\]
Hence, we obtain the continuity of $i\colon\cB_{p,q} \to \cC^{-q}(S)$.

Let us prove that the map $i$ is injective. Let $\ell \in \cB_{p,q}$, with
$\ell\neq 0$; to prove that $i(\ell) \neq 0$ in $\cC^{-q}(S)$, it suffices to
show that there exists $\widetilde{u} \in \cC^q(S)$ with $\| \widetilde{u}
\|_{\cC^q} = 1$, so that $\langle i(\ell), \widetilde{u} \rangle \neq 0$. By
definition, there exist $x\in S_{\reg}$ and a multi-index $\alpha$, with
$|\alpha|\leq p$, so that $\partial^{\alpha}\ell(x) \neq 0$ in
$\mathcal{C}^{-q-|\alpha|}$; in particular, if $\ell(x)=0$, then $\ell$ is not locally constant.
Therefore, there exist $b > 0$, $y\in S_{\reg}$ close to $x$, and $u \in
\mathcal{C}^{q}$, with $\|u\|_{\mathcal{C}^{q}}=1$, so that $\langle
\ell(y),u\rangle = b$. Let us denote $y_s = \phi^{\perp}_s(y)$. By continuity,
there exists $\varepsilon >0$ so that $\langle \ell(y_s), u\rangle > b/2$ for
all $|s|<\varepsilon$. Fix a smooth function $\psi \colon
(-\varepsilon,\varepsilon) \to \R_{\geq 0}$ with compact support and so that
$\int_{-\varepsilon}^{\varepsilon} \psi(t)\diff t = 1$ and let $\Omega$ be the
rectangle $\Omega = \{\phi_t(y_s) \ :\ s\in (-\varepsilon,\varepsilon),\ t\in
(0,1) \}$. We define the function $\hat u$ on $\Omega$ by 
\[
\hat u \circ \phi_t(y_s)\coloneqq u(t) \cdot \psi(s), 
\]
and let $\widetilde{u} = \frac{\hat u}{\|\hat u \|_{\cC^q}}$.
By density of $\cC^{p,\bdd}(S)$ in $\cB_{p,q}$, there exists $v \in \cC^{p,\bdd}(S)$ so that 
\[
\|\Gamma v - \ell\|_{p,q} < \frac{b}{6};
\]
in particular, $\|i(v) - i(\ell)\|_{\cC^{-q}} < \frac{b}{6}$. 
Since we have
\[
\left|\langle \Gamma v(y_s), u \rangle - \langle \ell(y_s), u \rangle\right| \leq \|\Gamma v(y_s) - \ell(y_s)\|_{\mathcal{C}^{-q}} \leq \|\Gamma v - \ell\|_{p,q} < \frac{b}{6},
\]
it follows that $\langle \Gamma v(y_s), u \rangle > \frac{b}{3}$ for all $s \in (-\varepsilon,\varepsilon)$. Therefore,
\[
\begin{split}
|\langle i(\ell), \widetilde{u}\rangle|&\geq |\langle i(v), \widetilde{u}\rangle| - \frac{b}{6} = \left|\int_{\Omega} v\cdot \widetilde{u} \, \diff\!\vol \right| - \frac{b}{6} =  \left|\int_{-\varepsilon}^{\varepsilon} \langle\Gamma v(y_s), u\rangle \cdot \psi(s) \, \diff s \right| - \frac{b}{6}
\geq \frac{b}{3} - \frac{b}{6} >0,
\end{split}
\]
which completes the proof.
\end{proof}
Finally, in order to apply Hennion's theorem in the next section, we will need the following compactness result.
\begin{lemma}\label{lem:compact}
For all integers $p \geq 1$ and $q \geq 0$, there is a compact inclusion $\cB_{p,q}\hookrightarrow \cB_{p-1,q+1} $.
\end{lemma}
\begin{proof}
	The existence of a continuous inclusion follows from the fact that the
	inclusion $\cC^{p,\bdd} \hookrightarrow \cB_{p-1,q+1}$ satisfies the bound $\|
	\partial^{\alpha}\Gamma v(x)\|_{\mathcal{C}^{-q-1-|\alpha|}} \leq \|
	\partial^{\alpha}\Gamma v(x)\|_{\mathcal{C}^{-q-|\alpha|}}$, for any $x \in
	S_{\reg}$ and $|\alpha|\leq p-1$, so that it extends to a continuous map
	$\cB_{p,q} \hookrightarrow \cB_{p-1,q+1}$.
	
In order to prove that the inclusion is compact, we are going to exploit the
abstract compactness criterion proved in~\cite[Proposition 2.8]{FaGoLa}: let
$B_1 \subseteq B_2$ be two Banach spaces. Assume that, for every $\epsilon >0$,
there exist finitely many continuous linear forms $\Xi_1,\dotsc, \Xi_N$ on $B_1$
such that, for any $v \in B_1$,\footnote{In~\cite{FaGoLa} there is the sum over
$k$ instead of the sup, but the argument of the proof is unchanged in the latter
case.}
\begin{equation}\label{eq:criterion}
\|v\|_{B_2}\le \epsilon \|v\|_{B_1}+ \sup_{1\le k\le N} |\Xi_k(v)|.
\end{equation}
Then the inclusion of $B_1$ in $B_2$ is compact. 

Let us fix $p,q$ and $\epsilon >0$. We can take a finite set of regular points
$x_1,\dotsc,x_{N}$ such that, for every $x\in S_0$, there exists $j\in
\{1,\dotsc,N\}$ for which  $\operatorname{dist}_S(x,x_j)<\epsilon$. Moreover, by
\labelcref{item:compact_inclusion} of \Cref{prop:facts_about_Ck}, for any
$a=1,\dotsc,p-1$, there exist $u^a_{1}, \dotsc, u^a_{N_a} \in
\mathcal{C}^{q+a+1}$ of norm at most 1 so that, for any $u\in
\mathcal{C}^{q+a+1}$ of norm at most 1, we have $\min_{k=1,\dotsc,N_a}
\|u^a_{k}-u\|_{\mathcal{C}^{q+a}}\leq \epsilon$. We take as linear forms on
$\cB_{p,q}$  
\[
\Xi_{\alpha,k,l} \colon \ell \mapsto \langle \partial^{\alpha}\ell(x_l),u^{|\alpha|}_{k}\rangle,
\]
for $|\alpha|\leq p-1$, $l=1,\dotsc,N$, and $k = 1,\dotsc, N_{|\alpha|}$.

Let now $\ell\in \cB_{p,q}$. For any $x\in S_{\reg}$, any $a=|\alpha|\leq p-1$, and any $u\in \mathcal{C}^{q+a+1}$, we can find $k \in \{1,\dotsc, N_{a}\}$ and $l \in \{1,\dotsc, N\}$ such that 
\[
\begin{split}
|\langle \partial^{\alpha}\ell(x), u\rangle|&\leq |\langle \partial^{\alpha}\ell(x), u^a_{k}\rangle| + \|\partial^{\alpha}\ell(x)\|_{\mathcal{C}^{-q-a}} \cdot \|u-u^a_{k}\|_{\mathcal{C}^{q+a}} \\
&\leq |\langle \partial^{\alpha}\ell(x_l), u^a_{k}\rangle| + \|\partial^{\alpha}\ell(x)-\partial^{\alpha}\ell(x_l)\|_{\mathcal{C}^{-q-a-1}}\cdot \|u^a_{k}\|_{\mathcal{C}^{q+a+1}} + \epsilon\|\ell\|_{p-1,q} \\
& \leq |\Xi_{\alpha,k,l}(\ell)| + 2\epsilon \sup_{y\in S_{\reg} } \|\partial^{\beta}\ell(y)\|_{\mathcal{C}^{-q-a-1}}+ \epsilon\|\ell\|_{p-1,q},
\end{split}
\]
where $|\beta| = a+1$. Since the middle summand above is bounded by $2\epsilon \|\ell\|_{p,q}$, we conclude 
\[
\|\ell\|_{p-1,q+1} \leq \sup_{\alpha, k,l}|\Xi_{\alpha,k,l}(\ell)| + 3\epsilon \|\ell\|_{p,q},
\]
which proves the result.
\end{proof}

\begin{lemma}\label{lem:multiplying_by_smooth}
	Let $p,q \in\Z_{\ge 0}$. Given $f\in\mathscr{C}^{\infty, \bdd}(S)$ and
	$\ell\in\cB_{p,q}$, then $f\ell\in\cB_{p,q}$. Moreover,
	$\|f\ell\|_{p,q} \le C \|f\|_{\cC^{p+q}} \|\ell\|_{p,q}$, for some
	constant $C>0$.
\end{lemma}

\begin{proof}
	By density, it is enough to prove the result for $\ell = \Gamma v$, with
	$v\in \cC^{p,\bdd}(S)$. Let $\alpha = (m, n)\in\Z^2_{\ge 0}$ be a multi-index with
	$|\alpha| \le p$. For any $x\in S_{\reg}$, and function $u\in
	\mathcal{C}^{q+|\alpha|}$ of norm less than $1$, we have
	\[
		|\langle\Gamma (\partial^\alpha (f v))(x), u\rangle| \le 
			\sum_{j=0}^{m} \sum_{k=0}^{n} \binom{m}{j} \binom{n}{k} 
				\left| \int_{0}^{1} X^{j}Y^{k} v \circ \phi_s (x) \cdot 
					X^{m-j}Y^{n-k} f \circ \phi_s (x) \cdot u(s) \, \diff s \right|.
	\]
	Since each of the integrals is bounded by $\|f\|_{\mathscr{C}^{p+q}}
	\|\Gamma v\|_{p,q}$, we are done.
\end{proof}

In the rest of the paper, we will suppress $\Gamma$ from the notation; as such, we identify a smooth function $f \in \mathscr{C}^{\infty, \bdd}(S)$ with the corresponding element $\Gamma f \in \cB_{p,q}$.

\section{Weighted transfer operators}
\label{sec:weighted_transfer_operators}

Let $T \colon S\to S$ be a linear pseudo-Anosov diffeomorphism which, in local coordinates\footnote{The coordinates are given by the charts of the translation atlas.} $x=(s,t)$ is given by $Tx = T(s,t) =(\lambda^{-1}s, \lambda t)$, for some $\lambda>1$. 

Let us fix a {  real-valued vector-valued} function
{  $F \in \cC^{\infty, \bdd}(S,\R^d)$}
with zero average.
For {  $z\in \C^d$}, we define the operator
\[
\cL_{z,F}f(x)
=
\big({  e^{z \cdot F}}\, f\big)\circ T^{-1}(x), 
\]
acting on smooth functions $f\in \cC^{\infty, \bdd}(S)$.
In this section, we prove that $\cL_{z,F}$ can be extended to a continuous linear operator on the spaces $\cB_{p,q}$ introduced in the previous section and we establish a Lasota-Yorke type inequality.

It can be easily checked that the $n$-th composition of the operator is given by
\[
\cL_{z,F}^n f(x)
=
\big({  e^{z \cdot S_n F}}\, f\big)\circ T^{-n}(x),
\]
where
\[
{  S_nF(x)\coloneqq\sum_{k=0}^{n-1}F\circ T^k(x)\in\R^d}
\]
is the usual Birkhoff sum of $F$.\\
\begin{definition}\label{def:pou}
	Let $n\in \N$, $x\in S_{\reg}$, and let $I(x) \subset S_0$ be the horizontal segment starting at $x$ of length 1. The horizontal segment $T^{-n}I(x)$ does not contain any singularity and is of length $\lambda^{n}$. 
	We say a partition of unity $\{\tilde{u}_a\}_{a\in \mathcal{A}}$ of $T^{-n}I(x)$ is \emph{good} if the following properties hold:
	\begin{enumerate}
	\item for each $a$, $\tilde{u}_a$ is supported on some horizontal unit interval $I_a \subset T^{-n}I(x)$,\\
	\item each point $y\in T^{-n}I(x)$ is contained in at most two intervals $I_a$,\\
	\item there exists a non-negative function $\theta \in \mathscr{C}^{\infty}_c(0,1)$, $\theta \neq 0$, independent of $x$ and $n$, such that, identifying each $\tilde{u}_a$ with a function on the unit interval, we have $\tilde{u}_a \geq \theta$.
	\end{enumerate} 
	Let $\mathcal{P}_n(x)$ be the set of good partitions of unity of $T^{-n}I(x)$. 
\end{definition}

For each {  $z\in \C^d$} we define the quantity
\[
R_n = R_n(z)
\coloneqq
\lambda^{-1}
\sup_{x\in S_{\reg}}
\sup_{\mathcal{A} \in \mathcal{P}_n(x)}
\left(
\sum_{a\in\mathcal{A}}
\big\|
{  e^{z \cdot S_n F}}\big|_{I_a}
\big\|_{\mathcal{C}^{0} }
\right)^{\frac1n}.
\]



\begin{notation}
To simplify notation when talking about partitions of unity, we will denote by $\mathcal{A}$ both the partition of unity itself as an element of 
$\mathcal{P}_n(x)$ and its index set. Further for an interval $I_a$ of the partition, we also denote by $I_a$ the set 
$\{s\in \R: \phi_s(x) \in I_a\}$. Hence an integral of the form $\int_{I_a} \psi \circ \phi_s(x) \diff s$ is an integral over the interval $I_a$.
\end{notation}

\begin{lemma}\label{lem:submultiplicativity}
The function $R_n^n$ is submultiplicative, i.e., for any $n, m \geq 1$,
\[R_{n+m}^{n+m} \leq R_n^n R_m^m.\]
Hence the sequence $(R_n)_n$ is monotone decreasing and there exists a limit
\begin{equation}\label{eq:def_Rz}
R(z) = \lim_{n \to \infty}R_n(z) = \inf_{n \in \N}R_n(z) >0.
\end{equation}
\end{lemma}

\begin{proof}
Consider a partition $\mathcal{A} \in \mathcal{P}_{n+m}(x)$ for some $x$. For
some partition $\cB \in \mathcal{P}_n(x)$ with intervals $\{J_b : b\in \cB\}$,
we can subdivide $\mathcal{A}$ as $\mathcal{A} = \cup_{b\in \cB} \mathcal{A}_b$,
so that for any $a\in \mathcal{A}_b$, $T^m(I_a) \subset J_b$. Then 
\[
\|{  e^{z \cdot S_{m+n} F}} |_{I_a}\|_{\mathcal{C}^0}
\leq
\|{  e^{z \cdot S_m F}} |_{I_a}\|_{\mathcal{C}^0}
\,
\|{  e^{z \cdot S_n F}} |_{J_b}\|_{\mathcal{C}^0}.
\]
Hence
\[
\sum_{a\in \mathcal{A}} \|{  e^{z \cdot S_{m+n} F}} |_{I_a}\|_{\mathcal{C}^0}
\leq
\sum_{b\in \cB}
\|{  e^{z \cdot S_n F}} |_{J_b}\|_{\mathcal{C}^0}
\sum_{a\in \mathcal{A}_b}
\|{  e^{z \cdot S_m F}} |_{I_a}\|_{\mathcal{C}^0}.
\]
Observing that if $x_b$ is the leftmost point of $J_b$, then $\mathcal{A}_b \in \mathcal{P}_m(x_b)$, we deduce that
\[
\lambda^{m+n}R_{n+m}^{n+m} \leq \lambda^n R_n^n\, \lambda^m R_m^m.
\]
Finally, to see that $R(z) >0$ for any {  $z\in\C^d$}, observe that for any $n$, any $x$, any partition $\mathcal{A}\in \mathcal{P}_n(x)$, 
\[
\sum_{a\in \mathcal{A}} \|{  e^{z \cdot S_n F}}|_{I_a}\|_{\mathcal{C}^0}
\geq
\lambda^n e^{-n\,{  |\Re (z \cdot F)|_{\mathcal{C}^0}}}.
\]
Hence
\[
R(z) \geq e^{-{  |\Re (z \cdot F)|_{\mathcal{C}^0}}}.
\]
\end{proof}

\begin{remark}\label{rk:Rn_and_R}
We note that, for any {  $z \in \C^d$} and $n\in \N$, we have
\[
R_n(z) = R_n({  \Re z}),
\]
which implies that $R(z) = R({  \Re z})$.
\end{remark}

\subsection{Technical lemmas}

The following technical results are essential in the proof of the Lasota-Yorke inequality. We will use the following notation: if $p,q \geq 0$, when writing $A \ll_{p,q} B$, we mean that $A \leq C_{p,q} B$ for some constant $C_{p,q}>0$ depending on $p$ and $q$ only. 

\begin{lemma}\label{lem:testfun_0}
	Let $(G_n)_{n\in\N}\subset \cC^{\infty, \bdd}(S)$, and assume that there exists $\theta \geq 1$ such that for all $a,b \geq 0$ with $a+b \geq 1$, we have
	\[
	\|X^aY^bG_n\|_{\mathscr{C}^0}\ll_{a,b}  \theta^{bn}.
	\]
	Then, we have for all $j,k \geq 0$,
	\[
	\|X^j Y^k e^{G_n}\|_{\mathscr{C}^0}
	\ll_{j,k}
	 \theta^{kn} \|e^{G_n}\|_{\mathscr{C}^0}.
	\] 
\end{lemma}

\begin{proof}
	One can see by induction that 
	\begin{multline*}
|X^j Y^k e^{G_n}|
\ll_{j,k}
|e^{G_n}|
\max\{
|X^{a_1}Y^{b_1}{G_n}|\cdots |X^{a_{r}}Y^{b_r}{G_n}|
:\ a_i, b_i \geq 0, a_i+b_i \geq 1,\\
a_1 + \cdots + a_r = j,\quad
b_1 + \cdots + b_r = k
\}.
	\end{multline*}
	The claim follows from the assumption on the norm of $G_n$.
\end{proof}

We want to apply \Cref{lem:testfun_0} with $G_n = z \cdot S_n F$, hence we need
the following estimate.

\begin{lemma}\label{lem:testfun_1}
	For every $j,k,n \geq 0$ with $j+k \geq 1$, we have
	\[
	\|X^jY^k  (z \cdot S_n F)\|_{\mathscr{C}^0}
	\ll_{j,k}
	{(  1+ |z|)}\,\lambda^{kn}\|F\|_{\mathscr{C}^{j+k}}.
	\]
\end{lemma}

\begin{proof}
	We have 
	\[
	|X^jY^k  (z \cdot S_n F)|
	=
	\left|
	\sum_{\ell=0}^{n-1} X^jY^k  (z \cdot F\circ T^{\ell})
	\right|
	\leq
	{  (1+|z|)}
	\sum_{\ell=0}^{n-1}\lambda^{k\ell-j\ell}\|F\|_{\mathscr{C}^{j+k}}.
	\]
	For $k=0$, since $j >0$, the latter term is bounded by $C_j
	\|F\|_{\mathscr{C}^{j+k}}$, for some constant $C_j>0$ which depends on $j$
	only. In the case $k \geq 1$, 
	\[
	\sum_{\ell=0}^{n-1}\lambda^{k\ell-j\ell} \|F\|_{\mathscr{C}^{j+k}}
	\ll_k
	\|F\|_{\mathscr{C}^{j+k}} \lambda^{kn}. \qedhere
	\]
\end{proof}

\begin{lemma}\label{lem:testfun}
For every $k\geq 0$, $n \geq 0$ and $y\in S_{\reg}$, define for $s\in [0,\lambda^n]$ the function
\[
\Phi_{z,n,y}^{(k)}(s)
\coloneq
Y^k \big( {  e^{z \cdot S_n F}} \circ T^{-n} \big)
\circ \phi_{ \lambda^{-n} s} (y).
\]
Fix a partition $\mathcal{A} \in \mathcal{P}_n(y)$. Then the $\mathcal{C}^j$ norm of $\Phi_{z,n,y}^{(k)}$ restricted to one interval $I_a$ of $\mathcal{A}$ is bounded by 
\[
\|\Phi_{z,n,y}^{(k)}|_{I_a}\|_{\mathcal{C}^j}
\ll_{j,k}
{  (1+|z|)^{j+k}}
\|F\|^2_{\mathscr{C}^{j+k+1}}
\|{  e^{z \cdot S_n F}} |_{I_a}\|_{\mathcal{C}^{0}}.
\]
\end{lemma}

\begin{proof}
First observe that 
\[
\Phi_{z,n,y}^{(k)}(s) =  Y^k \big( e^{z \cdot S_n F }  \circ T^{-n}  \circ \phi_{ \lambda^{-n} s} (y)\big) 
 = Y^k(e^{z \cdot S_n F } \circ \phi_s (T^{-n} y)).
\]
Hence
\[
\frac{\diff^j}{\diff s^j} \Phi_{z,n,y}^{(k)}(s) = \lambda^{-nj} X^jY^k \big( e^{z \cdot S_n F }  \circ \phi_s (T^{-n} y)\big).
\]
Now fix a point $s_a \in I_a$ and write 
\[
e^{z \cdot S_n F }  \circ \phi_s (T^{-n} y)
=
e^{(z \cdot S_n F)  (\phi_s(T^{-n} y)) - (z \cdot S_n F)  (\phi_{s_a}(T^{-n} y))}\,
e^{(z \cdot S_n F)  (\phi_{s_a}(T^{-n} y))}.
\]
We estimate separately the derivatives of the two terms above.
First, for $\ell \leq j, m \leq k$,
\begin{equation}\label{eqn:test_term2}
\begin{split}
\left| X^\ell Y^m e^{(z \cdot S_n F)  (\phi_{s_a}(T^{-n} y))} \right|
&= \lambda^{n\ell -nm} \left|X^\ell Y^m(e^{z \cdot S_n F }) \circ \phi_{s_a}(T^{-n} y) \right| \\
& \ll_{\ell,m} \lambda^{n\ell} {  (1+|z|)^{\ell+m}}\, \|F\|_{\mathscr{C}^{\ell+m}}\, {  \|e^{z \cdot S_n F } |_{I_a}\|_{\mathscr{C}^{0}}}\ ,
\end{split}
\end{equation}
where we applied {  \Cref{lem:testfun_0}} and \Cref{lem:testfun_1}.

Now for the first term, we need to estimate
\begin{align*}
&\left|X^\ell Y^m \left( S_n F(\phi_s (T^{-n} y)) - S_n F (\phi_{s_a} (T^{-n} y)) \right) \right| \leq\\
& \qquad \qquad \leq  \sum_{i=-n}^{-1} \left| X^\ell Y^m\left(F\circ T^i (\phi_{\lambda^{-n}s} (y)) - F\circ T^i (\phi_{\lambda^{-n}s_a} (y))\right)  \right| \\
& \qquad \qquad = \sum_{i=-n}^{-1} \lambda^{-\ell i+mi} \left| (X^\ell Y^m F) (\phi_{\lambda^{-n-i}s}(T^i y)) - (X^\ell Y^m F) (\phi_{\lambda^{-n-i}s_a}(T^i y)) \right| \\
& \qquad \qquad \leq \sum_{i=-n}^{-1} \lambda^{-\ell i+mi-n-i} \|F\|_{\mathscr{C}^{\ell+m+1}} \ll_{\ell,m} \lambda^{\ell n} \|F\|_{\mathscr{C}^{\ell+m+1}}.
\end{align*}
Hence by {  \Cref{lem:testfun_0}},
\begin{equation}\label{eqn:test_term1}
\left| X^{j-\ell} Y^{k-m} \Big(e^{(z \cdot S_n F)  (\phi_s(T^{-n} y)) - (z \cdot S_n F)  (\phi_{s_a}(T^{-n} y))}\Big)\Big|_{I_a}\right|
\ll_{j,k} {  (1+|z|)^{j+k-\ell-m}} \lambda^{jn-\ell n} \|F\|_{\mathscr{C}^{j+k+1}},
\end{equation}
since 
\[
\Big\| e^{(z \cdot S_n F)  (\phi_s(T^{-n} y)) - (z \cdot S_n F)  (\phi_{s_a}(T^{-n} y))}\Big|_{I_a}\Big\|_{\mathscr{C}^0}
\leq
{  \exp\!\left(\frac{|z|}{1-\lambda^{-1}} \|F\| _{\mathscr{C}^{1}}\right)}.
\]

We conclude by combining \eqref{eqn:test_term1} and \eqref{eqn:test_term2}.
\end{proof}



\subsection{The Lasota-Yorke Inequality}

 Let us call $\cL\coloneq\cL_{{  z},F}$ for simplicity. Recall that $R = R({  z}) = \lim_n R_n$. We will prove the following result.
\begin{proposition}\label{prop:LY}
For each $p,q \in \Z_{\ge 0}$ there exists $C = C(p,q,F)>0$ such that, for each $n\in \N$ and {  $z\in \C^d$}, the operator $\cL^n$ initially defined on $\cC^{\infty, \bdd}(S)$ extends continuously to $\cB_{p,q}$ and
\begin{equation}\label{LY-1}
\|\cL^n f\|_{p,q}\le C {  (1+|z|)}^{p+q}  R_n^n \|f\|_{p,q}.
\end{equation}
Moreover, for each $p, q\ge 1$, for each $\sigma \in (\lambda^{-\min\{p,q\}},1)$, there exists a $C= C(p,q,F,\sigma)>0$ such that, for any $n\in \N$ and {  $z\in \C^d$}, we have
\begin{equation}\label{LY-2}
\begin{split}
		\|\cL^n f\|_{p,q} \le &C {  (1+|z|)}^{p+q} \sigma^n R_n^{n}\|f\|_{p,q} + C {  (1+|z|)}^{2(p+q+1)} R_n^n \|f\|_{p-1,q+1}.
\end{split}
\end{equation}
\end{proposition}

\begin{proof}
	Let $p,q\geq 0$ be fixed. For any $0\leq p_0 \leq p$, for any $x\in S_{\reg}$, and for any $u \in \mathcal{C}^{q+p_0}$ with $\|u\|_{\mathcal{C}^{q+p_0}}\leq 1$, we need to estimate
	\[
	\left|\int_0^1 X^jY^k(\cL^n f)\circ \phi_s(x) u(s) \diff s\right|,
	\]
	where $0\leq j,k\leq p_0$ are so that $j+k = p_0$.
	Integration by parts yields
	\begin{equation*}
		\begin{split}
			\int_0^1 X^jY^k(\cL^n f)\circ \phi_s(x) u(s) \diff s &=\int_0^1 \frac{\diff^j}{\diff s^j}Y^k(\cL^n f)\circ \phi_s(x) u(s)\diff s\\
			&=(-1)^j\int_0^1 Y^k(\cL^n f)\circ \phi_s(x) u^{(j)}(s) \diff s;
		\end{split}
	\end{equation*}
	thus, it is sufficient to estimate 
	\[
	\left|\int_0^1 Y^k(\cL^n f)\circ \phi_s(x) u(s) \diff s\right|,
	\]
	where $0\leq k\leq p_0$ and $u \in \mathcal{C}^{q+k}$ with $\|u\|_{\mathcal{C}^{q+k}}\leq 1$.
	
	Recalling \Cref{def:pou}, we have
	\begin{equation*}
		\begin{split}
			&\left|\int_0^1 Y^k(\cL^n f)\circ \phi_s(x) u(s) \diff s\right|
			=\left|\int_0^1 Y^k\big(({  e^{z \cdot S_n F}}\, f)\circ T^{-n}\big)\circ \phi_s(x) u(s) \diff s\right|\\
			&\qquad \ll_k \max_{l + m = k} \left|\int_0^1 Y^{l}(f\circ T^{-n})\circ \phi_s(x)\, Y^{m}({  e^{z \cdot S_n F}}\circ T^{-n})\circ \phi_s(x)\, u(s)\diff s\right|\\
			&\qquad =\max_{l + m = k} \lambda^{-ln}\left|\int_0^1 Y^{l}f\circ \phi_{\lambda^{n} s} (T^{-n}x)\, \big(Y^{m}({  e^{z \cdot S_n F}}\circ T^{-n})\circ \phi_s(x)\, u(s)\big)\diff s\right|\\
			&\qquad = \max_{l + m = k} \lambda^{-ln-n}\left|\int_0^{\lambda^n} Y^{l}f\circ \phi_{s} (T^{-n}x)\, \big(Y^{m}({  e^{z \cdot S_n F}}\circ T^{-n})\circ \phi_{\lambda^{-n}s}(x)\, u(\lambda^{-n}s)\big) \diff s \right|\\
			&\qquad \le \max_{l + m = k} \lambda^{-ln-n} \sum_{a\in \mathcal{A}}\left|\int_{I_a}Y^{l}f\circ \phi_{s} (T^{-n}x)\, \big(\Phi_{z,n,x}^{(m)}(s) \, u(\lambda^{-n}s)\tilde{u}_a(s)\big) \diff s\right|,\\
		\end{split}
	\end{equation*}
	where we have chosen a partition $\mathcal{A} \in \mathcal{P}_n(x)$ and used the notation of \Cref{lem:testfun}.
	
	Using \Cref{lem:testfun} and the fact that the norms of the functions $\tilde{u}_a$ are uniformly bounded, we obtain that
	\begin{equation}\label{eq:first_bound}
		\begin{split}
			\left|\int_0^1 Y^k(\cL^n f)\circ \phi_s(x) u(s) \diff s\right| & 
			\ll_{k} \max_{l + m = k} \lambda^{-ln-n}  \sum_{a\in \mathcal{A}}\left|\int_{I_a} \|Y^{l}f\|_{0,q}\, \|\Phi_{z,n,x}^{(m)}|_{I_a}\|_{\mathcal{C}^{q}}\,\|u(\lambda^{-n}\cdot)\|_{\mathcal{C}^{q}}\right|\\
			& \ll_{k} \max_{l + m = k} \lambda^{-ln-n} \sum_{a\in \mathcal{A}} {  (1+|z|)}^{m+q}\|F\|^2_{\mathscr{C}^{m+q+1}}\,  \|f\|_{l,q} \, \|{  e^{z \cdot S_n F}}|_{I_a}\|_{\mathscr{C}^{0}}\\
			& \ll_{p,q}  {  (1+|z|)}^{p+q}\|F\|^2_{\mathscr{C}^{p+q+1}}  R_n^n \|f\|_{p,q},
		\end{split}
	\end{equation}
	from which~\eqref{LY-1} follows.
	
	We now claim that
	\begin{equation}\label{eq:claim2}
	\|\cL^n f\|_{p,q}\ll_{p,q} {  (1+|z|)}^{p+q+1}\|F\|^2_{\mathscr{C}^{p+q+2}} R_n^n (\lambda^{-\min\{p,q\}n}\|f\|_{p,q} + \|f\|_{p-1,q+1}).
	\end{equation}
	We start from \eqref{eq:first_bound} and we distinguish two cases. If $l = k = p$, then we have 
	\[
		\left|\int_0^1 Y^p(\cL^n f)\circ \phi_s(x) u(s) \diff s\right| 
		 \ll_{p,q} \lambda^{-pn}  {  (1+|z|)}^{q}\|F\|^2_{\mathscr{C}^{q+1}} R_n^n \|f\|_{p,q},
	\]
	hence the claim \eqref{eq:claim2} is verified.
	Otherwise, if $l\leq p-1$, we proceed as follows. 
	For each $\epsilon>0$ small enough, let us consider $u_\epsilon$ obtained by convolving $u \in \mathcal{C}^{q+k}$ with a mollifier $j_\epsilon$ with support in $(0,\epsilon)$ and $\int j_\epsilon=1$ so that,\footnote{The claimed estimates are straightforward once one notes that, if $j_\epsilon (x)=\epsilon^{-1}j(\epsilon^{-1} x)$ where $j\in \cC^{\infty}, \operatorname{supp}j \subset (-1,1), \int j =1$, then
		\[
		\int j_\epsilon(x-y) u(y)\diff y=\int j_\epsilon (y) u (x-y)\diff y.
		\]
	}
	\begin{equation}\label{eq:molly}
		\|u_\epsilon-u\|_{\mathcal{C}^{q-1}}\le \sC\epsilon, \qquad \|u_\epsilon\|_{\mathcal{C}^{q}}\le \sC, \qquad \|u_\epsilon\|_{\mathcal{C}^{q+1}}\le \sC \epsilon^{-1}.
	\end{equation}
	From \eqref{eq:molly}, it follows that
	\[
	\begin{split}
		\|(u-u_\epsilon)(\lambda^{-n}\cdot)\|_{\mathcal{C}^{q}}&\ll_q \lambda^{-qn}\|X^q u- X^q u_\epsilon\|_{\mathcal{C}^0}+ \|(u-u_\epsilon)(\lambda^{-n}\cdot)\|_{\mathcal{C}^{q-1}}\\
		&\ll_q \max\{\lambda^{-qn}, \epsilon\}.
	\end{split}
	\]
	Similarly, 
	\[
	\begin{split}
		\|u_\epsilon(\lambda^{-n}\cdot)\|_{\mathcal{C}^{q+1}}&\ll_q \lambda^{-(q+1)n}\|X^{q+1} u_\epsilon\|_{\mathcal{C}^0}+ \|u_\epsilon(\lambda^{-n}\cdot)\|_{\mathcal{C}^{q}}\\
		&\ll_q \max\{\lambda^{-(q+1)n}\epsilon^{-1}, 1\}.
	\end{split}
	\]
	
	Therefore, from \eqref{eq:first_bound}, choosing $\epsilon = \lambda^{-qn}$, we obtain
	\[
	\begin{split}
&\left|\int_0^1 Y^k(\cL^n f)\circ \phi_s(x) u(s) \diff s\right| \\
& \qquad\ll_{k} \max_{\ell + m = k} \lambda^{-\ell n-n} \sum_{a\in \mathcal{A}} \left|\int_{I_a}Y^{\ell}f\circ \phi_{s} (T^{-n}x)\, \big(\Phi_{z,n,x}^{(m)}(s) \, (u-u_\epsilon)(\lambda^{-n}s)\tilde{u}_a(s)\big) \diff s\right| \\
& \qquad \qquad +\max_{\ell + m = k} \lambda^{-\ell n-n} \sum_{a\in \mathcal{A}} \left|\int_{I_a}Y^{\ell}f\circ \phi_{s} (T^{-n}x)\, \big(\Phi_{z,n,x}^{(m)}(s) \, u_\epsilon(\lambda^{-n}s)\tilde{u}_a(s)\big) \diff s\right|\\
& \qquad \ll_k \max_{l+m = k} \lambda^{-\ell n-n} \sum_{a\in \mathcal{A}} \Big(\|Y^{\ell}f\|_{0,q} \|\Phi_{z,n,x}^{(m)}|_{I_a}\|_{\mathcal{C}^{q}}\|(u-u_\epsilon)(\lambda^{-n}\cdot)\|_{\mathcal{C}^{q}} \\
& \qquad \qquad  + \|Y^{\ell}f\|_{0,q+1} \|\Phi_{z,n,x}^{(m)}|_{I_a}\|_{\mathcal{C}^{q+1}}\|u_\epsilon (\lambda^{-n}\cdot)\|_{\mathcal{C}^{q+1}}\Big)\\
&\qquad \ll_{k,q} \max_{\ell + m= k} \lambda^{-\ell n} {  (1+|z|)}^{p+q+1}\|F\|^2_{\mathscr{C}^{p+q+2}} R_n^n
(\|f\|_{\ell,q} \lambda^{-qn}+ \|f\|_{\ell,q+1} )\\
&\qquad \ll_{p,q}{  (1+|z|)}^{p+q+1}\|F\|^2_{\mathscr{C}^{p+q+2}} R_n^n(\|f\|_{p,q} \lambda^{-qn} + \|f\|_{p-1,q+1} ).
	\end{split}
	\]
	We have then proved \eqref{eq:claim2}.
	
	We can now conclude the proof.
	Let us take any $\sigma \in (\lambda^{-\min\{p,q\}},1)$ and let $n_0\coloneqq n_0({z})$ the smallest integer such that
	\[
	\left(\|F\|^2_{\mathscr{C}^{p+q+2}}{  (1+|z|)}^{p+q+1}  \right)^{\frac{1}{n_0}}\lambda^{-\min\{p,q\}}\le \sigma.
	\]
	Note that $n_0$ is a constant depending on ${  z},p,q$ and $\sigma$, but not on $n$. 
	
	Let $n\in \N$ be arbitrary, and write $n=n_0s+m$, with $m<n_0$. Rewriting \eqref{eq:claim2} in terms of $\sigma$,  we have 
	\[
	\| \cL^{n_0}f\|_{p,q} \ll_{p,q} \sigma^{n_0} R_{n_0}^{n_0} \|f\|_{p,q}
	+ {  (1+|z|)}^{p+q+1} \|F\|^2_{\mathscr{C}^{p+q+2}} R_{n_0}^{n_0} \|f\|_{p-1,q+1}.
	\]
	
	Applying \eqref{LY-1} and using \cref{lem:submultiplicativity} we deduce
	\[
	\begin{split}
		\|\cL^n f\|_{p,q} &\ll_{p,q} \sigma^{n_0} R_{n_0}^{n_0}\|\cL^{n-n_0} f\|_{p,q}
		+ {  (1+|z|)}^{p+q+1}\|F\|^2_{\mathscr{C}^{p+q+2}} R_{n_0}^{n_0}  \|\cL^{n-n_0}f\|_{p-1,q+1} \\
		&\ll_{p,q} \sigma^{n_0} R_{n_0}^{n_0}\|\cL^{n-n_0} f\|_{p,q}
		+ {  (1+|z|)}^{2(p+q+1)}\|F\|^4_{\mathscr{C}^{p+q+2}} R_n^{n} \|f\|_{p-1,q+1}.
	\end{split}
	\]
	
	Iterating,
	\[
	\begin{split}
	\|\cL^n f\|_{p,q} &\ll_{p,q}  \big(\sigma^{n_0} R_{n_0}^{n_0}\big)^s \|\cL^m f\|_{p,q}
	+  {  (1+|z|)}^{2(p+q+1)}\|F\|^4_{\mathscr{C}^{p+q+2}} R_n^{n} \sum_{j=0}^{s-1} \sigma^{jn_0} \|f\|_{p-1,q+1} \\
	& \ll_{p,q} {  (1+|z|)}^{p+q} \|F\|^2_{\mathscr{C}^{p+q+1}} \sigma^n R_n^{n}\|f\|_{p,q}
	+  {  (1+|z|)}^{2(p+q+1)}\|F\|^4_{\mathscr{C}^{p+q+2}} R_n^n \|f\|_{p-1,q+1},
	\end{split}
	\]
	which completes the proof.
	
 \end{proof}



\subsection{The spectral radius and the discrete spectrum}

\Cref{prop:LY} implies that $\cL_{z,F}$ is a linear and continuous operator acting on the Banach space $\cB_{p,q}$, for any fixed $p,q \geq 0$. We are going to show that, in fact, for $p,q \geq 1$, it is a quasi-compact operator.
Let $\rho(z) = \rho(\cL_{z,F})$ and $\rho_{\ess}(z) = \rho_{\ess}(\cL_{z,F})$ denote the spectral radius and the essential spectral radius of $\cL_{z,F}$ on $\cB_{p,q}$  respectively. 
To get a lower bound on $\rho(z)$ we introduce the following definition.

Recall that elements $f$ of our Banach spaces $\cB_{p,q}$ are functions $S_{\reg} \to \mathcal{C}^{-q} = (\mathcal{C}^{q})^{\ast}$; we write $\langle f(x), u \rangle$ to denote the functional $f(x) \in  \mathcal{C}^{-q}$ applied to $u \in \mathcal{C}^{q}$.
\begin{definition}
We say that a functional $f \in \cB_{p,q}$ is \emph{positive} if for any $x \in S_\mathrm{reg}$ and any $u\in \mathcal{C}^q$ satisfying $u \geq 0$, 
$\langle f(x), u \rangle \geq 0$.

We say that a positive $f\in \cB_{p,q}$ is \emph{bounded away from zero} if
for every non-negative $u\in \mathcal{C}^{q}$ with $u\neq 0$, then $\inf \{ \langle f(x), u \rangle \ : \ x\in S_{\reg}\} >0$.
\end{definition}

Note that if $f \in \mathcal{C}^{\infty,\bdd}(S)$ is strictly positive, then it
is bounded away from zero since $  \langle f(x), u \rangle \geq \left( \min_{x\in S} f(x)\right) \int_0^1u(s) \diff s >0 $.

Recalling the definition in \eqref{eq:def_Rz}, we have the following estimates.

\begin{lemma}\label{lem:estimates_spectral_radius}
We have
\[
\rho(z) \leq R(z) \text{\ \  for any $p,q \geq 0$, and\ \  } \qquad \rho_{\ess}(z) \leq \lambda^{-\min\{p,q\}}R(z) \text{\ \ for any $p,q \geq 1$}.
\]
Further, for $p,q \geq 0$, suppose that $f \in \cB_{p,q}$ is positive and bounded away from zero.
Then, for $r\in \R^d$, there exists a constant $C_f >0$ such that, for all $n \geq 1$,
\begin{equation} \label{eq:lowerbound}
\|\cL_{r,F}^n f \|_{p,q} \geq C_f R_n^n(r).
\end{equation}
In particular, if  $r\in \R^d$, then $\rho(r) = R(r)$. 
\end{lemma}

\begin{proof}
From \Cref{prop:LY}, for any $\sigma \in (\lambda^{-\min\{p,q\}},1)$ and for any
$n$, we have $\rho_{\ess}(z) \leq \sigma R_n$ and $\rho(z) \leq R_n$, from which
the upper bounds follow immediately. It remains to prove the lower bound. 

Assume that $f$ is positive and bounded away from zero, and let $r\in\R^d$. Fix $\varepsilon \in (0, \frac{1-\lambda^{-1}}2)$, 
and let $u \in \mathscr{C}^{\infty}([0,1])$ be compactly supported in $(0,1)$,
satisfying $0\leq u(s) \leq 1$ and such that $u(s) = 1$ for all $s \in
[\varepsilon,1-\varepsilon]$. 
Note that we can write
\[
\langle \cL_{r,F}^nf(x), u\rangle =  \lambda^{-n} \sum_{a\in \mathcal{A}_n} \langle f(x_a),  e^{r\cdot S_nF} \, \tilde{u}_a \rangle,
\] 
where $\mathcal{A}_n$ is any good partition of $T^{-n}I(x)$, and for $a\in \mathcal{A}_n$, $x_a$ is the left endpoint of the interval $I_a$ and $\tilde{u}_a$ is the restriction to $I_a$ of $u\circ T^n$ times the element of the partition of unity supported on $I_a$.
Note that $\tilde{u}_a(s) \geq u\circ T^n (s) \theta(s)$ by definition of good partition. 
Let $ \mathcal{A}_n'$ denote the set of indices for which $u\circ T^n = 1$. By positivity, we have
\[
\langle \cL_{r,F}^nf(x), u\rangle \geq \lambda^{-n} \sum_{a\in \mathcal{A}_n'}  \min_{I_a} {  e^{r \cdot S_{n}F}} \langle f(x_a),  \theta \rangle.
\]
Thus, there exists a constant $\Theta = \Theta_f > 0$ such that 
\[
\langle \cL_{r,F}^nf(x), u\rangle \geq \Theta \lambda^{-n} \sum_{a\in \mathcal{A}_n'}  \min_{I_a} {  e^{r \cdot S_{n}F}}.
\]
Now on the intervals $I_a$ the ratio
\[
\frac{\min_{I_a}e^{r\cdot S_{n} F}}{\max_{I_a}e^{r \cdot S_{n} F}}
\]
is uniformly bounded independent of $n$. Indeed, for any $n$,
any $x$, any $t \in [-1,1]$,
\[
	e^{r \cdot (S_n F (x) - S_n F(\phi_t(x)))}
	\leq
	e^{{|r|} t\sum_{i=0}^{n-1} \lambda^{-i}
	\|F\|_{\mathscr{C}^1}}
	\leq
	e^{{|r|} \frac{1}{1-\lambda^{-1}}
	\|F\|_{\mathscr{C}^1}} =: K.
\]
Therefore, we obtain 
\begin{equation}\label{eq:positive_bdd_away_from_0}
\langle \cL_{r,F}^nf(x), u\rangle \geq \frac{\Theta}{K} \lambda^{-n} \sum_{a\in \mathcal{A}_n'}  \|{  e^{r \cdot S_{n} F}} |_{I_a}\|_{\mathcal{C}^0}.
\end{equation}

Recall that
\[
	R_{n}^{n} = \lambda^{-n} \sup_{x\in S_{\reg}} \sup_{\mathcal{A} \in
	\mathcal{P}_{n}(x)} \sum_{a\in \mathcal{A}} \|{  e^{r \cdot S_{n} F}}
	|_{I_a}\|_{\mathcal{C}^0}.
\] 
Fix any $\delta >0$, let $x \in  S_{\reg}$ and $\mathcal{A}$ be such that 
\[
 \lambda^{-n} \sum_{a\in \mathcal{A}} \|{  e^{r \cdot S_{n} F}}
	|_{I_a}\|_{\mathcal{C}^0} > R_{n}^{n} - \delta.
\]
Let $y \in S_{\reg}$ so that $T(x) = \phi_{\varepsilon}(y)$. By \eqref{eq:positive_bdd_away_from_0}, we get
\[
\langle \cL_{r,F}^{n+1}f(y), u\rangle \geq \frac{\Theta}{K} \lambda^{-(n+1)} \sum_{a\in \mathcal{A_{n+1}}}  \|{  e^{r \cdot S_{n+1} F}} |_{I_a}\|_{\mathcal{C}^0},
\]
where we chose the partition of unity $\mathcal{A}_{n+1} \in \mathcal{P}_{n+1}(y)$ so that  $ \mathcal{A}_{n+1}' = \mathcal{A}$.
Then, we conclude,
\[
\langle \cL_{r,F}^{n+1}f(y), u\rangle \geq \frac{\Theta}{K} e^{-{  |r \cdot F|_{\mathscr{C}^0}}} \lambda^{-(n+1)} \sum_{a\in \mathcal{A}}  \|{  e^{r \cdot S_{n} F}} |_{I_a}\|_{\mathcal{C}^0} >  \frac{\Theta}{\lambda K} e^{-{  |r \cdot F|_{\mathscr{C}^0}}}(R_{n}^{n} - \delta).
\]
Since $\delta$ was arbitrary, we conclude that
\[
\|\cL_{r,F}^n f\|_{p,q}
\geq
 \frac{\Theta_f}{\lambda K R_1} e^{-{  |r \cdot F|_{\mathscr{C}^0}}} R_n^n.
\]
Applying the previous estimate to the case $f=\one$, we deduce that there exists a constant $C >0$ such that $ \|\cL_{r,F}^n\|_{p,q} \geq C R_n^n$.
The spectral radius formula yields
\[
\rho(r) \geq \lim_{n\to \infty} \|\cL_{r,F}^n\|_{p,q}^{\frac{1}{n}} \geq R(r).
\]
Thus we conclude that $\rho(r) = R(r)$.
\end{proof}

\begin{corollary}\label{cor:estim_spectral_radius}
Let $z\in \C^d$ be such that $\Re z = r$, then $\rho(z) \leq \rho(r)$.
\end{corollary}

\begin{proof}
From \Cref{lem:estimates_spectral_radius} and \Cref{rk:Rn_and_R}, we conclude
\[
\rho(z) \leq R(z) = R(\Re z) = \rho(r). \qedhere 
\]
\end{proof}

In the next two lemmas, we prove that, for parameters $r\in \R^d$, there exists an eigenvector $\ell_{+}$ for $\cL_{r,F}$ with eigenvalue $\rho(r)$, which is a positive and bounded away from zero functional.

\begin{lemma}
Let $r\in \R^d$, and let $p,q \geq 1$. Then, the operator $\cL_{r,F}$ acting on $\cB_{p,q}$ has a real eigenvalue $\rho(r)$ and a positive eigenvector $\ell_{+} \in \cB_{p,q}$.
\end{lemma}

\begin{proof}
The Lasota-Yorke Inequality in \Cref{prop:LY} and the estimates in \Cref{lem:estimates_spectral_radius} imply that we can write $\cL_{r,F} = \cL = A + Q$, where $\mathcal{A} = \range(A)$ (the range of the operator $A$) is a finite dimensional invariant subspace of $\cB_{p,q}$, $AQ=QA=0$, and $\rho(A) = \rho(r)$ and $\rho(Q) < \rho(r)$.

Let $\Lambda = \{ \ell \in \mathcal{A} : \ell \text{ is positive}\}$. Let us first check that $\Lambda \neq \{0\}$. Take a positive functional $f$ 
that is bounded away from zero (e.g., a strictly positive smooth function on $S$), and consider $R_n^{-n} \cL^n f = R_n^{-n} A^n f +R_n^{-n}  Q^n f$.
By  \Cref{lem:estimates_spectral_radius} and \Cref{prop:LY}, there exist two positive constants $C_f' \geq C_f >0$ such that $C_f \leq \|R_n^{-n} A^n f \|_{p,q} \leq C_f'$.
Since  $\mathcal{A}$ is finite dimensional, there is an accumulation point $f_0 \neq 0$ of the sequence $R_n^{-n} A^n f $. 
As $R_n^{-n} \cL^n f$ is positive for all $n$ and $\rho^{-n}Q^n f \to 0$, 
$f_0$ must be positive, and hence $f_0 \in \Lambda \setminus \{0\}$.

It is easy to verify that $\Lambda$ is a closed convex set in $\mathcal{A}$.
Since $r \in \R^d$, the operator $\cL = \cL_{r,F}$ preserves the cone of positive
functionals; therefore $\cL(\Lambda) \subseteq \Lambda$. By the Brouwer
Fixed Point Theorem applied to the projectivisation of $\Lambda$, there exists $\ell_{+}
\in \Lambda$ such that $\cL \ell_{+} = \rho \ell_{+}$, which proves
the result.
\end{proof}

\begin{lemma}\label{lem:eigenfunction_positivity} 
	Let $r\in \R^d$. If $f\in \cB_{p,q}$ is a positive eigenfunction of $\cL_{r,F}$, then either $f$ is zero
	or it is bounded away from zero.
\end{lemma}

\begin{proof}
Let $f$ be a positive eigenfunction of $\cL = \cL_{r,F}$, with $\cL f = \mu f$ for some $\mu$.
Suppose that $f$ is not bounded away from zero, so there exists $u_0\in \mathscr{C}^q_c(0,1)$, with $u_0\neq 0$, such that, for every $\varepsilon >0$, there exists $x_{\varepsilon} \in S_{\reg}$ such that $\langle f(x_{\varepsilon}),u_0\rangle \leq \varepsilon$.
We will show that, for every $x \in S_{\reg}$, for every non-negative $u\in \mathscr{C}^q_c(0,1)$, and for every $\varepsilon >0$, we have $\langle f(x),u\rangle \leq \varepsilon$. This proves that $f =0$ as an element of $\cB_{p,q}$.
In fact, by positivity, it is sufficient to prove the claim above for any fixed non-negative function; in particular, it is sufficient to prove the claim with $u = \theta$ as in \Cref{def:pou}.

Let us fix $x \in S_{\reg}$ and $\varepsilon >0$. 
There exists $\varepsilon_1 >0$ such that $|\langle f(x),\theta\rangle - \langle f(y),\theta\rangle|\leq \varepsilon/2$ for every $y \in S_{\reg}$ at distance at most $\varepsilon_1$ from $x$.
Since $u_0\neq 0$, there exist $0<b<c<1$ and $m>0$ such that $u_0(s)\geq m$ for all $s\in [b,c]$. Let $n_0 \in \N$ be such that, if $J$ is any horizontal segment of length at least $c-b$, then $T^{-n_0}J$ is $\varepsilon_1$-dense in $S$. Finally, let 
\[
\delta = \frac{m\varepsilon}{2}(\rho(r)\lambda)^{-n_0}e^{-|r| \, n_0 \, \|F\|_{\mathscr{C}^0}},
\]
and let $x_{\delta} \in S_{\reg}$ be such that $\langle f(x_{\delta}),u_0\rangle \leq \delta$, as by our assumption.

Let $J = \phi_{[b,c]}(x_{\delta})$ and denote by $y$ the starting point of a horizontal unit segment in $T^{-n_0}J$ at distance less than  $\varepsilon_1$ from $x$.
As in the proof of \Cref{lem:estimates_spectral_radius}, by positivity, we have
\[
\begin{split}
\delta \geq \mu^{-n_0}\langle \cL^{n_0}f(x_{\delta}),u_0\rangle \geq  \mu^{-n_0} \lambda^{-n_0}  m e^{-|r| \, n_0 \, \|F\|_{\mathscr{C}^0}} \langle f(y),\theta\rangle.
\end{split}
\]
This proves that $ \langle f(y),\theta\rangle \leq \varepsilon/2$ and hence our claim.
\end{proof}

The next two lemmas focus on the eigenvalues of maximal modulus: we prove that the associated eigenspaces have no Jordan blocks and, for parameters $r\in \R^d$, the eigenspace is one-dimensional.

\begin{lemma}\label{lem:no_Jordan_blocks}
	Let $r \in \R^d$ and $z \in \C^d$ with $\Re z = r$. Let $\eta$ be an eigenvalue
	of $\cL_{z,F}$, and assume that $|\eta| = \rho(z)= \rho(r)$. Then, the
	associated eigenspace has no Jordan blocks. 
\end{lemma}

\begin{proof}
	Assume that there is a non-trivial Jordan block; then there exist
	$\mathcal{g},\mathcal{l} \in \cB_{p,q}$ such that
	$\cL_{z,F}\mathcal{l} = \eta \mathcal{l}$ and $\cL_{z,F}\mathcal{g} = \eta
	\mathcal{g} + \mathcal{l}$. But then there exists a constant $C_z>0$
	(depending on the norms of $\mathcal{g}$ and $\mathcal{l}$ as well) such
	that 
	\[
	\|\cL_{z,F}^n \mathcal{g}\|_{p,q} \geq C_z n |\eta|^{n-1} = C_z n \rho(r)^{n-1}.
	\]
	Let $\ell_{+} \in \cB_{p,q}$ be the normalized, positive and bounded away from zero eigenvector $\cL_{r,F}\ell_{+} = \rho(r)\ell_{+}$. 
	Then, up to changing the constant $C_z$, by the Lasota-Yorke Inequality, \Cref{rk:Rn_and_R}, and  \eqref{eq:lowerbound}, we have
	\[
	n \rho(r)^{n-1}  \leq C_z \|\cL_{z,F}^n \mathcal{g}\|_{p,q} \leq  C_z \|\mathcal{g}\|_{p,q} R_n^n(z) = C_z \|\mathcal{g}\|_{p,q} R_n^n(r) \leq  C_z C_{\ell_{+}} \|\mathcal{g}\|_{p,q} \rho(r)^n,
	\]
	which is impossible.
\end{proof}

\begin{lemma}\label{lem:rho_simple}
For  $r\in \R^d$, the eigenvalue $\rho(r)$ of $\cL_{r,F}$ is simple.
\end{lemma}

\begin{proof}
\Cref{lem:no_Jordan_blocks} implies, in particular, that there cannot be a
Jordan block associated to the maximal eigenvalue $\rho = \rho(r)$. Now let us show
that there can be no second eigenfunction $\ell'$ for the eigenvalue $\rho$.
Suppose $\ell'$ exists. Consider the following infimum:
\[
h =  \inf_{\substack{a\in \mathcal{C}^{q}\\ a \geq 0 \\ \|a\|_{\mathcal{C}^{q}} = 1}}  \inf_{x\in S_{\reg}}\frac{\langle \ell' (x), a \rangle }{\langle \ell_+(x), a \rangle }.
\]
Note that $\ell'$ may not be a positive functional, so $h$ can be negative; however, since $\ell_{+}$ is bounded away from zero by \Cref{lem:eigenfunction_positivity}, then $h$ is finite. 

Now, $\ell' -h \ell_+$ is also an eigenfunction with eigenvalue $\rho$. It is positive, since for any $x\in S_{\reg}$, any positive $a\in \mathcal{C}^{q}$, 
\[
\langle \ell'(x), a \rangle \geq h \langle \ell_+(x), a \rangle.
\]
However, $\ell' - h \ell_+$ is not bounded away from zero: for any $\epsilon >0$, there exist $x\in S_{\reg}$ and $a\in \mathcal{C}^{q}$ such that 
\[
\frac{\langle \ell' (x), a \rangle }{\langle \ell_+(x), a \rangle } < h + \epsilon,
\]
therefore
\[
\langle (\ell' -h\ell_+)(x), a\rangle < \epsilon \langle \ell_+(x), a \rangle \leq  \epsilon \|\ell_{+}\|_{p,q}.
\]
Thus by \Cref{lem:eigenfunction_positivity} we must have $\ell' - h\ell_+ = 0$, so $\ell'$ is a scalar multiple of $\ell_+$.
\end{proof}


For the reader's convenience, in the following we collect the main consequences of the results we have proven so far in this section.
\begin{corollary}\label{cor:spectral_decomposition}
Fix $p,q\ge 1$ and let $z\in\C^d$. Set $r\coloneqq\Re z$ and let $\cL\coloneqq\cL_{z,F}$ act on $\cB_{p,q}$.
The triple $(\cL_{z,F}, \cB_{p,q},\cB_{p-1,q+1})$ satisfies the hypothesis of Hennion's theorem~\cite{Hen} and we have:

\begin{enumerate}
\item[(i)] The operator $\cL$ is quasi-compact on $\cB_{p,q}$ and its essential spectral radius satisfies
\begin{equation}\label{eq:ess_rad_cor}
\rho_{\ess}(z) \leq \lambda^{-\min\{p,q\}}\,R(r).
\end{equation}
In particular, for every $\tau$ such that
\[
\lambda^{-\min\{p,q\}}R(r)<\tau<\rho(z) \leq \rho(r),
\]
the part of the spectrum of $\cL$ outside $\{|\xi|\le \tau\}$ consists of finitely many eigenvalues
of finite algebraic multiplicity.

\item[(ii)]
Let $\mathcal{S}_{+}$ be the (possibly empty) set of eigenvalues $\eta \in \C$ with $|\eta|=\rho(r)$, and let $\mathcal{S}_{0}$ be the (finite) set of eigenvalues $\eta \in \C$ with $\tau < |\eta| < \rho(r)$.
There exist bounded finite rank operators $\Pi_{\eta}, N_{\eta}, Q \colon \cB_{p,q}\to\cB_{p,q}$ such that
\begin{equation}\label{eq:spectral_decomp_cor}
\begin{split}
&\cL \;=\; \sum_{\eta \in \mathcal{S}_{+}} \eta\,\Pi_{\eta} \;+\;  \sum_{\eta \in \mathcal{S}_{0}} (\eta\,\Pi_{\eta} +N_{\eta}) + Q,
\qquad
\Pi_{\eta}\Pi_{\eta'}=\delta_{\eta \eta'}\Pi_{\eta},\\
&
\Pi_{\eta}N_{\eta'}=N_{\eta}\Pi_{\eta'}=\delta_{\eta \eta'}N_{\eta},
\qquad
N_{\eta}N_{\eta'}=\delta_{\eta \eta'}N_{\eta}^2,\\
&
\Pi_{\eta}Q=Q\Pi_{\eta}=N_{\eta}Q= QN_{\eta}=0,
\end{split}
\end{equation}
and there exists $j\geq 0$ such that 
\begin{equation}\label{eq:remainder_bound_cor}
N_{\eta}^j = 0, \quad \text{and} \quad
\rho(Q)\le \tau.
\end{equation}
In particular, there exists $\tau < \rho(r)$ such that, for all $n\ge 0$,
\begin{equation}\label{eq:iterate_decomp_cor}
\cL^n \;=\;  \sum_{\eta \in \mathcal{S}_{+}} \eta^n\,\Pi_{\eta} \;+\; O(\tau^n).
\end{equation}

\item[(iii)] 
If $z=r\in\R^d$, then $\rho(r)=R(r) \in \mathcal{S}_{+}$ and the associated eigenvector $\ell_+\in\cB_{p,q}$ for $\cL_{r,F}$ is positive and bounded away from zero. 
Moreover, $\rho(r)$ is simple.
\end{enumerate}
\end{corollary}

\begin{remark}
In the present setting, the presence of conical singularities does not prevent one from shrinking the essential spectral radius arbitrarily by increasing the regularity of the anisotropic Banach spaces. This is in sharp contrast with systems exhibiting genuine discontinuities, where the essential spectrum reflects intrinsic dynamical obstructions that cannot, in general, be removed by refining the underlying functional setting; see, for instance,~\cite{BuLi,BuCaJa,BuCaCa,BaCa}.
\end{remark}

\subsection{The peripheral spectrum}

In this subsection, we study in more detail the eigenvalues of maximal modulus and the associated eigenfunctions.

We first show that if $\cL_{z,F}$ has an eigenvalue $\eta$ with $|\eta|=\rho(\Re z)$, then the associated eigenvector is
 not only a distribution but, in fact, a (complex) measure.

\begin{lemma}\label{lem:evectors_are_measures}
	Let  $r \in \R^d$ and  $z \in \C^d$ with $\Re z = r$. Let $\eta$ be an eigenvalue
	of $\cL_{z,F}$, and assume that $|\eta| = \rho(r)= \rho(z)$. If
	$\mathcal{l}$ is an associated eigenvector, then $\mathcal{l}(x)$ is a
	(complex) measure for all $x \in S_{ \reg }$.
\end{lemma}

\begin{proof}
We begin by recalling that \Cref{lem:no_Jordan_blocks} implies that there are no
Jordan blocks for the eigenvalue $\eta$. Let $\mathcal{l}$ be an associated
eigenvector. Then, there exists $f\in \mathscr{C}^{\infty,\bdd}(S)$ so that
$\mathcal{l} = \Pi_{\eta}f$, where $\Pi_{\eta}$ denotes the projection onto the
eigenspace of $\eta$. By \eqref{eq:iterate_decomp_cor} in \Cref{cor:spectral_decomposition}, we can write
\begin{equation}\label{eq:expression_projection}
\mathcal{l} = \Pi_{\eta}f = \lim_{n \to \infty} \frac{1}{n} \sum_{k=0}^{n-1} \eta^{-k}\cL_{z,F}^k f.
\end{equation}
Let $u \in \mathcal{C}^q$. For any $x \in S_{ \reg }$, by \eqref{LY-1}, we have
\[
\begin{split}
|\langle \mathcal{l}(x), u\rangle|
&\leq \limsup_{n \to \infty} \frac{1}{n} \sum_{k=0}^{n-1} |\eta|^{-k} \, |\langle \cL_{z,F}^k f(x),u\rangle | \\
&\leq \limsup_{n \to \infty} \frac{1}{n} \sum_{k=0}^{n-1} |\eta|^{-k} \, \|\cL_{z,F}^k f\|_{p,0} \, \|u\|_{\mathcal{C}^0} \\
& \leq  \limsup_{n \to \infty} \frac{C_{p,z,F} }{n} \sum_{k=0}^{n-1} |\eta|^{-k} \, R_k^k(z) \|f\|_{p,0} \, \|u\|_{\mathcal{C}^0}, 
\end{split}
\]
for some constant $C_{p,z,F}$ depending on $p,z$, and $F$. Since, by assumption, $|\eta|= \rho(r)$ and $R_k^k(z) = R_k^k(r)$, up to changing the constant $C_{p,z,F}$, we conclude
\[
|\langle \mathcal{l}(x), u\rangle|\leq C_{p,z,F} \, \|f\|_{p,0} \, \|u\|_{\mathcal{C}^0},
\]
which proves the claim.
\end{proof}

Each element $\mathcal{l}$ as in the statement of \Cref{lem:evectors_are_measures} induces, via a partition of unity, a complex measure on orbits $\{\phi_t(x) \ : \ t\in\R \}$ of the translation flow, which we denote $\widehat{\mathcal{l}}(x)$. 
In the language of \cite[Section 9]{GoLi2}, $\widehat{\mathcal{l}}$ is a \emph{continuous leafwise measure}.

We will use the following result, which is an adaptation of Theorem 9.1 and Proposition 9.4 from \cite{GoLi2}. The proof is a minor modification of the proof in \cite{GoLi2}, but for completeness we include it in Appendix \ref{sec:leafwise_measures}.

\begin{proposition}
\label{prop:leafwise_measures}
Let $p,q \geq 1$ and let $\ell_1, \mathcal{l}_2 \in \cB_{p,q}$ be such that for any $x\in S_\mathrm{reg}$, $\ell_1(x)$ is a positive measure and $\mathcal{l}_2(x)$ is a non-zero complex measure. Let $\widehat{\ell_1}, \widehat{\mathcal{l}_2}$ be the corresponding leafwise measures on horizontal leaves. Assume that there exists a $C>0$ such that $|\widehat{\mathcal{l}_2}| \leq C \widehat{\ell_1}$. Further assume that there is a continuous positive function $g$ on $S$, Hölder on horizontal leaves, and a $\gamma \in \C$ with $|\gamma| = 1$ such that $\widehat \ell_1 = g T_* \widehat\ell_1$ and $\widehat{\mathcal{l}_2} = \gamma g T_* \widehat{\mathcal{l}_2}$. 

Then $\gamma=1$ and there exists a $c\in \C$ such that $\mathcal{l}_2 = c \ell_1$.
\end{proposition}

To apply the above result we will need the following estimate, which is a corollary of \Cref{lem:evectors_are_measures}.
\begin{corollary}\label{cor:evectors_are_measures_2}
Let $r\in \R^d$,  $z\in \C^d$, and $\mathcal{l}$ be as in the statement of \Cref{lem:evectors_are_measures}. Then $|\widehat{\mathcal{l}}| \leq C \widehat{\ell_{+}}$, where $|\widehat{\mathcal{l}}|$ is the variation of $\widehat{\mathcal{l}}$ and $\ell_{+}$ is the eigenvector of $\cL_{r,F}$ with eigenvalue $\rho(r)$.
\end{corollary}

\begin{proof}
Let $f\in \mathscr{C}^{\infty,\bdd}(S)$ so that $\mathcal{l} = \Pi_{\eta}f$. Since the eigenvalue $\rho(r)$ of $\cL_{r,F}$ is simple, by \Cref{lem:rho_simple}, up to a constant we can write $\ell_{+} = \Pi_{\rho(r)}1$. Let $u$ be a continuous function on the orbit of the point $x$. We can suppose that $u$ is supported on the segment $\gamma = \{ \phi_t(x) \ : \ t \in (0,1)\}$. Then, using \eqref{eq:expression_projection}, we have
\[
\begin{split}
\left\lvert \int_{\gamma} u \diff \widehat{\mathcal{l}}(x) \right\rvert
&= |\langle \mathcal{l}(x),u \rangle|
\leq \limsup_{n \to \infty} \frac{1}{n} \sum_{k=0}^{n-1} |\eta|^{-k} \, |\langle \cL_{z,F}^kf(x),u \rangle| \\
&\leq \limsup_{n \to \infty} \frac{\|f\|_{\infty}}{n} \sum_{k=0}^{n-1} \rho(r)^{-k} \, \langle \cL_{r,F}^k1(x), |u| \rangle \\
&= \|f\|_{\infty}\langle \ell_{+}(x), |u| \rangle
= \|f\|_{\infty} \int_{\gamma} |u| \diff \widehat{\ell_{+}}(x),
\end{split}
\]
which completes the proof.
\end{proof}

We can now show that, under a natural assumption on~$F$, there are no eigenvalues $\eta$ of $\cL_{z,F}$ of modulus $\rho(\Re z)$ unless $z=\Re z$ and $\eta = \rho(\Re z)$.


\begin{definition}
For a set $A \subset S$ we say that a function $f: A \to \R$ is \emph{cohomologous to a constant  $\hspace{-0.5em}\pmod{2\pi}$} if there exist a constant $c$ and a measurable function $\alpha: A \to \R$ such that for Lebesgue-almost all $x\in A$,
\[f(x) - c - (\alpha(Tx)-\alpha(x)) \in 2\pi \Z.\]
\end{definition}

\begin{proposition}\label{prop:unique_max_eigenvalue}
	Assume that, for all $v\in (-\pi,\pi)^d \setminus \{0\}$, the function $F\cdot v$ is not cohomologous to a
	constant $\hspace{-0.5em}\pmod{2\pi}$.

	Let $z \in \R^d +\oplus i (-\pi,\pi)^d$, let $r = \Re(z)$ and let $\eta\in\C$ be an eigenvalue of $\cL_{z,F}$ with
	$|\eta|=\rho(r)$. Then, $r=z$, $\eta = \rho(r)$, and the eigenvalue $\rho(r)$ is simple. 
	
	As a consequence, if $z\notin \R^d$, then $\rho(z) < \rho(\Re(z))$.

\end{proposition}

\begin{proof}
Let ${z\in \R^d +\oplus i (-\pi,\pi)^d}$ and let us consider $\cL_{z,F}$. Suppose $\eta$ is an eigenvalue of $\cL_{z,F}$, with $\eta = e^{ic}\rho$ for some $c\in \R$, where $\rho = \rho(r)$. Then, there exists some $\mathcal{l}\in \cB_{p,q}$ so that $\cL_{z,F} \mathcal{l} = e^{ic} \rho \mathcal{l}$. 

We claim that $|\widehat{\mathcal{l}}| = \widehat{\ell_{+}}$. Indeed, we have
\[
|\widehat{\mathcal{l}}| = \left|e^{-ic}\rho^{-1} \widehat{\cL_{z,F}\mathcal{l}}\right| = \rho^{-1} \left|{  e^{(z \cdot F)\circ T^{-1}}}\, T_{\ast}\widehat{\mathcal{l}} \right|
= \rho^{-1} {  e^{(r \cdot F) \circ T^{-1}}}\, T_{\ast}|\widehat{\mathcal{l}}|.
\]
Let $\ell_+$ be an eigenvector of $\cL_{r,F}$, where $r = \Re(z)$, with eigenvalue $\rho$.
Since also 
\[\widehat{\ell_+} = \rho^{-1}\widehat{ \cL_{r,F} \ell_+} =  \rho^{-1} e^{(r \cdot F) \circ T^{-1}}  T_\ast \widehat{\ell_+},\]
the claim follows from \Cref{prop:leafwise_measures} applied to $|\widehat{\mathcal{l}}|$ and $\widehat{\ell_{+}}$, with $\gamma = 1$ and $g = \rho^{-1} e^{(r \cdot F)\circ T^{-1}}$.

Thus, $|\widehat{\mathcal{l}}| = \widehat{\ell_{+}}$ and hence for all $x\in S_{\reg}$, there exists a bounded real-valued function $\alpha(x)$ on the orbit of $x$ so that $\widehat{\mathcal{l}}(x) = e^{i \alpha(x)} \widehat{\ell_{+}}(x)$. Since $\widehat{\mathcal{l}}$ and $ \widehat{\ell_{+}}$ vary continuously, in each open rectangle in $S$, $\alpha(x)$ varies continuously on orbits; in particular, we can identify $\alpha$ with a measurable function on $S$.

We have, suppressing for convenience the dependence
on $x$, 
\[
\begin{split}
\rho^{-1} \cL_{r,F} \ell_{+} &= \ell_{+}
= e^{-i \alpha}\mathcal{l}
= e^{-i \alpha} \rho^{-1} e^{-ic} \cL_{z,F} \mathcal{l} \\
&= e^{-i \alpha} \rho^{-1} e^{-ic} {  \cL_{r,F} (e^{i \Im(z) \cdot F}\mathcal{l})} \\
&= e^{-i \alpha} \rho^{-1} e^{-ic} {  \cL_{r,F} (e^{i \left(\Im(z) \cdot F + \alpha\right)}\ell_+)} \\
&= \rho^{-1} \cL_{r,F} \big(e^{i \left(-\alpha\circ T -c + \Im(z) \cdot F + \alpha\right)}\ell_{+}\big).
\end{split}
\]
Comparing the first and last term we see that for any $x$,
\[
	\alpha(x) -\alpha \circ T (x) + \Im(z) \cdot F(x) -c = 0 \quad \text{(mod $2\pi$)},
\]
which contradicts the assumption that $F \cdot v$ is not cohomologous to a
constant $\hspace{-0.5em}\pmod{2\pi}$ for all $v\in (-\pi,\pi)^d\setminus\{0\}$, taking $v=\Im z \neq 0$. Thus in fact $\Im(z) = 0$ and $z=r$.

Given that $z=r$, we can apply \Cref{prop:leafwise_measures} again to $\widehat{\mathcal{l}}$ and $\widehat{\ell_+}$, now with $\gamma = e^{-ic}$, to see that $e^{-ic} = 1$ and $\mathcal{l} = \ell_+$. Thus we deduce that $\eta = \rho$, and $\rho$ is simple.

For the last claim, we know from \Cref{cor:estim_spectral_radius} that $\rho(z)\leq \rho(\Re(z))$. Since we have shown that for $z\in \C^d\setminus\R^d$, there is no eigenvalue of $\cL_{z,F}$ of modulus $\rho(\Re(z))$, we conclude that $\rho(z) < \rho(\Re(z))$.
\end{proof}


\subsection{The peripheral spectrum of the dual transfer operator}\label{sec:periph_spectrum_dual}
Thanks to the result of the last section, we know that, when
$r\in \R^d$, the operator $\cL_{r,F}$ has a simple eigenvalue at $\rho(r)$, with a positive eigenvector
$\ell_+$. Let us assume that $\ell_+$ is normalized so that $\ell_{+} = \lim_{n\to \infty} \rho(r)^{-n} \cL_{r,F}^n(\one)$. In this section, we show that the dual to $\ell_+$ is a (positive)
probability measure. 

Let $\cB'_{p,q}$ denote the dual space of $\cB_{p,q}$ and let $\cL'_{z,F}$ be the dual of the transfer operator $\cL_{z,F}$, acting on
$\cB'_{p,q}$. 

\begin{remark}\label{rk:discrete_spectrum_of_dual}
Let $z\in \C^d$, and let $\mathcal{S} = \mathcal{S}_{+} \cup  \mathcal{S}_{0}$ be the discrete spectrum of $\cL_{z,F}$ on $\cB_{p,q}$, as in \Cref{cor:spectral_decomposition}.
For each $\eta \in \mathcal{S}$, the space $\Pi_{\eta}\cB_{p,q}$ is finite dimensional and there exists $\ell_{\eta} \in \Pi_{\eta}\cB_{p,q}$ such that $\cL_{z,F}\ell_{\eta} = \eta \ell_{\eta}$. Let $L_{\eta}$ be a subspace such that $\C\ell_{\eta} \oplus L_{\eta} = \Pi_{\eta}\cB_{p,q}$; then, for every $f \in \cB_{p,q}$, we can write $\Pi_{\eta}f = \nu_{\eta}(f) \ell_{\eta} + f_{\eta}$, where $f_{\eta} \in L_{\eta}$. It is easy to check that $ \nu_{\eta} \in \cB_{p,q}'$ and $\cL_{z,F}'\nu_{\eta} = \eta \nu_{\eta}$; that is, $\eta$ is an eigenvalue of the dual operator $\cL_{z,F}'$, and $\nu_{\eta}$ an associated eigenvector. 
\end{remark}

By \Cref{lem:rho_simple} and \Cref{rk:discrete_spectrum_of_dual}, we have $\Pi_{\rho} f = \nu(f) \ell_{+}$, where $\nu=\nu_{\rho}$ is the only eigenvector of $\cL'_{r,F}$ with eigenvalue $\rho$, normalized so that $\nu(\ell_+) = 1$.

\begin{lemma}\label{lem:nu_is_measure_1}
	There exists a constant $C > 0$ such that $|\nu (f )| \leq C \|f\|_{0,p+q} \le C
	\|f\|_{\mathscr{C}^0}$ for all $f\in
	\mathscr{C}^{\infty,\bdd}(S)$. 
	In particular, $\nu$ induces a (complex) measure on $S$.
\end{lemma}

\begin{proof}
	The eigenvector $\nu$ satisfies $\nu = \rho(r)^{-n} (\cL'_{r,F})^n \nu$ for all $n \in \N$. Hence, for any
	$f\in\cB_{p,q}$, 
	\[
	\begin{split}
		|\nu (f)| &= \rho(r)^{-n} |\nu (\cL^n_{r,F}(f))| 
					\le C \rho(r)^{-n} \| \cL^n_{r,F} (f)\|_{p,q} \\
					&\le C \rho(r)^{-n} (C \sigma^n \rho(r)^n \|f\|_{p,q} + C \rho(r)^n \|f\|_{0,p+q}),
	\end{split}
	\]
	where we used the Lasota-Yorke inequality~\eqref{LY-2}, and induction over
	$p$ to get the last norm. Since $\sigma < 1$, letting $n$ go to infinity, we
	obtain
	\[
		|\nu (f)| \le C \|f\|_{0,p+q}
	\]
	for some constant $C > 0$ and all $f\in\cB_{p,q}$. 
	If $f\in
	\mathscr{C}^{\infty,\bdd}(S)$, we also have $\|f\|_{0,p+q} \leq \|f\|_{\mathscr{C}^0}$, which completes the proof.
\end{proof}

By \Cref{lem:multiplying_by_smooth}, for any $f\in
\mathscr{C}^{\infty,\bdd}(S)$, we can also define
\[
\nu_T(f) \coloneqq \nu (f\ell_{+}).
\]
\Cref{lem:nu_is_measure_1} and \Cref{lem:evectors_are_measures} yield
\[
|\nu_T(f) |\leq C \| f \ell_{+}\|_{0,p+q}\leq C \|f\|_{\mathscr{C}^0},
\]
which implies that $\nu_T$ is also a complex measure on $S$.

\begin{lemma}\label{lem:nu_is_measure_2}
	The measures $\nu$ and $\nu_T$ are positive probability measures. Furthermore, $\nu_T$ is $T$-invariant.
\end{lemma}

\begin{proof}
	We first verify that the measure $\nu_T$ is $T$-invariant: we compute
	\[
	\begin{split}
		\nu (f \ell_+) &= \rho(r)^{-1} (\cL'_{r,F}\nu) (f\ell_+) 
		= \rho(r)^{-1} \nu(\cL_{r,F} (f\ell_+)) \\
		&= \nu (f\circ T^{-1} \cdot \rho(r)^{-1}\cL_{r,F} (\ell_+))
		= \nu (f\circ T^{-1} \cdot \ell_+),
	\end{split}
	\]
	which proves the claim.
	
	By \Cref{cor:spectral_decomposition}, for every
	$\ell \in\cB_{p,q}$, we have that
	\[
	\lim_{n\to\infty} \rho(r)^{-n} \cL^n_{r,F} (\ell) = \nu(\ell) \ell_+.
	\]
	If $f_1$, $f_2$ are positive functions, then, for every $x\in S_{\reg}$,
	\[
	0 \le \langle \lim_{n\to\infty} \cL^n_{r,F}(f_2)(x), f_1 \rangle = \nu(f_2) \langle \ell_+(x), f_1 \rangle. 
	\]
	In particular, since $\ell_+$ is positive and not the zero measure, we can
	find an $x\in S_{\reg}$ and a positive function $f_1$ such that $\langle
	\ell_+(x), f_1 \rangle > 0$. This implies that, if $f \ge 0$, then also
	$\nu(f)\ge 0$. 
	
	We verify that also $\nu_T$ is a positive measure; indeed
	\[
	\nu_T(f) = \nu (f \ell_+)  = \lim_{n\to\infty} \nu(f\cdot\rho(r)^{-n} \cL^n_{r,F}1)
	= \lim_{n\to\infty} \nu(\rho(r)^{-n} \cL^n_{r,F}(f\circ T^n))
	= \lim_{n\to\infty} \nu (f\circ T^n) \ge 0,
	\]
	where, in the last equality we used that $\nu$ is the $\rho(r)$-eigenvector
	of the dual of $\cL_{r,F}$.  The
	normalization $\nu(\ell_+)=1$ implies that $\nu_T$ is a probability measure. Finally, by our normalization of $\ell_{+}$, namely $\ell_{+} = \lim_{n\to \infty} \rho(r)^{-n} \cL_{r,F}^n(\one)$, we conclude that $\nu(\one) = \lim_{n\to \infty} [\rho(r)^{-n} (\cL_{r,F}')^n(\nu)](\one) = \nu(\ell_{+}) = 1$, so that $\nu$ is a probability measure as well.
\end{proof}

We further study the relation between the measures $\nu$ and $\nu_T$: we now show that, once they are disintegrated along the orbits of the flow $\{\phi_t\}_{t \in \R}$, they have the same transverse measure. 
We start with some preliminaries.

Let $\Omega \subset S_0$ be an open rectangle with vertical side $\Omega_v$, and let $f\in \mathscr{C}^{\infty,\bdd}(S)$ be a smooth function supported in $\Omega$. For any $\ell \in \cB_{p,q}$, we can define  
\[
f(\ell) \coloneqq \int_{\Omega_v} \langle \ell(x),f_x \rangle \diff x,
\]
where $\diff x$ denotes the Lebesgue measure on the vertical side $\Omega_v$ and $f_x$ is the function $f_x\colon s \mapsto f\circ \phi_{s}(x)$, which can be seen as an element of $\mathcal{C}^q = \cC^q_c(I)$.
Then, by \Cref{prop:localize}, as in the proof of \Cref{lem:inclusion}, we have an embedding 
\[
\cC^{\infty, \bdd}(S) \hookrightarrow \cB_{p,q}'. 
\] 
\begin{lemma}\label{lem:expression_for_L_prime}
Let $r\in \R^d$. For any $f \in \cC^{\infty, \bdd}(S)$, we have $\cL_{r,F}'f = e^{r \cdot F} \cdot f\circ T$.
\end{lemma}
\begin{proof}
We can assume that $f$ is supported in an open rectangle. Let $g\in \cC^{\infty, \bdd}(S)$; then
\[
\langle \cL_{r,F}'f, g \rangle = \langle f, \cL_{r,F}g \rangle = \langle f, (e^{r \cdot F} g) \circ T^{-1}\rangle = \langle e^{r \cdot F} \cdot f\circ T, g \rangle, 
\]
where we used the fact that $T$ preserves the Lebesgue measure. Since $r\in \R^d$, and by density of $\cC^{\infty, \bdd}(S)$ in $\cB_{p,q}$, the proof is complete.
\end{proof}

For $r\in \R^d$ and $x\in S_{\reg}$, we define 
\begin{equation}\label{eq:Gx}
	G_{r,x}(s)= G_{x}(s) = \prod_{k=0}^{\infty} e^{r \cdot (F\circ T^k(\phi_s(x)) - F\circ T^k(x))}, \qquad \text{for $s\in [0,1]$.}
\end{equation}
Note that $G_x(s)$ is uniformly bounded in $x$ and $s\in [0,1]$.

\begin{lemma}\label{lem:expression_for_nu}
There exists a measure $\vartheta$ on vertical segments such that, if $\ell \in \cB_{p,q}$ is supported on an open rectangle $\Omega$ with vertical side $\Omega_v$, then
\[
\nu(\ell) = \int_{\Omega_v} \langle \ell(x), G_x \rangle \diff \vartheta(x).
\]
\end{lemma}
\begin{proof}
First of all, we note that, for any $N\in \N$, we can write $e^{r \cdot S_NF}\circ \phi_s(x) \widetilde{G}_{N,r,x}(s) = G_{r,x}(s) e^{r \cdot S_NF(x)}$, where 
\[
\widetilde{G}_{N,r,x}(s) = e^{r \cdot (\sum_{k=N}^{\infty} F\circ T^k(\phi_sx) - F\circ T^k(x))} \ll e^{|r| \, \lambda^{-N}}.
\]

Now, for any $\ell \in \cB_{p,q}$ supported on an open rectangle $\Omega$ as above, define
\[
\nu_N(\ell) \coloneqq \int_{\Omega_v} \rho(r)^{-N} e^{r \cdot S_N F(x)} \langle\ell(x), G_x \rangle \diff x.
\] 
Then, $\nu_N$ can be seen as an element of $\cB_{p,q}'$.
We claim that $\nu_N \to \one(\ell_{+}) \cdot \nu$ in $ \cB_{p,q}'$. Once this is proved, we would deduce $\diff \vartheta(x) = (\one(\ell_{+}))^{-1}\lim_{N\to \infty} \rho(r)^{-N} e^{r \cdot S_N F(x)} \diff x$, and hence prove the result.

For $N\in \N$, we define 
\[
\mu_N \coloneqq \rho(r)^{-N} (\cL_{r,F}')^N (\one) \in \cB_{p,q}'.
\]
Clearly, $\mu_N(\ell_{+}) = \one(\ell_{+}) >0$, which implies that the sequence $\mu_N$ does not converge to $0 \in \cB'_{p,q}$. Hence, since $\rho(\cL_{r,F}') = \rho(r)$, the sequence $\mu_N$ must converge to its projection on the eigenspace with maximal eigenvalue $\rho$; more precisely, we have $\lim_{N\to\infty} \mu_N = \one(\ell_{+}) \nu$. We claim that $\nu_N - \mu_N \to 0$; this would complete the proof. It is sufficient to verify the claim for all nonnegative $f\in \cC^{\infty,\bdd}(S)$ supported on an open rectangle $\Omega$. 
Using \Cref{lem:expression_for_L_prime}, we can rewrite
\[
\begin{split}
\mu_N(f) &= [\rho(r)^{-N} (\cL_{r,F}')^N (\one)](f) = \rho(r)^{-N} e^{r \cdot S_N F} \one(f) = \int_{\Omega_v} \rho(r)^{-N}\langle f_x, e^{r \cdot S_N F} \circ \phi_{\cdot}(x)\rangle \diff x \\
&= \int_{\Omega_v} \rho(r)^{-N}e^{r \cdot S_N F(x)}\langle f_x, \widetilde{G}_{N,r,x}^{-1} G_x\rangle \diff x. 
\end{split}
\]
Therefore, 
\[
\begin{split}
	|\mu_N(f) - \nu_N(f)|\ &= \left\lvert \int_{\Omega_v} \rho(r)^{-N}e^{r \cdot S_N F(x)}\langle f_x, (\widetilde{G}_{N,r,x}^{-1}-1) G_x\rangle \diff x \right\rvert\\
	&= \left\lvert \int_{\Omega_v} \rho(r)^{-N}\langle f_x, (\widetilde{G}_{N,r,x}^{-1}-1) \widetilde{G}_{N,r,x} e^{r \cdot S_N F} \circ \phi_{\cdot}(x) \rangle \diff x \right\rvert \\
	&=\left\lvert [\rho(r)^{-N}(\cL_{r,F}')^N((\widetilde{G}_{N,r,x}^{-1}-1) \widetilde{G}_{N,r,x})]( f) \right\rvert\\
	&\ll \|\rho(r)^{-N}(\cL_{r,F}')^N\|_{0,0} \|(\widetilde{G}_{N,r,x}^{-1}-1) \widetilde{G}_{N,r,x}\|_{\infty} \|f\|_{\infty},
\end{split}
\]
which tends to $0$ as $N \to \infty$. 
This shows that $\nu_N - \mu_N \to 0$ in $\cB_{p,q}'$, and hence completes the proof.
\end{proof}

We then obtain the following corollary.
\begin{corollary}\label{cor:transverse_measure}
The measures $\nu$ and $\nu_T$ can be disintegrated along the orbits of the flow $\{\phi_t\}_{t\in \R}$ and have the same transverse measure $\vartheta$.
\end{corollary}
\begin{proof}
If $f \in \cC^{\infty, \bdd}(S)$ is supported on an open rectangle $\Omega$ with vertical side $\Omega_v$, then, in local coordinates $(s,x) = \phi_s(x)$ for $x \in \Omega_v$, by \Cref{lem:expression_for_nu}, we have
\[
\nu(f) = \int_{\Omega_v} \int_0^1 f(s,x) G_x(s) \diff s \diff \vartheta(x), \qquad \text{and} \qquad \nu_T(f) = \int_{\Omega_v} \int_0^1 f(s,x) G_x(s) \diff \widehat{\ell_{+}}(x)(s) \diff \vartheta(x),
\]
which proves the result.
\end{proof}

We conclude this section by showing that the measure $\nu_T$ has strong mixing properties for $T$.

\begin{proposition}\label{prop:nu_T_is_exponentially_mixing}
The map $T$ is exponentially mixing with respect to $\nu_T$.
\end{proposition}
\begin{proof}
Let $f$ and $g$ be sufficiently smooth functions. Then, we have
\begin{equation}\label{eq:corr}
\nu_T(f\circ T^n \cdot g) = \nu(f \cdot g\circ T^{-n} \cdot \ell_{+}) = \nu( f \cdot \rho(r)^{-n}\cL_{r,F}^n(g\cdot \ell_{+})).
\end{equation}
By \Cref{prop:unique_max_eigenvalue}, there exists $\sigma \in (0,1)$ such that $\rho(r)^{-n}\cL_{r,F}^n (g\cdot \ell_{+})= \nu (g\cdot \ell_{+}) \ell_{+}+ O(\sigma^n)$. Thus,
\[
\nu_T(f\circ T^n \cdot g) = \nu( f \, \ell_{+}) \cdot \nu( g\, \ell_{+}) + O(\sigma^n) = \nu_T(f) \, \nu_T( g ) + O(\sigma^n).
\]
The proof is complete.
\end{proof}

\begin{remark}\label{rmk:nu_T_is_Gibbs}
By following the arguments in~\cite[Section 6.2]{GoLi2}, it is possible to prove that
$\nu_T$ is the unique Gibbs measure of $T$ corresponding to the potential
$r \cdot F - \log \lambda$, for $r\in \R^d$.

In particular, the simple eigenvalue $\rho(r)$ satisfies
\[
	\log \rho(r) = P\big(T,r \cdot F - \log \lambda\big),
\]
where $P(T,\cdot)$ denotes the topological pressure. Equivalently, the variational
principle takes the form
\[
	\log \rho(r) = \sup_{\mu \in \operatorname{Erg}(T)} \left\{
	h_\mu(T) + \int r \cdot F \, \diff\mu - \log \lambda\right\}.
\]
In particular, $\nu_T$ is the unique invariant probability measure achieving the
above supremum.
\end{remark}

\subsection{Perturbations of the maximal eigenvalue} In this section, for $z=r+ 2\pi i\theta$, where $r,\theta \in \R^d$, we consider the operator $\cL_{\theta} \coloneqq \cL_{r+2\pi i\theta,F}$ as a perturbation of $\cL_{r,F}$ with respect to $\theta$, while fixing $r$. From the previous section we know that $\cL_{r,F}$ has a spectral gap with maximal simple eigenvalue $\rho$ with associated eigenfunction $\ell_+>0$. Moreover, the dual operator has $\rho$ as maximal simple eigenvalue as well, with associated eigenfunction $v$ such that $v(\ell_+)=1$\footnote{Note that the quantities $\rho, \ell_+$ and $\nu$ depend on $r$ which is  fixed and, therefore, dropped from the notations.}. Recall also the measure $\nu_T$ from the previous section, which will have a key role in the following.\\
By the Lasota-Yorke inequality and \Cref{prop:unique_max_eigenvalue} we have the following.

\begin{lemma}\label{eq:quasi-twist}
	For each $\theta \neq 0$ the spectrum of $\cL_\theta$ acting on $\cB_{p,q}$
	is contained in the disk $\{w\in \C: |w|\le \rho\}$ and the essential
	spectrum is contained in the disk $\{w\in \C: |w|\le \sigma \rho\}$, for
	some $\sigma\in(0,1)$. Moreover, if, for all $v\in (-\pi,\pi)^d\setminus\{0\}$,
	$F\cdot v$ is not cohomologous to a constant  $\hspace{-0.5em}\pmod{2\pi}$, then the spectrum of
	$\cL_\theta$ is contained in the open disk $\{w\in \C: |w|< \rho\}$.
\end{lemma}

The following results give the analytic perturbations of the maximal eigenvalue
and provide expressions for the first and second derivatives in the Taylor
approximation.

\begin{proposition}\label{prop:pert}
	There exist $\sigma < 1$ and a neighborhood $\mathcal N_0$ of $\theta=0$ for
  	which the following properties hold.
  	\begin{enumerate}[label=(\Alph*)]
	\item \label{item:pertA} $\|\cL_{\theta}^n f\|_{p,q} \leq C \sigma^n \|f\|_{p,q}
	$ for all $\theta \notin \mathcal N_0$.
	\item \label{item:pertB} For all $\theta \in \mathcal N_0$, there exist $\rho_\theta \in \C$, with $\rho_0=\rho$ and $|\rho_\theta| < \rho$, a rank-1 projector $\Pi_\theta$, and a family of operators $\theta \mapsto Q_{\theta}$ with $\Pi_\theta \, Q_{\theta} = Q_{\theta} \, \Pi_\theta = 0$, such that
	\[
	\cL_{\theta} = \rho_\theta \Pi_\theta + Q_{\theta},
	\]
	and $Q_{\theta}$ satisfies $\|Q_{\theta}^n f\|_{p,q} \leq C \sigma^n \| f\|_{p,q}$.
	\item \label{item:pertC} The functions $\theta \to \rho_\theta,\Pi_\theta, Q_{\theta}$ are analytic on $\mathcal N_0$ and, denoting by $D$ the derivative in direction $\theta$, we have the formula
	\begin{equation}\label{eq:logrho}
	D^2\log \rho_\theta|_{\theta=0}=\dfrac{D^2 \rho_\theta|_{\theta=0}}{\rho}-\dfrac{D\rho_\theta|_{\theta=0} \otimes D\rho_\theta|_{\theta=0}}{\rho^2}\eqqcolon 4\pi^2 \Sigma
	\end{equation}
	where
	\begin{equation}\label{eq:taylor}
		\Re(D\rho_\theta |_{\theta=0})=0 \qquad \text{and} \qquad  \Im(D\rho_\theta |_{\theta=0})= 2\pi \rho \nu_T(F)
	\end{equation}
and, setting $\overline{F}\coloneqq F-\nu_T(F)$,
\begin{equation}\label{eq:sigma}
	\Sigma=-\sum_{m\in \mathbb{Z}} \nu_T(\overline{F}\otimes \overline{F} \circ T^m) \qquad \text{and} \qquad \Im(D^2\rho_\theta |_{\theta=0})=0.
\end{equation}
Finally, the covariance matrix $\Sigma$ is non-degenerate if and only if $F
\cdot v$ is not cohomologous to a constant  $\hspace{-0.5em}\pmod{2\pi}$ for all $v\in (-\pi,\pi)^d\setminus \{0\}$.
  	\end{enumerate}
\end{proposition}

\begin{proof}
We first explain why \labelcref{item:pertA,item:pertB}, and the analyticity
statements in \labelcref{item:pertC} follow from standard analytic perturbation
theory, so that it remains only to compute the Taylor coefficients in
\cref{eq:taylor,eq:sigma}.

By \Cref{prop:LY} and Hennion's theorem (as in
\Cref{cor:spectral_decomposition}), each operator $\cL_\theta$ is quasi-compact
on $\cB_{p,q}$, and its essential spectral radius is bounded by
$\rho_{\ess}(\cL_\theta)\le \sigma_0 \rho$ for some $\sigma_0\in(0,1)$,
uniformly for $\theta$ in a neighborhood of $0$. Moreover, by
\Cref{prop:unique_max_eigenvalue} the eigenvalue $\rho=\rho_0$ of $\cL_0$ is
simple and isolated: there exists $R\in(\sigma_0\rho,\rho)$ such that
\[
\operatorname{spec}(\cL_0)\cap\{|\mu|\ge R\}=\{\rho\}.
\]
Since $\theta\mapsto \cL_\theta$ is analytic as a map $\mathcal N_0\to \mathcal L(\cB_{p,q})$
(in the sense of Kato~\cite[Ch.~VII, \S1--3]{Kato}), the Riesz projector
\[
\Pi_\theta \coloneqq \frac{1}{2\pi i}\int_{\Gamma} (\zeta-\cL_\theta)^{-1}\,\diff\zeta,
\]
where $\Gamma$ is a small circle around $\rho$ contained in $\{|\zeta|>R\}$, is well-defined
and depends analytically on $\theta$ for $\theta$ small. In particular, $\Pi_\theta$ has rank
$1$, and $\rho_\theta$ (which is the trace of $\cL_\theta\Pi_\theta$) is the unique eigenvalue of $\cL_\theta$
inside $\Gamma$, depending analytically on $\theta$.
Setting $Q_\theta\coloneqq\cL_\theta-\rho_\theta\Pi_\theta$, we obtain the decomposition
\[
\cL_\theta=\rho_\theta\Pi_\theta+Q_\theta,\qquad \Pi_\theta Q_\theta=Q_\theta\Pi_\theta=0,
\]
and the spectral radius of $Q_\theta$ is $\le R<\rho$. This yields
\labelcref{item:pertB} and the exponential bound on $Q_\theta^n$; moreover
\labelcref{item:pertA} follows by the uniform spectral gap away from $\theta=0$.

All these conclusions are standard consequences of Kato's analytic perturbation
theory for isolated eigenvalues; see~\cite[Ch.~VII, \S1--3]{Kato}. Therefore it
remains to compute the derivatives of $\rho_\theta$ at $\theta=0$, namely
\cref{eq:taylor,eq:sigma}.

In particular, let $\ell_{+,\theta}$ be the eigenfunction of $\cL_\theta$ with eigenvalue $\rho_\theta$, and
$\nu_\theta$ the corresponding dual eigenfunction such that
$\nu_{\theta}(\ell_{+,\theta})=1$. We also consider the corresponding measure
$\nu_{T,\theta}(\cdot )=\nu_\theta(\cdot \ell_{+,\theta})$. Moreover we have,
\begin{enumerate}[label=\roman*)]
	\item \label{item:Pi_i} $\Pi_\theta \cL_\theta=\cL_\theta\Pi_\theta=\rho_\theta \Pi_\theta$,
	\item \label{item:Pi_ii} $\Pi_\theta(f)= \nu_{\theta}(f)\ell_{+,\theta}$,
	\item \label{item:Pi_iii} $\Pi_\theta^2=\Pi_\theta$.
\end{enumerate}

Let $D$ be the derivative in the direction of $\theta$. Differentiating
\labelcref{item:Pi_i}, composing with $\Pi_\theta$ on the left and using
\labelcref{item:Pi_iii}:
\[
	D \rho_\theta \Pi_\theta+\rho_\theta \Pi_\theta D\Pi_\theta=\Pi_\theta D\cL_\theta \Pi_\theta+\Pi_\theta \cL_\theta D\Pi_\theta.
\]
Using again \labelcref{item:Pi_i} for the last term, the above equation yields:
\begin{equation}\label{eq:DLDP}
	D\rho_\theta \Pi_\theta =\Pi_\theta D\cL_\theta \Pi_\theta.
\end{equation}

We observe that if, for some scalar $c$ and some operator $A$, it holds that
$c\Pi_\theta = \Pi_\theta A \Pi_\theta$, then, for any $f$, 
\[
	c \nu_\theta(f) \ell_{+,\theta} = c \Pi_\theta (f) = \Pi_\theta A (\nu_\theta(f) \ell_{+,\theta}) = \nu_\theta(f) \nu_\theta(A \ell_{+,\theta}) \ell_{+,\theta},
\]
and hence
\begin{equation}\label{eq:scalar_formula}
	c = \nu_\theta(A \ell_{+,\theta}).
\end{equation}

Recalling that 
\begin{equation}\label{eq:L1}
\cL_\theta(f)=[e^{(r+2\pi i \theta) \cdot F} f]\circ T^{-1},
\end{equation}
we have
\begin{equation}\label{eq:DL1}
D\cL_\theta (f)|_{\theta=0}=2\pi i\sum_{k=1}^d \cL_{0}(F_k \cdot f) e_k,
\end{equation}
where $F_k$ are the components of $F$. 

Hence 
\begin{equation}\label{eq:DL1.5}
D\cL_\theta|_{\theta=0}(f\ell_{+,\theta}) = 2\pi i\sum_{k=1}^d (F_k f)\circ T^{-1} \rho \ell_{+} e_k.
\end{equation}
Therefore,  
by \eqref{eq:scalar_formula}, \eqref{eq:DL1.5} and the invariance of $\nu_T$,
\[
	D\rho_\theta |_{\theta=0} = \sum_{k=1}^d \nu (D \cL_\theta |_{\theta=0} (F_k \ell_{+})) e_k = \sum_{k=1}^d 2\pi i \rho \nu_T(F_k) e_k,
\]
which proves \eqref{eq:taylor}.

Next, we differentiate \cref{eq:DLDP} and we get
\[
	D^2\rho_\theta \Pi_\theta=D\Pi_\theta D\cL_\theta \Pi_\theta+\Pi_\theta D^2\cL_\theta \Pi_\theta+\Pi_\theta D\cL_\theta D\Pi_\theta-D\rho_\theta D\Pi_\theta.
\]
Composing with $\Pi_\theta$ on the left and on the right, and recalling
\labelcref{item:Pi_iii}, we have
\[
	D^2\rho_\theta \Pi_\theta=\Pi_\theta D\Pi_\theta D\cL_\theta \Pi_\theta+\Pi_\theta D^2\cL_\theta \Pi_\theta+\Pi_\theta D\cL_\theta D\Pi_\theta \Pi_\theta-  D\rho_\theta\Pi_\theta D\Pi_\theta \Pi_\theta.
\]
Note that the last term satisfies $\Pi_\theta D\Pi_\theta \Pi_\theta=0$. Indeed,
by \labelcref{item:Pi_iii}
\[ 
	D\Pi_\theta=D\Pi_\theta \Pi_\theta+\Pi_\theta D\Pi_\theta, 
\] 
which implies the claim by composing with $\Pi_\theta$ on the right and using
\labelcref{item:Pi_iii} again. Hence,
\begin{equation}\label{eq:hell1}
	D^2\rho_\theta \Pi_\theta=\Pi_\theta D\Pi_\theta D\cL_\theta \Pi_\theta+\Pi_\theta D^2\cL_\theta \Pi_\theta+\Pi_\theta D\cL_\theta D\Pi_\theta \Pi_\theta.
\end{equation}
 
Next, we want to exploit the following properties of the operators $(\mathbf{1}-
\Pi_\theta)D\Pi_\theta$ and $D\Pi_\theta(\mathbf{1}- \Pi_\theta)$. We claim that  
\begin{equation}\label{sub:resolvent}
\begin{split}
(\mathbf{1}- \Pi_\theta)D\Pi_\theta &=\sum_{n\ge 0} \rho_\theta^{-n-1} \cL_\theta^n(\mathbf{1}-\Pi_\theta)D\cL_\theta \Pi_\theta ; \\
D\Pi_\theta(\mathbf{1}- \Pi_\theta) &=\sum_{n \ge 0} \rho_\theta^{-n-1} \Pi_\theta D\cL_\theta (\mathbf{1}- \Pi_\theta) \cL_\theta^n 
\end{split}
\end{equation}
where the expressions are well defined in the neighborhood $\mathcal{N}_0$. Let
us assume this claim for the moment, and carry on with the proof.

Recalling that $\Pi_\theta D\Pi_\theta \Pi_\theta=0$, we add the (zero) terms
$-\Pi_\theta D\Pi_\theta \Pi_\theta D\cL_\theta \Pi_\theta$ and $-\Pi_\theta
D\cL_\theta  \Pi_\theta D\Pi_\theta \Pi_\theta$ to \cref{eq:hell1}
and, by simple algebraic manipulation, we get
\[
D^2\rho_\theta \Pi_\theta=\Pi_\theta D\Pi_\theta (\mathbf{1}-\Pi_\theta) D\cL_\theta \Pi_\theta+\Pi_\theta D^2\cL_\theta \Pi_\theta+\Pi_\theta D\cL_\theta (\mathbf{1}-\Pi_\theta) D\Pi_\theta \Pi_\theta.
\]
Using \eqref{eq:scalar_formula}, the above equation yields
\begin{equation}\label{eq:hell2}
D^2\rho_\theta=\nu_\theta\left(D\Pi_\theta\left(\mathbf{1}-\Pi_\theta\right) D\mathcal{L}_\theta \ell_{+,\theta}+D^2\mathcal{L}_\theta \ell_{+,\theta}+D\mathcal{L}_\theta\left(\mathbf{1}-\Pi_\theta\right)D \Pi_\theta \ell_{+,\theta}\right) .
\end{equation}
We now evaluate \eqref{eq:hell2} in $\theta=0$. First, we note that, by \eqref{eq:DL1},
\[
	({\partial}_{jk}^2  \cL_\theta  f) |_{\theta=0}=-4\pi^2\cL_{0}(fF_{j}F_{k} ).
\]
It follows that
\[
	\nu_\theta(D^2\mathcal{L}_\theta \ell_{+,\theta})|_{\theta=0}= - 4\pi^2A_0,
\]
where the matrix is $A_0=\{\rho\nu_T(F_{j}F_{k})\}_{jk}$.

Finally, let us study one of the other terms in \cref{eq:hell2}, (the other one
is similar). By \cref{sub:resolvent},
\begin{equation}\label{eq:final}
	\nu_\theta (D\mathcal{L}_\theta\left(\mathbf{1}-\Pi_\theta\right)D \Pi_\theta \ell_{+,\theta})=\rho_\theta^{-1}\sum_{n \ge 0} \nu_\theta[D\cL_\theta  (\rho_\theta^{-1}\cL_\theta)^{n}(\mathbf{1}-\Pi_\theta)D\cL_\theta \Pi_\theta \ell_{+,\theta}].
\end{equation}
Let us analyze the term inside the square bracket at $\theta=0$. We have by
\eqref{eq:DL1.5}
\[
	D\cL_\theta \Pi_\theta \ell_{+,\theta} |_{\theta=0} =D\cL_\theta(\ell_{+,\theta}) |_{\theta=0}= 2\pi i\rho\sum_{k=1}^d F_k\circ T^{-1}\, \ell_{+}\, e_k.
\]
Therefore
\[
\big((1-\Pi_\theta)D\cL_\theta \Pi_\theta \ell_{+,\theta} \big)\big|_{\theta=0} = 2\pi i\rho \sum_{k=1}^d \overline{F_k} \circ T^{-1}\, \ell_{+}\, e_k,
\]
where $\overline{F_k} = F_k -\nu_{T}(F_k)$. As, for any $f$, $\rho^{-1} \cL_0
(f\ell_{+}) = f\circ T^{-1} \ell_{+}$,
\[
\big( (\rho_\theta^{-1} \cL_\theta)^n(1-\Pi_\theta)D\cL_\theta \Pi_\theta \ell_{+,\theta} \big) \big|_{\theta=0} =
2\pi i\rho \sum_{k=1}^d \overline{F_k} \circ T^{-(n+1)}\, \ell_{+}\, e_k.
\]
Finally, applying $D\cL_\theta |_{\theta=0}$, we see that the $(j,k)$ component
of the term in square brackets is equal to
\[-4\pi^2\rho^2 (F_j \cdot (\overline{F_k} \circ T^{-(n+1)})) \circ T^{-1} \, \ell_{+},\]
from which it follows that 
\[
 \nu_\theta \big(D\mathcal{L}_\theta\left(\mathbf{1}-\Pi_\theta\right)D \Pi_\theta \ell_{+,\theta} \big)\big|_{\theta=0} =-4\pi^2 A_1,
\]
where 
\[
A_1= \left\{\rho \sum_{n \ge 1} \nu_T(\overline{F}_j\circ T^n \cdot {F}_k)\right\}_{jk}.
\]

The computation for the remaining term in \eqref{eq:hell2} is completely
symmetrical and leads to
\[
\nu_\theta( D\Pi_\theta\left(\mathbf{1}-\Pi_\theta\right) D\mathcal{L}_\theta \ell_{+,\theta} )|_{\theta=0} =-4\pi^2 A_2= - 4\pi^2\left\{\rho \sum_{n \ge 1} \nu_T(\overline{F}_{k}\circ T^n \cdot {F}_{j})\right\}_{jk}.
\]
Putting all the above result together we get, starting from \eqref{eq:hell2},
\[
 D^2\rho_\theta|_{\theta=0} =-4\pi^2(A_0+A_1+A_2).
\]

We observe that, since $\nu_T(\overline{F}_k) = 0$, then
$\nu_T(\overline{F}_k\circ T^n \cdot F_j) = \nu_T(\overline{F}_k\circ T^n \cdot
\overline{F}_j)$, and, similarly, $\nu_T(F_j F_k) = \nu_T(\overline{F}_j
\overline{F}_k) + \nu_T(F_j)\nu_T(F_k)$, hence
\[
 D^2\rho_\theta|_{\theta=0} =-4\pi^2\rho \sum_{m\in \Z} \nu_T(\overline{F}\otimes\overline{F}\circ T^m) - 4\pi^2\rho\nu_T(F) \otimes \nu_T(F).
\]
Thus we conclude, using the formula that we obtained for $D\rho_\theta |_{\theta=0}$, that
\[
D^2 \log \rho_\theta|_{\theta=0} = - 4\pi^2\sum_{m\in \Z} \nu_T(\overline{F}\otimes\overline{F}\circ T^m) =: 4\pi^2\Sigma,
\]
where we have used the invariance of $\nu_T$ to obtain the sum over
$m\in\Z$. For the last statement of the \nameCref{prop:pert}, see for example
\cite{Gouezel04}. 

To conclude, we are left only with the proof of \cref{sub:resolvent}. We will
exploit the fact that, by \cref{item:pertB}, for the operator $\rho_\theta
Q_\theta$ the von Neumann formula holds in $\mathcal{N}_0$:
\begin{equation}\label{eq:Neum}
(\mathbf{1}-\rho_\theta^{-1} Q_\theta)^{-1}=\sum_{n\ge 0} (\rho_\theta^{-1} Q_\theta)^n.
\end{equation}
By $\Pi_\theta Q_\theta=Q_\theta \Pi_\theta=0$ and $\cL_\theta=\rho_\theta \Pi_\theta +Q_\theta$, we have
\begin{equation}
\begin{split}
\left(\mathbf{1}-\rho_\theta^{-1} Q_\theta\right)\left(\mathbf{1}-\Pi_\theta\right) D\Pi_\theta
&=(\mathbf{1}-\Pi_\theta-\rho_\theta^{-1}Q_\theta+\rho_\theta^{-1}Q_\theta \Pi_\theta)D\Pi_\theta\\
&=(\mathbf{1}-\rho_\theta^{-1} \cL_\theta)D\Pi_\theta\\
&=\rho_\theta^{-1}(\rho_\theta D\Pi_\theta-\cL_\theta D\Pi_\theta).
\end{split}
\end{equation}
By differentiating \labelcref{item:Pi_i},
\[
	\rho_\theta D\Pi_\theta=D\cL_\theta\Pi_\theta+\cL_\theta D\Pi_\theta-D\rho_\theta \Pi_\theta,
\]
and the previous equation thus becomes
\[
\begin{split}
\left(\mathbf{1}-\rho_\theta^{-1} Q_\theta\right)\left(\mathbf{1}-\Pi_\theta\right) D\Pi_\theta&=\rho_\theta^{-1}(D\cL_\theta \Pi_\theta-D\Pi_\theta)\\
&=\rho_\theta^{-1}(D\cL_\theta \Pi_\theta-\Pi_\theta D\cL_\theta \Pi_\theta)\\
&=\rho_\theta^{-1}(\mathbf{1}-\Pi_\theta)D\cL_\theta \Pi_\theta,
\end{split}
\]
where in the second step we have also used \eqref{eq:DLDP}. We have showed that
\[
\begin{split}
\left(\mathbf{1}-\rho_\theta^{-1} Q_\theta\right)\left(\mathbf{1}-\Pi_\theta\right) D\Pi_\theta  =\rho_\theta^{-1}(\mathbf{1}-\Pi_\theta)D\cL_\theta \Pi_\theta
\end{split}
\]
which, by \eqref{eq:Neum} implies
\[
\begin{split}
\left(\mathbf{1}-\Pi_\theta\right) D\Pi_\theta &=\rho_\theta^{-1}\sum_{n\ge 0}( \rho_\theta^{-1}Q_\theta)^n (\mathbf{1}-\Pi_\theta)D\cL_\theta \Pi_\theta\\
&=\rho_\theta^{-1}\sum_{n\ge 0}(\rho_\theta^{-n} \cL_\theta^n-\Pi_\theta^n ) (\mathbf{1}-\Pi_\theta)D\cL_\theta \Pi_\theta\\
&=\sum_{n\ge 0} \rho_\theta^{-n-1} \cL_\theta^n(\mathbf{1}-\Pi_\theta)D\cL_\theta \Pi_\theta,
\end{split}
\]
since $\Pi_\theta^n=\Pi_\theta$ and $\Pi_\theta(\mathbf{1}-\Pi_\theta)=0$. This concludes the proof of the first statement. The second statement is done in a similar way. $\qedhere$
\end{proof}

We conclude this section with the following theorem which gives the central
limit theorem for the observable $\overline{F}$. The proof follows an approach
used by Guivarc'h and Nagaev (see~\cite{Gou} for a description of the argument
applied to dynamical systems).

\begin{theorem}\label{thm:CLT}
The sequence of random variables
\[
\frac{S_nF-n\nu_T(F)}{\sqrt n}
\]
on the probability space $(S,\nu_T)$ 
converges in distribution to a multivariate centered Gaussian law whose characteristic
function is
\[
t\longmapsto \exp\!\bigl(2\pi^2\, \Sigma(t,t)\bigr),
\]
where $\Sigma$ is the quadratic form in \Cref{prop:pert}. 
Moreover, the limiting Gaussian is non-degenerate if and only if
$F\cdot v$ is not cohomologous to a constant  $\hspace{-0.5em}\pmod{2\pi}$ for every
$v\in (-\pi,\pi)^d\setminus\{0\}$.
\end{theorem}

\begin{proof}
As above let $\ell_+$ be the eigenfunction of $\cL_0$ associated with the eigenvalue
$\rho$, $\nu$ the corresponding eigenfunction of $\cL_0'$ normalized so that $\nu(\ell_+)=1$,  $\nu_T(\cdot) = \nu(\cdot\,\ell_+)$.
Then, by the $T$-invariance of $\nu_T$ and by definition
of the weighted transfer operators, for every $\theta\in\R^d$ and every $n\ge 1$,
\begin{equation}
\begin{split}
\label{eq:char-op}
\mathbb{E}_{\nu_T}\!\left(e^{2\pi i  \theta\cdot S_nF}\right)
&=
\nu\!\left(e^{2\pi i  \theta\cdot S_nF} \circ T^{-n} \ell_+\right) \\
&=
\rho^{-n} \nu\!\left( e^{2\pi i  \theta\cdot S_nF} \circ T^{-n} \cL_0^n(\ell_+)  \right)\\
&=
\rho^{-n}\,\nu\!\left(\cL_\theta^n \ell_+\right).
\end{split}
\end{equation}
Hence, for $t\in\R^d$,
\[
\varphi_n(t)
\coloneqq
\mathbb{E}_{\nu_T}\!\left(
e^{2\pi i t\cdot \frac{S_nF-n\nu_T(F)}{\sqrt n}}
\right)
=
e^{-2\pi i\sqrt n\, t\cdot \nu_T(F)}\,
\rho^{-n}\,
\nu\!\left(\cL_{t/\sqrt n}^n \ell_+\right).
\]

Fix $t\in\R^d$. For $n$ large enough, $\theta_n\coloneqq t/\sqrt n$
belongs to the neighborhood $\mathcal N_0$ of \Cref{prop:pert},
so by the spectral decomposition
\[
\cL_{\theta_n}^n
=
\rho_{\theta_n}^n\Pi_{\theta_n}+Q_{\theta_n}^n,
\qquad
\|Q_{\theta_n}^n \ell_+\|_{p,q}\le C\sigma^n\|\ell_+\|_{p,q}.
\]
Applying $\nu$ gives
\[
\nu(\cL_{\theta_n}^n \ell_+)
=
\rho_{\theta_n}^n\,\nu(\Pi_{\theta_n}\ell_+)
+
\nu(Q_{\theta_n}^n \ell_+),
\]
where the last term is
$O(\sigma^n)$. Therefore
\begin{equation}\label{eq:phi-dec}
\varphi_n(t)
=
e^{-2\pi i\sqrt n\, t\cdot \nu_T(F)}
\left(\frac{\rho_{\theta_n}}{\rho}\right)^n
\nu(\Pi_{\theta_n}\ell_+)
+
O(\sigma^n).
\end{equation}

We now analyze the three factors in \eqref{eq:phi-dec}.
By analyticity 
\[
\log \rho_\theta
=
\log \rho
+
D\log\rho_\theta|_{\theta=0}\cdot \theta
+
\frac12\,D^2\log\rho_\theta|_{\theta=0}(\theta,\theta)
+
o(|\theta|^2).
\]
Using \eqref{eq:logrho} and \eqref{eq:taylor},
\[
D\log\rho_\theta|_{\theta=0}
=
\frac{D\rho_\theta|_{\theta=0}}{\rho}
=
2\pi i\,\nu_T(F),
\]
because $\Re(D\rho_\theta|_{\theta=0})=0$ and
$\Im(D\rho_\theta|_{\theta=0})=2\pi\rho\,\nu_T(F)$.
Also,
\[
\frac12\,D^2\log\rho_\theta|_{\theta=0}(\theta,\theta)
=
2\pi^2\,\Sigma(\theta,\theta).
\]
Hence
\begin{equation}\label{eq:log-exp}
\log\frac{\rho_\theta}{\rho}
=
2\pi i\, \nu_T(F)\cdot \theta
+
2\pi^2\,\Sigma(\theta,\theta)
+
o(|\theta|^2).
\end{equation}

Substituting $\theta=\theta_n=t/\sqrt n$, we obtain
\[
n\log\frac{\rho_{\theta_n}}{\rho}
=
2\pi i\,\sqrt n\, \nu_T(F)\cdot t
+
2\pi^2\,\Sigma(t,t)
+
o(1).
\]
Therefore
\begin{equation}\label{eq:main-exp}
\lim_{n\to \infty }e^{-2\pi i\sqrt n\, t\cdot \nu_T(F)}
\left(\frac{\rho_{\theta_n}}{\rho}\right)^n
=
\exp\!\bigl(2\pi^2\,\Sigma(t,t)\bigr).
\end{equation}

Next, since $\theta\mapsto \Pi_\theta$ is analytic and $\Pi_0 \ell_+=\ell_+$, we have
$
\lim_{n\to \infty}\nu(\Pi_{\theta_n}\ell_+)= \nu(\ell_+)=1
$
and, combining this with \eqref{eq:phi-dec} and \eqref{eq:main-exp},
\[
\lim_{n\to \infty}\varphi_n(t)= \exp\!\bigl(2\pi^2\,\Sigma(t,t)\bigr)
\qquad\text{for every }t\in\R^d.
\]
This limit is continuous at $t=0$, so Lévy's continuity theorem ensures that
$
\frac{S_nF-n\nu_T(F)}{\sqrt n}
$ converges in distribution to a Gaussian with mean $0$ and covariance 
$-\Sigma$.

Finally, by the last part of Proposition ~\ref{prop:pert},  the non-degeneracy of
$\Sigma$ is coded by the condition that
$F\cdot v$ is not cohomologous to a constant  $\hspace{-0.5em}\pmod{2\pi}$ for any
$v\in (-\pi,\pi)^d\setminus\{0\}$.
\end{proof}

\begin{remark}
Proposition ~\ref{prop:pert}  can also be used to prove finer statistical properties, such as the Local Limit Theorem or even a Large Deviations Principle. It would be interesting to discuss the implications of these results in the infinite volume case. These extensions are left for future work.   
\end{remark}


\section{Fourier analysis on Abelian covers}\label{sec:twisted_hilbert_spaces}

In this section, we briefly recall the definition of Abelian covers of the translation surface $S$ and we introduce the Fourier decomposition and the Frobenius function, which provide the connection with the theory of twisted transfer operators we developed so far. 

\subsection{Abelian covers}\label{sec:abelian_covers}

Recall that $S_0 = S \setminus \Sigma$, where $\Sigma$ denotes the finite set of singularities of $S$ (namely, the zeros of the holomorphic 1-form $\alpha$ defining the translation structure). 
Let $G = \pi_1(S_0)$ be the fundamental group of the punctured surface $S_0 $ and let $G' = [G,G]$ denote the derived subgroup. 
The quotient $G^{\abel} = G/G'$ is isomorphic to the first homology $H_1(S_0,\Z)$ of $S_0$, which is a free abelian group of rank $2g+\# \Sigma-1$, where $g \geq 2$ is the genus of $S$ and $\# \Sigma$ is the cardinality of $\Sigma$.

Intermediate subgroups $G' \leq \Gamma \leq G$ are in 1-to-1 correspondence with subgroups of $H_1(S_0,\Z)$ via the projection $\Gamma \mapsto \Gamma^{\abel, G} = \Gamma/G'$. Each $\Gamma$ defines a cover $\mathcal{p}_0 \colon \tSz \to S_0$ with a Galois group
\[
\Deck \coloneqq \Aut(\tSz / S_0)
\]
of deck transformations isomorphic to $G/\Gamma = H_1(S_0,\Z) /  \Gamma^{\abel, G}$. We will assume that the latter has no torsion, which is always the case up to a finite cover, and hence is a free abelian group of rank $1\leq d \leq 2g+\# \Sigma-1$.
We fix $d$ linearly independent primitive homology classes
\[
[\gamma_i] =  \gamma_i \, G' \in H_1(S_0 ,\Z), \qquad \text{for $i=1,\dots , d$}
\]
so that the elements
\[
D_i =  \gamma_i \, \Gamma \in \Deck
\]
form a basis of $\Deck \simeq \Z^d$.
We say that the associated cover $\mathcal{p}_0 \colon \tSz \to S_0$ is a $\Z^d$-cover of $S_0$.
The cover $\tSz$ is equipped with the pullback flat metric under $\mathcal{p}_0$, so that $\mathcal{p}_0 \colon \tSz \to S_0$ is a Riemannian cover, and the deck transformations act isometrically.

We equip the cover $\tSz$ with the pullback Riemannian area form under $\mathcal{p}_0$. By a little abuse of notation, we still denote by $\vol$ the area form induced by $\alpha$ on $\tSz$. It induces an infinite measure, normalized so that $\vol(S_0)=1$. We also denote by $X$ and $Y$ the horizontal and vertical vector fields on $\tSz$.

We assume that the cover $\mathcal{p}_0 \colon \tSz \to S_0$ is \emph{locally compact} in the following sense: let $\gamma$ be any small loop going around a puncture $x\in \Sigma$; then, we say that the cover is locally compact if the homology class $[\gamma]$ of $\gamma$ is in $\Gamma$. This implies that there exists a locally compact surface $\tS$ and a cover $\mathcal{p} \colon \tS \to S$ so that $\tSz \subset \tS$ and the restriction of $\mathcal{p}$ to $\tSz$ coincides with $\mathcal{p}_0$. Moreover, for each $x\in \widetilde{\Sigma} \coloneqq \mathcal{p}^{-1}\Sigma$, the restriction of $\mathcal{p}$ to a neighborhood of $x$ is 1-to-1. 

\begin{remark}
	More generally, we could relax the above condition by asking that there exists $l\in \N$ so that $l\cdot [\gamma] \in \Gamma$. This would imply that $\mathcal{p}\colon \tS \to S$ is a branched cover at the points $x\in \widetilde{\Sigma}$. The group $\Deck$ would then have torsion elements. 
	
	In any case, up to going to a finite cover (corresponding precisely to the subgroup $\Gamma'$ generated by the torsion elements), we could always reduce to the case where the locally compact assumption above is satisfied.
\end{remark}

Let us fix a compact connected subset $\cF \subset \tSz$ whose boundary has zero measure such that the restriction of $\mathcal{p}_0$ to the interior of $\cF$ is injective and $\mathcal{p}_0(\cF) = S_0$. We say that $\cF$ is a fundamental domain for the cover $\mathcal{p}$.
Any function on $S$ can be seen as a $\Deck$-invariant function on $\tS$, and vice-versa. Under this identification, it is not hard to see that 
\[
\int_S f \diff\!\vol = \int_{\cF} f \diff\!\vol.
\]
In particular, for all $q\geq 1$, the Banach spaces $L^q(S)$ and $L^q(\cF)$ are naturally isomorphic.

\subsection{A Fourier decomposition}

For any $i=1,\dots, d$, let $\omega_i \in H^1(S,\R)$ be the cohomology class defined by
\[
\omega_i(\Gamma)=0, \qquad \text{and} \qquad \omega_i([\gamma_j]) = \delta_{i,j}.
\]
By Hodge Theory, we can identify each $\omega_i$ with a harmonic 1-form on $S$. Let $\omega \coloneqq (\omega_1,\dots,\omega_d)$ be the vector of 1-forms.


For any $z \in  \R^d \oplus 2\pi i \left(-\frac12,\frac12\right)^d$, where $z=r+2\pi i \theta$ with $r \in \R^d$ and $\theta \in \left(-\frac12,\frac12\right)^d$, we define
\[
E_{z} \colon \Deck \to \C, \qquad E_{z}([\gamma] + \Gamma^{\abel, G}) = e^{z \cdot \omega ([\gamma])}.
\] 
Given $z \in \R^d \oplus 2\pi i \left(-\frac12,\frac12\right)^d$, we define
\[
\mathscr{C}^\ell(S,z) \coloneqq \left\{ f \in \mathscr{C}^\ell(\tS) \mid  f \circ D^{-1} = E_{z}(D) \, f \text{\ for all $D \in \Deck$ }\right\}. 
\]
Moreover, for any continuous function $f \in \mathscr{C}^0_c(\tS)$ with compact support, let us define
\[
\pi_{z}(f)(x)= \sum_{D \in \Deck} E_{z}(D) \cdot f\circ D(x).
\]
Note that, since $f$ has compact support, the sum on the right hand side above is finite for any $x \in \tS$.

\begin{lemma}\label{lem:proj_cl}
	For every $\ell \geq 0$ and $z=r+2\pi i \theta$, 
	\[
	\pi_{z} \colon \mathscr{C}^\ell_c(\tS) \to \mathscr{C}^\ell(S,z).
	\]
	Moreover, for any $r \in \R^d$, for every $f \in \mathscr{C}^\ell_c(\tS)$ and for any $x \in \tS$, we have
	\[
	f(x) = \int_{\left(-\frac12,\frac12\right)^d} \pi_{r+2\pi i \theta}(f)(x) \diff \theta.
	\]
\end{lemma}
\begin{proof}
	Fix $z=r+2\pi i \theta \in \R^d \oplus 2\pi i \left(-\frac12,\frac12\right)^d$. Let us show that $\pi_{z}(f) \circ D_0^{-1} = E_{z}(D_0) \cdot \pi_{z}(f)$.
	For every $D_0 = [\gamma_0] + \Gamma^{\abel, G} \in \Deck$, we have $E_{z}(D_0^{-1}) = E_{z}( -[\gamma_0] + \Gamma^{\abel, G}) = E_{-z}(D_0) = E_{z}(D_0)^{-1}$, so that 
	\[
	\begin{split}
		\pi_{z}(f) \circ D_0^{-1} &= \sum_{D \in \Deck} E_{z}(D) \cdot f\circ (D_0^{-1} \cdot D ) \\
		&= \sum_{D \in \Deck} E_{z}(D_0^{-1} \cdot D) \cdot  E_{z}(D_0) \cdot f\circ (D_0^{-1} \cdot D ) =E_{z}(D_0) \cdot \pi_{z}(f).
	\end{split}
	\]
	By definition, it follows that $\pi_{z}(f)$ is a $\mathscr{C}^\ell$-function whenever $f \in \mathscr{C}^\ell_c(\tS)$.
	For the last claim, we note that
	\[
	\int_{\left(-\frac12,\frac12\right)^d} e^{2 \pi \imath \theta \cdot \omega([\gamma])} \diff \theta = 0 \qquad \text{if and only if} \qquad D = [\gamma] + \Gamma^{\abel, G} \neq 0,
	\]
	and is equal to 1 otherwise. Therefore,
	\[
	\int_{\left(-\frac12,\frac12\right)^d} \pi_{z}(f)(x) \diff \theta = \sum_{D \in \Deck}  f\circ D(x)e^{r \cdot \omega([\gamma])} \int_{\left(-\frac12,\frac12\right)^d} e^{2 \pi \imath \theta \cdot \omega([\gamma])} \diff \theta  = f(x),
	\]
	which completes the proof.
\end{proof}
Henceforth, we will simply write $f_{z}$ in place of $\pi_{z}(f)$.

Note that the pullback $\mathcal{p}_0^{\ast}\omega_i$ of $\omega_i$ on $\tS$ is an exact 1-form.
Fix $x_0 \in \cF$. For any $x \in \tS$, the vector
\begin{equation}\label{eq:roof}
	\xi(x)\coloneqq \left(\int_{x_0}^x \mathcal{p}^\ast \omega_1,\dots,\int_{x_0}^x \mathcal{p}^\ast \omega_d\right) \in \R^d
\end{equation}
is well defined, since each integral does not depend on the choice of path connecting $x_0$ to $x$. 
For any measurable function $f$ on $\tS$, we define 
\begin{equation}\label{eq:Xi}
	\Xi_{z}(f) = f \cdot e^{z\cdot \xi}.
\end{equation}

\begin{lemma}\label{lem:proj_Xi}
	For every $z=r+2\pi i \theta\in \R^d \oplus 2\pi i \left(-\frac12,\frac12\right)^d$, we have
	\[
	\pi_{z} = \Xi_{-z} \circ \pi_{0} \circ \Xi_{z}.
	\]
\end{lemma}
\begin{proof}
	Note that $D^\ast \mathcal{p}^\ast \omega = \mathcal{p}^\ast \omega$ for every deck transformation $D$, since $\mathcal{p} \circ D = \mathcal{p}$. Thus, the $i$-th component $\xi(D(x))_i$ of the vector $\xi(D(x))$ satisfies
	\[
	\xi(D(x))_i = \int_{x_0}^{D(x)} \mathcal{p}^\ast \omega_i = \int_{x_0}^{D(x_0)} \mathcal{p}^\ast \omega_i + \int_{D(x_0)}^{D(x)} \mathcal{p}^\ast \omega_i = \int_{x_0}^{D(x_0)} \mathcal{p}^\ast \omega_i + \int_{x_0}^{x} \mathcal{p}^\ast \omega_i.
	\]
	Hence, it follows that $\xi(D(x))= \omega([\gamma])+\xi(x)$. Therefore,
	\[
	e^{z \cdot \xi(D(x))}=E_{z}(D)e^{z \cdot \xi(x)}.
	\]
	From this, we conclude
	\[
	\pi_{0} \circ \Xi_{z}(f) = \sum_{D\in \Deck} ( f \cdot e^{z \cdot \xi}) \circ D =  \sum_{D\in \Deck} f \circ D \cdot E_{z}(D) \cdot e^{z \cdot \xi}= \Xi_{z}\circ \pi_{z}(f),
	\]
	which proves the result.
\end{proof}

\begin{lemma}\label{lem:isom_L2omega}
	Let $z \in \R^d \oplus 2\pi i \left(-\frac12,\frac12\right)^d$. For every $z' \in \R^d \oplus 2\pi i \left(-\frac12,\frac12\right)^d $, the map $\Xi_{z}$ is a linear isomorphism between $\mathscr{C}^\ell(S,z+z')$ and $\mathscr{C}^\ell(S,z')$ for every $\ell \geq 0$.
\end{lemma}
\begin{proof}
	We note that, for each $D\in \Deck$, we have
	\[
	\begin{split}
		E_{z'}(D)\Xi_{z}(f)&=E_{-z}(D)e^{z \cdot \xi}E_{z+z'}(D)f=E_{-z}(D)e^{z\cdot \xi}f\circ D^{-1}=e^{z\cdot \xi\circ D^{-1}}f\circ D^{-1}\\
		&=(\Xi_{z}f)\circ D^{-1}.
	\end{split}
	\]
	Since $e^{z\cdot \xi}$ is a smooth map on $\tSz$, if $f\in \mathscr{C}^\ell(\tS)$, also $\Xi_{z}(f) \in \mathscr{C}^\ell(\tS)$.
	
\end{proof}


\subsection{Pseudo-Anosov map and twisted transfer operators}\label{sec:twisted_transfer_operators}
We now assume that we can lift the linear pseudo-Anosov $T\colon S \to S$ to a linear pseudo-Anosov homeomorphism $\tT\colon \tS\to \tS$ which commutes with the deck transformations. 

In this setting, if $\{\tphi_s\}_{s\in \R}$ denotes the straight-line flow on $\tS$ in the contracted direction of $\tT$, we know:

\begin{theorem}[See~\cite{Tum} {or~\cite[Theorem 2.10]{BFRT}}]
\label{thm:flow_ergodic}
The flow $\{\tphi_s\}_{s\in \R}$ on $\tS$ is ergodic with respect to Lebesgue measure.
\end{theorem}

As $\tT$ commutes with deck transformations, the Koopman operator $U \colon f \mapsto f \circ \tT^{-1}$ is well-defined as an operator
\[
U \colon \mathscr{C}^\ell(S,z) \to \mathscr{C}^\ell(S,z).
\]
By \Cref{lem:isom_L2omega}, we are interested in studying
\[
\Xi_{z} \circ  U\circ \Xi_{-z} \colon \mathscr{C}^\ell(S) \to \mathscr{C}^\ell(S),
\]
under the identification $\mathscr{C}^\ell(S,0) = \mathscr{C}^\ell(S)$.

We define the function $F=F_T \colon \tS \to \R^d$ as
\begin{equation}\label{eq:F}
	F(x)=\xi(\tT(x))-\xi(x), 
\end{equation}
known as the \textit{Frobenius function}. From the proof of \Cref{lem:proj_Xi}, it is easy to see that $F$ is $\Deck$-invariant, hence can be identified with a $\R^d$-valued function on $S$.

\begin{lemma}\label{lem:invar}
	The Frobenius function $F$ belongs to $\mathscr{C}^{\infty, \bdd}(S)$. Furthermore, 
	for each $z \in \R^d \oplus 2\pi i \left(-\frac12,\frac12\right)^d$, we have 
	\begin{equation}\label{eq:twisted}
		\Xi_{z} \circ  U\circ \Xi_{-z} (f) =\cL_{z,F}(f),
	\end{equation}
	where $\cL_{z,F}$ is the twisted transfer operator as defined in \S\ref{sec:weighted_transfer_operators}.
\end{lemma}
\begin{proof}
	First of all, recalling \eqref{eq:Xi}, we notice that
	\[
	\begin{split}
		{\left(\Xi_{z} \circ U \circ \Xi_{-z} \right)(f) } & =\Xi_{z} (e^{-z\cdot \xi\circ \tT^{-1}}f\circ T^{-1})=(e^{z\cdot F} f)\circ T^{-1}=\cL_{z,F}(f).
	\end{split}
	\]
	We need to show that $F$ is smooth. The  $j$-th component $F_j$ can be written as
	\[
	F_j(x) = \int_{x}^{T x} \omega_j, \qquad \text{for $x\in S$.}
	\]
	Since the 1-form $\omega_j$ is the real part of a holomorphic 1-form on $S$, it is smooth, and hence so is $F_j$.
\end{proof}

We conclude this section by showing that the Frobenius function $F$ satisfies the assumption of \Cref{prop:unique_max_eigenvalue}.

\begin{proposition}
	For any $v\in (-\pi,\pi)^d$, the function $v\cdot F$ is cohomologous to a constant  $\hspace{-0.5em}\pmod{2\pi}$ if and only if $v =0$.
\end{proposition}
\begin{proof}
Assume that there exists $v\in (-\pi,\pi)^d \setminus \{0\}$ such that $v\cdot F$ is cohomologous to a constant $c$ $\hspace{-0.5em}\pmod{2\pi}$; namely, there exists a measurable function $\alpha$ such that $ v\cdot F - c - \alpha\circ T+\alpha \in 2\pi \Z$ for Lebesgue almost every $x$. 

By Lusin's theorem, as $v\cdot F$ is defined on the locally compact space $S_{\reg}$, there exists a compact set $K \subset S_{\reg}$ of positive Lebesgue measure, such that $\alpha$ is continuous (and moreover uniformly continuous) on $K$.

As $K$ has positive Lebesgue measure, we can take a point $x\in K$ that is generic for Lebesgue under $T$, and has a preimage $\tx \in \mathcal{p}^{-1}(x)$ that is generic for Lebesgue under the horizontal flow (using the ergodicity of $\{\tphi_t\}_{t\in \R}$). Then for any $n \in \Z^d$, and any $\epsilon > 0$, there exists a $t>0$ such that $\tphi_t(\tx)$ is close enough to $D_n(\tx)$ and $N \in \N$ is large enough that the following hold:
\begin{enumerate}
\item $|\alpha(\phi_t(x)) - \alpha(x)| < \epsilon$,
\item $|\xi(\tphi_t(\tx)) - \xi(\tx) - n| < \epsilon$,
\item The segment $\phi_{[0,\lambda^{-N}t]}(T^Nx)$ lies inside $K$ and $|\alpha(T^N x) - \alpha(\phi_{\lambda^{-N}t} T^N x)| < \epsilon$,
\item $|\xi(\tT^N(\tx)) - \xi(\tphi_{\lambda^{-N}t} \tT^N \tx)| < \epsilon$.
\end{enumerate}

It follows that 
\begin{align*}
v \cdot (\xi(\tphi_t(\tx)) - \xi(\tx)) & = v \cdot (\xi(\tphi_t(\tx)) - \xi(\tT^N \tphi_t(\tx)) + \xi(\tT^N \tx) - \xi(\tx)) +O(\epsilon)\\
&= v \cdot S_N F(x) - v\cdot S_N F(\phi_t(x)) +O(\epsilon)\\
&= O(\epsilon) + 2\pi\Z.
\end{align*}

Hence we see that $v\cdot n = O(\epsilon)+2\pi \Z$. As this holds for any $n$ and any $\epsilon$, $v=0$.

\end{proof}


\section{Maharam measures, Maharam distributions, and a description of the spectra}\label{sec:Maharam_stuff}

In this section we consider the operators $\cL_{z,F}$, where the function $F$ is defined in \eqref{eq:F}. Note that $F$ is Deck-invariant, and hence projects to a function $F: S\to \R^d$.

First we show that for any real parameter $r\in \R^d$, the eigenfunctions $\upsilon$ of $\cL'_{r,F}$  with eigenvalues of modulus greater than $\lambda^{-1}\rho(r)$ can be used to construct $\phi_s$-invariant Maharam distributions on $\tS$. 
In particular, starting from the eigenfunction $\nu_r$ one obtains a Maharam measure. This gives a $\R^d$-parametrised family of invariant measures for the horizontal flow on $\tS$.

Second, we study the spectrum of $\cL_{z,F}$ on $\cB_{p,q}$, and show that the spectrum is closely connected with the spectrum of an associated linear map, namely the induced action of $\tT$ on a twisted cohomology group on $S_0$, where the twisting is determined by $z$. We will denote this induced action by $\cL_z^\#$.

In particular, denoting by $\Theta_0(z)$ the spectrum of $\cL_z^\#$ acting on $H^1_z(S_0,\C)$, and by $\Theta(z)$ the spectrum of the action restricted to $H^1_z(S,\C)$, we see that the eigenvalues of $\cL_{z,F}$ on $\cB_{p,q}$ that are of modulus greater than $\lambda^{-p}\rho(z)$, are contained in the set $\{\lambda^{-n}\mu: 1 \leq n \leq p, \mu \in \Theta_0(z)\}$. 
Further we recover the eigenvalues $\mu \in \Theta(z)$ that satisfy $|\mu| > \lambda^{-\min\{p,q\}+1}\rho(z)$ as having corresponding eigenvalues in the spectrum of $\cL_{z,F}$.

As a corollary we identify the spectral radius $\rho(z)$ of $\cL_{z,F}$ as being $\lambda^{-1}$ times the spectral radius of $\cL_z^\#$.

\subsection{Twisted differential operators and twisted cohomology.}
A key tool in this section will be the following two twisted differential operators $X_z,Y_z$, which act on the eigenspaces of $\cL_{z,F}$.

\begin{definition}
The differential operator $X_z: \cB_{p,q} \to \cB_{p,q-1}$ is defined as:
\[X_z \ell = X(\ell) - z\cdot \omega(X)  \ell .\]
The differential operator $Y_z: \cB_{p,q} \to \cB_{p-1,q}$ is defined as:
\[Y_z \ell = Y(\ell) - z\cdot \omega(Y) \ell .\]
Here $\omega = (\omega_1,\dots,\omega_d)$ is the vector of 1-forms on $S$ corresponding to the cover $\tS$ (see \cref{sec:twisted_hilbert_spaces}).
\end{definition}

\begin{lemma}\label{lem:X_z commutation}
The following identities hold:
\begin{align*}
X_z \circ \cL_{z,F} &= \lambda \cL_{z,F} \circ X_z\\
Y_z \circ \cL_{z,F} &= \lambda^{-1} \cL_{z,F} \circ Y_z.
\end{align*}
\end{lemma}
\begin{proof}
We prove the first statement, the second follows by an analogous computation.
First, observe that for $x\in S$, if $\tx $ is a lift of $x$ to $\tS$, then
\[X(F\circ T^{-1})(x) = X\big(\xi(\tx)-\xi(\tT^{-1}\tx)\big) = \omega(X)(x) - \lambda\omega(X)(T^{-1}x),\]
where we use the fact that $X(\xi)$ is a Deck-invariant function on $\tS$, that projects to $\omega(X)$ on $S$.

Hence,
\begin{align*}
X_z (\cL_{z,F} f) &= X(\cL_{z,F}f) - z\cdot \omega(X)\cL_{z,F}(f)\\
&= X(e^{z\cdot F\circ T^{-1}}\cdot f\circ T^{-1}) - z\cdot \omega(X) \cL_{z,F}(f)\\
&= X(e^{z\cdot F\circ T^{-1}})\cdot f\circ T^{-1} +e^{z\cdot F\circ T^{-1}}\cdot X(f\circ T^{-1}) - z\cdot \omega(X) \cL_{z,F}(f) \\
&= z \cdot  \big(\omega(X)-\lambda \omega(X)\circ T^{-1} \big) \, \cL_{z,F}(f) + e^{z\cdot F\circ T^{-1}} \lambda X(f)\circ T^{-1} - z\cdot \omega(X)\cL_{z,F}(f)\\
&= \lambda e^{z\cdot F\circ T^{-1}} \big( X(f) \circ T^{-1} - z \cdot \omega(X)\circ T^{-1}  f\circ T^{-1} \big)\\
&= \lambda  \cL_{z,F} \big(X(f)- z\cdot  \omega(X)f \big) \\
&= \lambda  \cL_{z,F}(X_z f).\qedhere
\end{align*}
\end{proof}

As in~\cite{Forni:twisted}, we also define the twisted differential as
\[
\diff_z = \diff - z \cdot \omega \wedge.
\]
Note that, given a function $f$, if we consider its restriction to $S_0$, then we have $\diff_z f = X_zf \diff x + Y_z f \diff y$.  
It is easy to see that $\diff_z^2 = 0$ since the forms $\omega_i$ are closed. In particular, $\diff_z$ defines a  \emph{twisted cohomology complex} $H_z(S_0,\C)$ for the punctured surface $S_0$, as well as $H_z(S,\C)$ for the closed surface $S$.

\begin{lemma}
For $z \notin 2\pi i \Z^d$, we have $H^1_z(S,\C) \cong \C^{2g-2}$ and $H^1_z(S_0,\C) \cong \C^{2g-2+|\Sigma|}$.
\end{lemma}
\begin{proof}
The result for $H^1_z(S,\C)$ is the content of \cite[Lemma 4.3]{Forni:twisted}.
For $S_0$, we use the observation from \cite{Forni:twisted} that the twisted cohomology $H^*_z(S_0,\C)$ is isomorphic to the cohomology with local coefficients $H^*(S_0,\mathscr{L}_z)$, where $\mathscr{L}_z$ is the local system defined by the representation $\varrho_z: \pi_1(S_0)~\to~\C^\ast$ given by 
\[\varrho_z(\gamma) = e^{\int_\gamma z\cdot \omega }.\]

By \cite[Lemma 1.1]{Watanabe}, $\dim H^1(S_0,\mathscr{L}_z) = 2g-2+|\Sigma|$, and hence $H^1_z(S_0,\C) \cong \C^{2g-2+|\Sigma|}$.
\end{proof}

We can identify the twisted classes on $S_0$ with regular cohomology classes on $\tS_0$, which 
transform in a Maharam way under the Deck group action. 

\begin{definition}
Define the $z$-Maharam cohomology group $H^1(\tS_0, z, \C)$ as
\[H^1(\tS_0, z, \C) = \{h \in H^1(\tS_0, \C): h\circ D = e^{z\cdot \omega( [\gamma])}h \text{ for } D = [\gamma] +\Gamma^{\abel, G} \in \Deck\}.\]
\end{definition}

\begin{lemma}
There is an isomorphism between $H^1_z(S_0,\C)$ and $H^1(\tS_0, -z,\C)$.
\end{lemma}
\begin{proof}
For $[\alpha] \in H^1_z(S_0,\C)$, define $\tilde \alpha$ on $\tS_0$ by 
\[\tilde \alpha = e^{-z\cdot \xi} \mathcal{p}^*\alpha.\]
Then, observing that $\mathcal{p}^*\omega = \diff\xi$, we see that
\[\diff \tilde \alpha = e^{-z\cdot\xi}(\mathcal{p}^*\diff \alpha - z\cdot\diff\xi \wedge \mathcal{p}^*\alpha) = e^{-z\cdot \xi}\mathcal{p}^*(\diff{\alpha} - z\cdot \omega \wedge \alpha) = e^{-z\cdot \xi}\mathcal{p}^*(\diff_z \alpha) =0,\]
so $\tilde \alpha$ is closed and does indeed define a cohomology class. 
For $D = [\gamma]+\Gamma^{\abel, G} \in \Deck$,
\[\tilde \alpha \circ D= e^{-z\cdot \xi}e^{-z\cdot (\xi\circ D -\xi)} \mathcal{p}^*\alpha = e^{-z\cdot \omega ([\gamma])} \tilde \alpha,\]
thus $\tilde \alpha$ is a closed $-z$-Maharam 1-form.

The map is a well-defined map on cohomology, since if $\alpha = \diff_z \beta$ for $\beta \in \Omega^0(S_0)$, then 
$\tilde \alpha = \diff (e^{-z\cdot \xi}\mathcal{p}^*\beta)$. 
Thus the map $\alpha \mapsto \tilde \alpha$ is a homomorphism.

One similarly defines the inverse map, as for $(-z)$-Maharam $\tilde \alpha$, the 1-form $e^{z\cdot \xi}\tilde \alpha$ is $\Deck$-invariant and so projects down to 
a 1-form $\alpha$ on $S_0$. One checks as above that this is also a well-defined homomorphism, proving the isomorphism.
\end{proof}

We will come back to the twisted cohomology in \Cref{sec:description_spectrum}.

\subsection{Maharam distributions from eigenfunctions of the dual transfer operator}
In this subsection we consider real parameters $r\in \R^d$.
\begin{definition}
\label{defn:maharam}
For $r\in \R^d$, a distribution $\zeta$ on $\tS$ is said to be \emph{$e^{r\cdot \omega}$-Maharam} if  
\[\zeta \circ D = e^{r\cdot \omega([\gamma])}\zeta\]
for $D = [\gamma] + \Gamma^{\abel, G} \in \Deck$.
\end{definition}

We will show that $e^{r\cdot \omega}$-Maharam $\phi_s$-invariant distributions can be constructed starting from eigenfunctions of the dual transfer operator $\cL'_{r,F}$.

In particular, starting with the eigenfunction $\nu_r$ of maximal eigenvalue $\rho(r)$, which is a measure, we obtain a Maharam measure.\\

Let $\upsilon$ be an eigenfunction of $\cL'_{r,F}$.
To define the corresponding distribution on $\tS$ we use the Fourier decomposition developed in \Cref{sec:twisted_hilbert_spaces}. 
Given $g\in \mathscr{C}_c^{\infty,\bdd}(\tS)$, the function $\pi_0 \circ \Xi_{r} g$ is a Deck-invariant function on $\tS$. Hence we can identify it with its projection to $S$, and define
\begin{equation}\label{eq:Maharam_distribution}
\zeta(g) = \upsilon(\pi_0 \circ \Xi_{r} g).
\end{equation}

We will show that $\zeta$ is $\phi_s$-invariant, provided that its associated eigenvalue has modulus greater than $\lambda^{-1}\rho(r)$. We use the following lemma.
\begin{lemma}\label{lem:Xi_z}
For any $z \in \C^d, g \in \mathscr{C}_c^{\infty,\bdd}(\tS)$,
\[\pi_0 \circ \Xi_{z} (Xg) = X_z(\pi_0 \circ \Xi_{z} g).\]
\end{lemma}
\begin{proof}
Observe that
\[\Xi_z(Xg) = e^{z\cdot \xi} Xg = X\big(  e^{z\cdot\xi} g\big) - z\cdot \omega(X)  e^{z\cdot \xi}g = X_z(\Xi_z g).\]
As $\pi_0$ commutes with $X_z$, we get the desired relation.
\end{proof}

From \Cref{lem:Xi_z} it follows that to show that $\zeta$ is $\phi_s$ invariant, it suffices to show that $\upsilon(X_r f) = 0$ for all $f\in \mathscr{C}^{\infty, \bdd}(S)$.

\begin{proposition}\label{prop:Maharam_distribution}
Let $r\in \R^d$ and let $\upsilon$ be an eigenfunction of $\cL'_{r,F}$ with eigenvalue $\eta$,  $|\eta| > \lambda^{-1}\rho(r)$. Define $\zeta$ on $\tS$ as in \eqref{eq:Maharam_distribution}. Then $\zeta$ is Maharam and invariant under the horizontal flow on $\tS$.
\end{proposition}
\begin{proof}
Suppose $\upsilon \in \cB'_{p,q}$ satisfies $\cL'_{r,F} \upsilon = \eta \upsilon$.
We define the operator $X'_r : \cB'_{p,q-1} \to \cB'_{p,q}$ to be the dual of $X_r$, namely
such that for $f\in \mathscr{C}^{\infty,\bdd}(S)$, 
\[\langle X'_r \upsilon, f \rangle \coloneqq \langle \upsilon, -X_r f \rangle.\]

Then we compute that $X'_r \upsilon$ is also an eigenfunction of $\cL'_{r,F}$, with eigenvalue $\lambda \eta$. Indeed, 
using \Cref{lem:X_z commutation}, we see that for any smooth function $f$,
\[
\langle \cL'_{r,F}\, X'_r \upsilon, f \rangle = \langle \upsilon, -X_r\, \cL_{r,F} f \rangle
= \langle \upsilon, -\lambda\, \cL_{r,F}\, X_r\, f \rangle
= \lambda \langle X'_r\, \cL'_{r,F} \upsilon, f \rangle
=\lambda \eta \langle X'_r \upsilon,f \rangle.
\]
Hence by the density of $\mathscr{C}^{\infty,\bdd}(S)$, we see that $X'_r\upsilon$ is an eigenfunction of $\cL'_{r,F}$ with eigenvalue 
$\lambda \eta$. But by assumption, $\lambda \eta$ is greater than $\rho(r)$, which is the spectral radius of $\cL'_{r,F}$, so this is only possible if $X'_r \upsilon  = 0$.

Thus we deduce that for any $f\in \mathscr{C}^{\infty,\bdd}(S)$, $\upsilon(X_r f) = 0.$
Now, recalling \Cref{lem:Xi_z}, for $g \in \mathscr{C}_c^{\infty,\bdd}(\tS)$, 
\[\zeta(Xg) = \upsilon(\pi_0\circ\Xi_r Xg) = \upsilon(X_r(\pi_0\circ\Xi_r g)) = 0.\]

Hence we see that $\zeta$ is an invariant distribution for the horizontal flow on $\tS$. Further, $\zeta$ is a Maharam distribution as for $D = [\gamma] + \Gamma^{\abel, G}\in \Deck$, 
\[e^{r \cdot\xi(Dx) - r \cdot \xi(x)} = e^{ r \cdot \omega([\gamma]) },\]
so $\Xi_r (g\circ D) = e^{r \cdot \omega([\gamma])} \Xi_r g$, and thus
\[\zeta \circ D = e^{r\cdot \omega([\gamma])} \zeta. \qedhere \]
\end{proof}

\begin{corollary}
\label{cor:maharam_measures}
For every $r\in \R^d$ there exists a $\phi_s$-invariant Radon measure $\mu_r$ on $\tS$ which is $e^{r\cdot \omega}$-Maharam.
\end{corollary}
\begin{proof}
By \Cref{lem:nu_is_measure_2}, we know that the eigenfunction $\nu_r$ of the dual transfer operator $\cL'_{r,F}$ corresponding to eigenvalue $\rho(r)$ is a probability measure on $S$.

Taking $\nu_r$ as $\upsilon$ in \Cref{prop:Maharam_distribution}, we see that if we define the measure $\mu_r$ by 
\begin{equation}\label{eq:Maharam_measure}
\mu_r(g) = \nu_r(\pi_0 \circ \Xi_{r} g),
\end{equation}
then $\mu_r$ is a $\phi_s$-invariant measure on $\tS$, that satisfies 
\[\mu_r \circ D = e^{r \cdot \omega ([\gamma])} \mu_r.\]
The measure $\mu_r$ is Radon since $\nu_r$ is finite.
\end{proof}

\subsection{The description of the spectrum}\label{sec:description_spectrum}
The goal of this section is to understand the structure of the spectrum of $\cL_{z,F}$ acting on $\cB_{p,q}$. 
The arguments are inspired by the work \cite{DaRi}. 

The twisted transfer operator $\cL_z = \cL_{z,F}$ induces a natural action $\cL_z^\#$ on differential forms $\alpha$ by $\cL_z^\# (\alpha) = T_{\ast}(e^{z\cdot F} \alpha)$; e.g., given a 1-form $\alpha = f \diff x + g\diff y$, we have
\begin{equation}\label{eq:hash_action}
\cL_z^\# ( f \diff x + g\diff y) = \lambda \, \cL_zf \, \diff x + \lambda^{-1} \, \cL_z g\, \diff y.
\end{equation}
We note that 
\[
\diff_z \circ \cL_z^\# = \cL_z^\# \circ \diff_z;
\]
in particular, the action of $\cL_z^\#$ induces an action on the twisted cohomology complex as follows: let $[\alpha] \in H^k_z(S,\C)$ be a twisted cohomology class, then 
\[
\cL_z^\#[\alpha] = [\cL_z^\# \alpha].
\]
Our goal is to relate the discrete spectrum of $\cL_z$ with the eigenvalues of $\cL_z^\#$ acting on $H^1_z(S,\C)$.

\subsubsection{Preliminaries}

Denote by $E_\alpha$ the generalised eigenspace of $\cL_{z,F}$ with eigenvalue $\alpha$.
\begin{lemma}\label{lem:eigenspaces}
For a generalised eigenspace $E_\alpha$, $X_z(E_\alpha) \subset E_{\lambda^{-1}\alpha}$ and 
$Y_z(E_\alpha) \subset E_{\lambda\alpha}$.
\end{lemma}
\begin{proof}
Suppose $\ell \in E_\alpha$, so for some $n$, $(\cL_{z,F} - \alpha \id)^n \ell = 0$. Then by \Cref{lem:X_z commutation}, 
\[\lambda^{n}(\cL_{z,F}-\lambda^{-1}\alpha \id)^n X_z\ell = X_z \left( (\cL_{z,F} -\alpha \id)^n \ell  \right) = 0.\]
Hence $X_z \ell \in E_{\lambda^{-1}\alpha}$. Similarly for $Y_z$.
\end{proof}

Hence for $\ell\in E_\alpha$, if $n$ is the least integer such that $|\alpha|\lambda^n > \rho(z)$, we must have $Y_z^n \ell = 0$. 
Thus, we will first aim to understand the space $\cB_{p,q} \cap \ker Y_z$.

\begin{lemma}\label{lem:Xz_injectivity}
Let $\ell\in \cB_{p,q}$ be such that $X_z \ell = 0$. If $z \in 2\pi i \Z^d$, then $\ell$ is a constant; otherwise, $\ell = 0$.
\end{lemma}
\begin{proof}
Consider the lift $\tilde \ell$ to $\tS$. If $X_z \ell = 0$, then $X \tilde \ell = 0$. This implies that the distribution induced by $\tilde \ell$ on a bi-infinite horizontal leaf is invariant under horizontal translation. Since we know from~\Cref{thm:flow_ergodic} that the horizontal flow on $\tS$ is ergodic, it follows that there exists a dense bi-infinite horizontal leaf. As $\ell$ is continuous under vertical translation, and hence so is $\tilde \ell$, $\tilde \ell$ must be constant on $\tS$.
But it is also $(-z)$-Maharam, so, if  $z \notin 2\pi i \Z^d$, it must be 0. Hence also $\ell$ must be 0.
\end{proof}

We now show that if a smooth form is an exact current, then it is an exact form as well.

\begin{lemma}\label{lem:exact_current}
Let $\alpha$ be a $\diff_z$-closed 1-form on $S$ or on $S_0$, and assume that $\alpha = \diff_z \ell$ (in the sense of currents) for some $\ell \in \cB_{p,q}$. Then, $\alpha$ is $\diff_z$-exact.
\end{lemma}
\begin{proof}
Assume $\alpha$ is a 1-form on $S_0$, and let $\beta$ be any closed 1-form on $S$ vanishing at $\Sigma$. Then, 
\[
\left\lvert \int_{S_0} \alpha \wedge \beta \right\rvert = \left\lvert \langle \diff_z \ell, \beta \rangle \right\rvert = \left\lvert \langle \ell, \diff_z \beta \rangle \right\rvert =0.
\]
Since the pairing between $H^1(S_0,\C)$ and $H^1(S,\Sigma,\C)$ is perfect, we deduce that $[\alpha] = 0$. The proof in the case of a 1-form on $S$ is analogous.
\end{proof}

\subsubsection{From twisted cohomology to the anisotropic Banach spaces}

Inspired by the work \cite{DaRi}, we now prove that for all eigenvalues of the induced action on the twisted cohomology there is an associated eigenvalue of the operator $\cL_z$ (unless the eigenvalue is $\lambda^{-1}$ and $z \in 2\pi i \Z^d$).

\begin{proposition}\label{prop:recovering spectrum}
Let $[\alpha] \in H^1_z(S,\C)$ be a cohomology class which is an eigenvector for $\cL_z^\#$; i.e., $\cL_z^\#[\alpha] = \mu [\alpha]$.
For any $p,q \in \Z_{\geq 0}$ satisfying $\rho_{\ess}(z) < \lambda^{-1} |\mu|$, either $\mu = \lambda^{-1}$ and  $z \in 2\pi i \Z^d $, or there exists a $\ell \in E_{\lambda^{-1}\mu} \setminus \{0\}$.
\end{proposition}

\begin{proof}
By \cite[Section 5]{Forni:twisted}, we can assume that $\alpha$ is a twisted holomorphic 1-form on $S$; in particular, it is a smooth bounded function. The restriction to $S_0$ can be written as $\alpha = f\diff x  + g \diff y$, with $f,g \in \mathscr{C}^{\infty,\bdd}(S)$. By assumption and \eqref{eq:hash_action}, there exists a smooth function $h\in \mathscr{C}^{\infty,\bdd}(S)$ such that 
\[
\mu ( f\diff x  + g \diff y) =  \lambda \cL_zf \diff x + \lambda^{-1} \cL_zg \diff y + \diff_zh.
\]
Since, by assumption, $\lambda^{-1} |\mu|$ is larger than the essential spectral radius of $\cL_z$, we can find a small loop $\gamma$ going around $\mu$ 
such that $\gamma$, $\lambda^{-1} \gamma$ and $\lambda \gamma$ do not intersect the spectrum of~$\cL_z$ (equivalently, $\gamma$ does not intersect the spectra of $\lambda\cL_z$ and of $\lambda^{-1}\cL_z$). 
We can then define
\[
\begin{split}
    P_1&:= \Pi_{\lambda^{-1}\mu}
    = \frac{1}{2 \pi i} \int_{\lambda^{-1}\gamma} (w-\cL_{z})^{-1} \diff w
    = \frac{1}{2 \pi i} \int_{\gamma} (w-\lambda \cL_{z})^{-1} \diff w, \\
    P_2&:= \Pi_{\lambda \mu}
    = \frac{1}{2 \pi i} \int_{\lambda\gamma} (w-\cL_{z})^{-1} \diff w
    = \frac{1}{2 \pi i} \int_{\gamma} (w-\lambda^{-1} \cL_{z})^{-1} \diff w.
\end{split}
\]

The $P_j$ are bounded operators on $\cB_{p,q}$ and satisfy $P_j^2= P_j$ (namely, they are projections), so that 
\[
\cB_{p,q} = P_j(\cB_{p,q}) \oplus \ker P_j.
\]
Note that the projection on $\ker P_j$ is simply $ \id - P_j$. 
We also note that, since $\lambda\cL_{z}$ and $P_1$ commute, we have 
\[
(\mu - \lambda\cL_z) \circ (\id - P_1) = (\id - P_1) \circ (\mu -\lambda\cL_z),
\]
and similarly for $\lambda^{-1}\cL_z$ and $P_2$.

The spectrum of the restriction of $\lambda\cL_z$ on $\ker P_1$ coincides with the spectrum of $\lambda\cL$ outside the loop $\gamma$; therefore, it follows that $\mu - \lambda\cL_z$ is invertible on $\ker P_1$. 
Reasoning similarly for $P_2$, we have
\[
\begin{split}
f &= P_1f + (\mu - \lambda \cL_z)^{-1}(\mu - \lambda \cL_z)(\id - P_1)f, \qquad \text{and}\\
g&=P_2g + (\mu - \lambda^{-1} \cL_z)^{-1}(\mu - \lambda^{-1} \cL_z)(\id - P_2)g.
\end{split}
\]

Define the operator $\Pi^{\ast}$ acting on $\cB_{p,q}$ 1-currents (namely, on 1-forms on $S_0$ with coefficients in $\cB_{p,q}$) by 
\[
\Pi^{\ast}(f\diff x + g\diff y) = P_1f \diff x + P_2g\diff y.
\]
Let $P_0$ be the Riesz projection of $\cL_z$ associated to $\mu$, that is 
\[
P_0\coloneqq \frac{1}{2\pi i}\int_\gamma (w-\cL_z)^{-1}\,\diff w.
\]
 By the
commutation relations
\[
X_z\cL_z=\lambda\cL_zX_z,
\qquad
Y_z\cL_z=\lambda^{-1}\cL_zY_z,
\]
we have
\[
P_1X_z=X_zP_0,\qquad P_2Y_z=Y_zP_0.
\]
Moreover,
\[
(\mu-\lambda\cL_z)^{-1}(\id-P_1)X_z
=
X_z(\mu-\cL_z)^{-1}(\id-P_0),
\]
and
\[
(\mu-\lambda^{-1}\cL_z)^{-1}(\id-P_2)Y_z
=
Y_z(\mu-\cL_z)^{-1}(\id-P_0).
\]
Therefore,
\[
\begin{split}
\alpha
&= \Pi^{\ast}\alpha
+(\mu-\lambda\cL_z)^{-1}(\id-P_1)(X_zh)\,\diff x
+(\mu-\lambda^{-1}\cL_z)^{-1}(\id-P_2)(Y_zh)\,\diff y \\
&= \Pi^{\ast}\alpha
+\diff_z\bigl((\mu-\cL_z)^{-1}(\id-P_0)h\bigr).
\end{split}
\]

We have shown that there exists $\widetilde{h} \in \cB_{p,q}$ such that $\alpha=\Pi^{\ast} \alpha +\diff_z \widetilde{h}$. 

We now claim that $\Pi^{\ast} \alpha \neq 0$. Indeed, if this was not the case, then $\alpha$ would be an exact 1-current $\diff_z \widetilde{h}$. By \Cref{lem:exact_current}, the cohomology class $[\alpha]$ would be 0, contradicting the assumptions.
Since $\Pi^{\ast} \alpha \neq 0$, we have that either $P_1 f \neq 0$ or $P_2 g \neq 0$. In the first case, we obtain that $\cL_z \circ P_1 f = \lambda^{-1}\mu P_1 f \neq 0 $, which is our desired conclusion.
Let us assume that $\widetilde{g} = P_2g \neq 0$, with $\cL_z \widetilde{g} = \lambda\mu \widetilde{g}$. Then, either $\widetilde{g}$ is a constant, and hence $\lambda\mu =1$ and $z \in 2\pi i \Z^d$, or $\widetilde{g}$ is not constant. In this case, $X_z^2\widetilde{g} \neq 0$ by \Cref{lem:Xz_injectivity} and it is an eigenvector with eigenvalue $\lambda^{-1}\mu$. 
In all cases, the proof is complete.
\end{proof}
\subsubsection{From the anisotropic Banach spaces to the twisted cohomology}

Consider an element $\ell \in \cB_{p,q} \cap \ker Y_z$. As before, since $ \cB_{p,q}$ is a space of distributions, we will 
associate to $\ell$ a current $\ell \diff x$. 
Then, $\ell \diff x$ is a $\diff_z$-closed current, since
\begin{align*}
\diff_z (\ell \diff x) &= \diff(\ell \diff x) - z\cdot  \omega \wedge \ell \diff x\\
&= (X (\ell) \diff x + Y (\ell) \diff y) \wedge \diff x - \ell \big( z\cdot \omega(X) \diff x + z\cdot \omega(Y) \diff y \big)  \wedge \diff x\\
&= \big(Y (\ell) - z\cdot \omega(Y) \ell\big) \diff y \wedge \diff x = Y_z(\ell) \diff y \wedge \diff x = 0.
\end{align*}

We now show a de Rham Theorem for our space $ \cB_{p,q}$; namely, we prove that, for any given cohomology class $[\ell \diff x]$ in the space of 1-currents, we can find a smooth representative.
\begin{proposition}\label{prop:de_Rham}
Let $\ell \in \cB_{p,q}$. There exist $R\ell \in \mathscr{C}^{p}(S_{\reg})$ and $A\ell \in  \cB_{p,q-1}$ such that $R\ell - \ell = X_z(A\ell)$. Furthermore, if  $\ell \in \ker Y_z$, then $R\ell \in \ker Y_z$ and $A\ell \in \ker Y_z$.
\end{proposition}
\begin{proof}
Let $\ell \in \cB_{p,q}$, and let $f_n \in \mathscr{C}^{\infty,\bdd}(S)$ be so that $f_n \to \ell$ in $\cB_{p,q}$.
We start from observing that
\[
Y_zf_n \to Y_z\ell \; \text{\ in\ } \; \cB_{p-1,q}, \qquad \text{and} \qquad X_zf_n \to X_z\ell \; \text{\ in\ } \; \cB_{p,q-1}.
\]
Indeed, let us justify the former claim (the second is analogous): from the definition of the norms, it follows immediately that the sequence $Y_zf_n$ is Cauchy in $\cB_{p-1,q}$ and hence has a limit $\widetilde{\ell}$. As distributions in $\mathscr{C}^{-q}(S)$, we clearly have that  $\widetilde{\ell} - \ell = 0$, so that $\widetilde{\ell} = \ell$ in $\cB_{p-1,q}$ by \Cref{lem:inclusion}.

Let us fix $\theta$ a smooth, nonnegative, compactly supported function in $(0,1)$ such that $\int \theta = 1$; we define
\[
J(t) = H(t) - \int_{-\infty}^t \theta(r)\diff r,
\]
where $H$ is the Heaviside step function, and
\[
K(x,t) = \exp\left(-z\cdot\int_{x}^{\phi_t(x)}\omega\right) = \exp\left(-z\cdot\int_{0}^t \omega(X) \circ \phi_r(x) \diff r \right).
\]
We notice that $J(t)=0$ for all $t\in \R\setminus [0,1]$ and, as a distribution, $J'(t) = \delta_0(t) - \theta(t)$, where $\delta_0$ is the Dirac delta.

Let $f$ be any continuous function on $S_{\reg}$. For any $x \in S_{\reg}$, we define
\[
If(x) = \int f \circ \phi_t(x) \, K(x,t) \, J(t) \diff t.
\]  
The integral above (over $\R$ or over $[0,1]$ equivalently) defines a function on $S_{\reg}$ with the same regularity as $f$, and with comparable norms, since $\|K(\cdot,t)\|_{\mathscr{C}^k} \ll 1$ uniformly in $t \in [0,1]$, for all fixed $k \geq 0$.
We make the following claim.

\medskip

\noindent \textbf{Claim:} the sequence $If_n$ converges in $\cB_{p,q-1}$.

\medskip

Let us assume the claim, and denote $A\ell \in \cB_{p,q-1}$ the limit. For the sake of notation, we also denote $\omega_X = \omega(X)$ and $\omega_Y = \omega(Y)$. Then, by our initial observation, we have
\[
\begin{split}
X_zA\ell(x) &= \lim_{n \to \infty} XIf_n(x) - z\cdot \omega_X(x) If_n(x) \\
&= \lim_{n \to \infty}  \int \Big[ \frac{\diff}{\diff t}(f_n \circ \phi_t(x)) \, K(x,t) + f_n \circ \phi_t(x) \, XK(x,t) \Big] \, J(t) \diff t  - z\cdot \omega_X(x) If_n(x)\\
&= \lim_{n \to \infty}  \int - f_n \circ \phi_t(x) \, \partial_t[ K(x,t)\, J(t)]  -f_n \circ \phi_t(x) \, z\cdot \omega_X\circ \phi_t(x) K(x,t)\, J(t) \diff t \\
&= \lim_{n \to \infty} \int - f_n \circ \phi_t(x) \, K(x,t) \, (\delta_0(t) -\theta(t))\diff t  = -\ell(x) + R\ell(x),
\end{split}
\]
where 
\[
R\ell(x) = \langle \ell(x), K(x,\cdot)\theta(\cdot)\rangle
\]
is a $\mathscr{C}^{p,\bdd}$-function on $S$.

Finally, let us assume that $Y_z\ell =0$. From the definition, we have $Y_zR\ell = R(Y_z\ell)=0$. Moreover, analogous computations as the ones above show that
\[
\begin{split}
Y_zA\ell(x) &= \lim_{n \to \infty} YIf_n(x) - z\cdot \omega_Y(x) If_n(x) \\
&= \lim_{n \to \infty}  \int [Yf_n \circ \phi_t(x) \, K(x,t) + f_n \circ \phi_t(x) \, YK(x,t) ] \, J(t) \diff t  - z\cdot \omega_Y(x) If_n(x)\\
&= \lim_{n \to \infty}  \int [Yf_n \circ \phi_t(x) \, K(x,t) + f_n \circ \phi_t(x) (-z\cdot  \omega_Y\circ \phi_t(x)) \, K(x,t)]\, J(t) \diff t \\
&= \lim_{n \to \infty} \int Y_zf_n \circ \phi_t(x) \, K(x,t)\, J(t) \diff t = Y_z\ell(x) \ast [K(x,\cdot)\, J(\cdot)],
\end{split}
\]
where, in the last equality, $\ast$ denotes the convolution of distributions. Since $ Y_z\ell(x) = 0$, we deduce that $Y_zA\ell(x) = 0$ in $(\mathscr{C}^{q-1}_c(0,1) )^{\ast}$, showing  $A\ell \in \ker Y_z$.

\medskip

It remains to prove the claim. We need to show that the sequence of functions $If_n$ is Cauchy with respect to $\|\cdot\|_{p,q-1}$. Let us then fix $x\in S_{\reg}$, $k_1,k_2 \in \Z_{\geq 0}$ with $k = k_1+k_2 \leq p$, and $u\in \mathscr{C}^{q-1+k}_c(0,1)$ with $\|u\|_{\mathscr{C}^{q-1+k}} \leq 1$; we estimate
\[
\begin{split}
|\langle X^{k_1}Y^{k_2}(If_n - If_m)(x),u\rangle| &= \left\lvert\int \left(\int Y^{k_2}[(f_n-f_m) \circ \phi_t(\phi_r(x)) \, K(\phi_r(x),t)] \, J(t) \diff t \right) u^{(k_1)}(r)\diff r\right\rvert \\
&\ll_p  \left\lvert\int \int Y^{k_3}(f_n-f_m) \circ \phi_{t+r}(x) \, Y^{k_4}K(\phi_r(x),t) \, J(t) \, u^{(k_1)}(r) \diff t  \diff r\right\rvert,
\end{split}
\]
where $k_3+k_4 \leq k_2$. Then, we have
\[
|\langle X^{k_1}Y^{k_2}(If_n - If_m)(x),u\rangle| \ll_p  \left\lvert \int Y^{k_3}(f_n-f_m) \circ \phi_{s}(x) U(s)\diff s\right\rvert,
\]
where 
\[
U(s) = \int Y^{k_4}K(\phi_{s-t}(x),t) \, J(t) \, u^{(k_1)}(s-t) \diff t.
\]
Since $u^{ (k_1)}$ is of class $\mathscr{C}^{q-1+k_2}$, the function $U$ is of class  $\mathscr{C}^{q+k_2}$; furthermore, its support is contained in $(0,2)$. Using a partition of unity $\{v_1,v_2, v_3\}$, with each $v_i$ supported on an interval of length one, we deduce
\[
\begin{split}
&|\langle X^{k_1}Y^{k_2}(If_n - If_m)(x),u\rangle| \ll_p  \left\lvert \int Y^{k_3}(f_n-f_m) \circ \phi_{s}(x) U(s)v_i\diff s\right\rvert \\
& \qquad \qquad \ll_p \|Y^{k_3}(f_n-f_m) \|_{0,q+k} \, \|U v_i\|_{\mathscr{C}^{q+k}} \ll_p \|f_n-f_m\|_{p,q} \, \|u\|_{\mathscr{C}^{q-1+k}} \ll_p \|f_n-f_m\|_{p,q}.
\end{split}
\]
This shows that $If_n$ is Cauchy and hence completes the proof.

\end{proof}

Using \Cref{prop:de_Rham}, we can prove the following result.
\begin{proposition}\label{prop:de_Rham_vice_versa}
Let $\ell \in \cB_{p,q} \cap \ker Y_z$, and assume that $\cL_z\ell = \lambda^{-1}\mu \ell$. Then, either $\mu$ is an eigenvalue of $\cL_z^\#$ on $H^1_z(S_0,\C)$, or there exists $\widetilde{\ell}\in \cB_{p,q} \cap \ker Y_z$ such that $\cL_z\widetilde{\ell} =\mu \widetilde{\ell}$.
\end{proposition}
\begin{proof}
Given $\ell$ as in the assumptions, by  \Cref{prop:de_Rham}, there exist $f \in \mathscr{C}^{p}(S_{\reg})\cap \ker Y_z$ and $h \in  \cB_{p,q-1} \cap \ker Y_z$ such that $\ell = f + X_zh$. The 1-form $f\diff x$ on $S_{\reg}$ is then $\diff_z$-closed, and hence defines a twisted cohomology class $[f \diff x] \in H^1_z(S_{\reg},\C) \simeq   H^1_z(S_0,\C)$. We distinguish two cases.

If $[f \diff x] \neq 0$, using \Cref{lem:X_z commutation}, we have
\[
\cL_z^\# (f \diff x) = \lambda (\cL_z f) \, \diff x = \lambda (\cL_z \ell - \cL_z X_zh) \, \diff x = (\mu\ell -  X_z \cL_zh) \, \diff x = \mu f \diff x + X_z(\mu h-\cL_zh)  \diff x. 
\]
Since $Y_z h =0$, we can rewrite the equation above as $ \cL_z^\# (f \diff x) = \mu f \diff x + \diff_z \widetilde{h}$, for the distribution $ \widetilde{h} = \mu h-\cL_zh$. By \Cref{lem:exact_current}, we obtain $\cL_z^\#[f \diff x] = \mu [f \diff x]$, which holds in $ H^1_z(S_0,\C)$, thus proving that $\mu$ is an eigenvalue of $\cL_z^\#$ on $H^1_z(S_0,\C)$.

Otherwise, if $[f \diff x] = 0$, then $f\diff x = \diff_z \widetilde{f}$ for some $\widetilde{f} \in \mathscr{C}^{p}(S_{\reg})$, which implies that $\ell = X_z\widetilde{h}$, where $\widetilde{h} = \widetilde{f} + h \in \cB_{p,q-1}$ satisfies $Y_z\widetilde{h}=0$. \Cref{lem:X_z commutation} implies that $X_z(\cL_z\widetilde{h} - \lambda \mu \widetilde{h})=0$ and \Cref{lem:Xz_injectivity} completes the proof.
\end{proof}

\subsubsection{Maharam-Pollicott-Ruelle resonances}

We can finally describe the spectrum of the operator $\cL_z$ on $\cB_{p,q}$; in particular its discrete spectrum.

The case $z = 0$ is already treated in \cite{FaGoLa}, here we focus on the case $z \in \R^d \oplus i(-\pi,\pi)^d \setminus \{0\}$.

\begin{theorem}
\label{thm:cohomological_Maharam}
Let $p,q \geq 1$ and $z \in \R^d \oplus i(-\pi,\pi)^d \setminus \{0\}$. Denote by $\sigma_{p,q}(\cL_{z,F})$ the spectrum of $\cL_{z,F}$ on $\cB_{p,q}$, and let
\[\sigma^+_{p,q}(\cL_{z,F}) = \{\alpha \in \sigma_{p,q}: |\alpha| > \lambda^{-p} \rho(z)\}.\]
Let $\Theta(z) \subset \Theta_0(z)$ denote the spectrum of the action of $\cL_z^\#$ on $H^1_z(S,\C)$ and on $H^1_z(S_0,\C)$ respectively.
Then 
\[
\{\lambda^{-n}\mu: n \geq 1, \, \mu \in \Theta(z), \, |\mu|> \lambda^{-\min\{p,q\}+1} \rho(z)\} \subseteq \sigma_{p,q}(\cL_{z,F}),
\]
and 
\[ 
\sigma^+_{p,q}(\cL_{z,F}) \subseteq \{ \lambda^{-n}\mu: 1 \leq n \leq p, \mu \in \Theta_0(z) \}.
\]
\end{theorem}
\begin{proof}
Let $\mu \in \Theta(z)$, with $|\mu|> \lambda^{-\min\{p,q\}+1} \rho(z)$. By \Cref{prop:recovering spectrum}, $\lambda^{-1}\mu \in\sigma_{p,q}(\cL_{z,F})$, so that $E_{\lambda^{-1}\mu} \neq \{0\}$. By \Cref{lem:eigenspaces}, 
$X_z$ maps $E_{\lambda^{-k}\mu}$ into $E_{\lambda^{-k-1}\mu}$, and by \Cref{lem:Xz_injectivity} this map is injective, so by induction $E_{\lambda^{-n}\mu} \neq \{0\}$ for all $n \geq 1$. This proves the first claim.

We now prove the second statement. 
Suppose that $\ell$ is an eigenfunction of $\cL_{z,F}$ with eigenvalue $\mu$, and $|\mu| > \lambda^{-p}\rho(z)$. By \Cref{lem:eigenspaces}, $Y_z^p \ell$ is an eigenfunction with eigenvalue $\lambda^p \mu$ of modulus exceeding $\rho(z)$, hence $Y_z^p \ell = 0$. Let $n \leq p$ be minimal such that $Y_z^n \ell = 0$, so $Y_z^{n-1} \ell$ is a non-zero eigenfunction in $\ker Y_z$, with eigenvalue $\lambda^{n-1}\mu$. 
By \Cref{prop:de_Rham_vice_versa}, either  $\lambda^{n}\mu \in  \Theta_0(z)$, or there exists $0 \neq \ell_1 \in E_{\lambda^{n}\mu} \cap \ker Y_z$. We repeat the reasoning with $\ell_1$ instead of $\ell$: either  $\lambda^{n+1}\mu \in  \Theta_0(z)$, or there exists $0 \neq \ell_2 \in E_{\lambda^{n+1}\mu} \cap \ker Y_z$. This process has to stop, since all eigenvalues of $\cL_{z,F}$ are of modulus less than $\rho(z)$. In particular, there exists  $\lambda^{k}\mu \in  \Theta_0(z)$, which completes the proof.

\end{proof}


\begin{corollary}
Let $r\in \R^d$ and let $\rho_T(r)$ denote the spectral radius of $\cL_r^\#$ on $H^1_r(S_0, \C)$. Then 
\[\rho(r) = \lambda^{-1} \rho_T(r).\]
\end{corollary}
\begin{proof}
Fix any $p,q \geq 1$, then $\rho(r) = \max \{|\mu|: \mu \in \sigma_{p,q}(\cL_{r,F})\}$.
Since $\ell_+$ is an eigenfunction of $\cL_{r,F}$ with eigenvalue $\rho(r)$, we see that for some $n \geq 1$, $\lambda^n \rho(r) \in \Theta_0(r)$. Hence $\rho_T(r) \geq \lambda \rho(r)$. 

Conversely, if there is any eigenvalue $\mu \in \Theta_0(r)$ such that $|\mu| > \lambda \rho(r)$, then also $\mu \in \Theta(r)$, and $|\mu| > \lambda^{-\min\{p,q\}+1}\rho(r)$, so $\lambda^{-1}\mu \in \sigma_{p,q}(\cL_{r,F})$. This would be a contradiction as $\lambda^{-1} |\mu| > \rho(r)$, so we conclude that $\rho_T(r) = \lambda \rho(r)$.
\end{proof}

\begin{corollary}
For every $r\in \R^d$ there is an infinite family of $\phi_s$-invariant, $e^{r\cdot \omega}$-Maharam distributions on $\tS$, parametrised by the set $\{\lambda^{-n}\mu: n \geq 1, \mu \in \Theta(r)\}$.
\end{corollary}
\begin{proof}
For every $\mu \in \Theta(r)$, we can fix some $p,q \geq 1$ such that $|\mu| > \lambda^{-\min\{p,q\}+1}\rho(r)$. (Note that $0 \notin \Theta(r)$ since $T$ is invertible.) Hence for $n\geq 1$, $\lambda^{-n}\mu$ belongs to the spectrum of $\cL'_{r,F}$ on $\cB'_{p,q}$. If $\upsilon$ is the corresponding eigenfunction, by \Cref{prop:Maharam_distribution}, to $\upsilon$ corresponds a $\phi_s$-invariant $e^{r\cdot \omega}$-Maharam distribution. 

As $X$ and $Y$ commute when acting on smooth functions on $\tS$, $X'_z$ and $Y'_z$ commute on $\cB'_{p,q}$. Hence for any $n\geq 1$, $(Y'_r)^n \upsilon$ is in $\ker X'_r$, and hence as in the proof of \Cref{prop:Maharam_distribution}, $(Y'_r)^n \upsilon$ can be also used to construct a $\phi_s$-invariant $e^{r\cdot \omega}$-Maharam distribution.
\end{proof}


\section{Ergodic integrals}\label{sec:ergodic_integrals}

The goal of this section is to describe the ergodic integrals of smooth functions for points which are generic for the Maharam measures. 
%
We first make some preliminary assumptions. 
Let $C = C_{\omega} >0$ be such that, for any $\tx,\ty \in \tS$, the norm of the vector
\[
\int_{\tx}^{\ty} \mathcal{p}^{\ast}\omega = \left( \int_{\tx}^{\ty} \mathcal{p}^{\ast}\omega_1,\dots, \int_{\tx}^{\ty} \mathcal{p}^{\ast}\omega_d \right)
\]
can be bounded by $C$ times the distance between $\tx$ and $\ty$.
Up to replacing $T$ with a power of itself, we can assume that $\lambda >4$.  

Fix any $ \eta \in (\lambda^{-1},1)$.
Since $r\mapsto \rho(r)$ is smooth and $\rho(0) = 1$, there exists $c=c_{\omega} >0$ such that, for every $r\in \R^d$ with $\|r\|_{\infty} <c$, we have 
\begin{equation}\label{eq:constraint_on_r}
    (\lambda \rho(r))^{-1} \, e^{\|r\|_{\infty} \, (\|F\|_{\infty} + C)} < \eta < 1.
\end{equation}
In the following, we will use the notation $O(A)$ to refer to a term (which might change from line to line) which can be bounded in absolute value by $C' \cdot A$, where the constant $C'$ depends only on $C$ and on $c$.

Let us fix $r\in \R^d$, with $\|r\|_{\infty} <c$ as above. We call the $r$-displacement at time $t$ the function on $\tS$ defined by
\[
\delta_r(\tx,t) \coloneqq r \cdot \int_{\tx}^{\tphi_t(\tx)} \mathcal{p}^{\ast}\omega.
\]
It is easy to verify that $\delta_r(\tx,t)$ is invariant by $\Deck$ and thus can be identified with a function in $S$. 
The name is motivated by the fact that the vector $\int_{\tx}^{\phi_t(\tx)} \mathcal{p}^{\ast}\omega$, whose components are $\int_{\tx}^{\phi_t(\tx)} \mathcal{p}^{\ast}\omega_j$ for $j=1,\dots, d$, represents the displacement in the cover $\tS$ between the point $\tx$ and $\phi_t(\tx)$.

For any $L\geq 0$, let us also define 
\begin{equation}\label{eq:definition_D_r}
    \Delta_r(x,L) \coloneqq \langle \ell_{+,-r}(x), \exp(\delta_r(x,s)) \, \one_{[0,L]}(s) \rangle = \int_0^L \exp(\delta_r(x,s)) \, \diff \widehat{\ell_{+,-r}(x)}(s);
\end{equation}
in other words, it is the integral of the function $ \exp(\delta_r(x,s))$ over a horizontal segment of length $L$ starting at $x \in S_{\reg}$ with respect to the measure induced by $\ell_{+,-r}(x)$.
The parametrization $\Delta_r(x,\cdot)$ of the horizontal orbit starting at $x$ has the following \emph{self-similarity} property.
\begin{lemma}\label{lem:self_similarity_of_Delta}
For every  $x \in S_{\reg}$ and $L\geq 0$, we have
\[
\Delta_r(Tx,L) = (\lambda \rho(-r))^{-1} \exp(r\cdot F(x) ) \, \Delta_r(x, \lambda L).
\]
\end{lemma} 
\begin{proof}
Let $\tx \in \tS$ such that $\mathcal{p}(\tx)=x$. 
Since $\ell_{+,-r} = \lim_{n \to \infty} \rho(-r)^{-n} \cL_{-r,F}^n(\one)$, it follows that 
\[
\begin{split}
&\Delta_r(x,L) = \lim_{n \to \infty} \rho(-r)^{-n-1} \int_0^L \exp(\delta_r(\tx,s) ) \exp\left( - r \cdot \int_{\tT^{-n-1}\circ \tphi_s(\tx)}^{\tphi_s(\tx)} \mathcal{p}^{\ast}\omega \right) \diff s \\
&= \rho(-r)^{-1} \lim_{n \to \infty} \rho(-r)^{-n} \int_0^L \exp\left(  r \cdot \int_{\tx}^{\tT^{-1}(\tx)} \mathcal{p}^{\ast}\omega +  r \cdot \int_{\tT^{-1}(\tx)}^{\tphi_{\lambda s}(\tT^{-1}\tx)} \mathcal{p}^{\ast}\omega - r \cdot \int_{\tT^{-n}\circ \tphi_{\lambda s}(\tT^{-1}\tx)}^{ \tphi_{\lambda s}(\tT^{-1}\tx)} \mathcal{p}^{\ast}\omega \right) \diff s \\
&=(\lambda \rho(-r))^{-1}\lim_{n \to \infty} \rho(-r)^{-n} \int_0^{\lambda L} \exp(r\cdot F(T^{-1}x) ) \, \exp(\delta_r(T^{-1}x,s) ) \cL_{-r,F}^n(\one)\circ \phi_{ s}(T^{-1}x)\diff s \\
&=(\lambda \rho(-r))^{-1} \exp(r\cdot F(T^{-1}x) )  \Delta_r(T^{-1}x, \lambda L). 
\end{split}
\]
This proves the lemma.
\end{proof}
We fix a point $\tx_0 \in \tS$. 
According to the definition in \eqref{eq:Maharam_measure},
we remark that the choice of the point $\tx_0 \in \tS$ corresponds to fixing a normalization of the (infinite) Maharam measure associated to the parameter $r\in \R^d$.
Let $\cF \subset \tS$ be a fundamental domain for the cover containing the point $\tx_0$. Let us also fix an isomorphism of $\Deck$ with $\Z^d$, and we will enumerate elements $D_m \in \Deck$ with $m\in \Z^d$.

Given $t\geq 1$, we will denote $n=n(t) \in \Z_{\geq 0}$ the integer $n= \lceil \frac{\log t}{\log \lambda} \rceil$.

We will prove the following result.
\begin{theorem}\label{thm:ergodic_integrals_asymptotics}
There exists a constant $c=c_{\omega} >0$ such that for any $r\in \R^d$ with $\|r\|_{\infty} < c$ the following holds, where $\Sigma = \Sigma(r)$ is the $d\times d$ symmetric negative definite matrix given by \Cref{prop:pert}.
    
    For any $f\in \mathscr{C}^1_c(\tS)$ with $\mu_r(f) \neq 0$, any $\tx\in \tS$ with an infinite forward orbit satisfying
    \[
    \lim_{j \to \infty} \frac{S_jF(\mathcal{p}(\tx))}{j} = \nu_{T,r}(F),
    \]
    and for all $t\geq 1$ we have
    \[
    \int_0^{t} f\circ \phi_s(\tx)\diff s = \frac{(\lambda \rho(r))^n e^{-r\cdot \xi(\tT^n \tx)}}{(2\pi n)^{d/2} \sqrt{\det(-\Sigma)}} \mu_{r}(f) \Bigg( \Delta_{r}\big(T^n \mathcal{p}(\tx),\frac{t}{\lambda^n}\big) e^{Z_n(\tx)} + o(1) \Bigg),
    \]
    where $o(1)$ denotes a term which tends to zero as $t\to \infty$, and where $Z_n(\tx)$ is explicitly defined in \eqref{eq:definition_Z_j} and satisfies
    \[
    Z_n(\tx) \to  -\frac{1}{2} (\xi(\tx)+Z) \cdot \Sigma^{-1} (\xi(\tx)+Z) \text{ \ in distribution (w.r.t.~$\mu_r$)}, \qquad \text{ where }Z \sim \matholdcal{N}(0,-\Sigma).
    \]
\end{theorem}

Let $K \in \N$, and let $f\in \mathscr{C}^{1}_c(\tS)$ be so that the support of $f$ is contained in the union $\cup \{D_m(\cF) \ : \ m\in \Z^d, \ \|m\|_{\infty} \leq K\}$.
In the following, the implicit constant in the big-$O$ notation is allowed to depend on $K$ as well. 

Let $\tx\in \tS$ and let $t\geq 1$ be large enough
; our goal is to study the integral $\int_0^{t}f\circ \phi_s(\tx)\diff s$. Let $x = \mathcal{p}(\tx)$.

To exploit the theory developed so far, we need to \lq\lq smoothen\rq\rq\ the indicator function $\one_{[0,t]}$, for which we will use the following lemma.

\begin{lemma}\label{lem:smoothen_integral}
    There exist non-negative $\mathscr{C}^2$-functions $w, w^{\pm}$ compactly supported on $\pm (0,1)$ such that, for every $n\geq 1$, we have 
    \[
    \one_{[\lambda^{-n},1-\lambda^{-n}]}(s) \leq w(s) + \sum_{j=1}^{n-1}\big(w^{+}(\lambda^{j}s) + w^{-}(\lambda^{j}(s-1))\big) \leq \one_{[0,1]}(s).
    \]
\end{lemma}
\begin{proof}
    Let us define 
    \[
    u''(x) \coloneqq 
    \begin{cases}
        0 & \text{ if $x \leq 1/4$ or $x \geq 3/4$,} \\
        64(4x-1) & \text{ if $x\in [1/4,3/8]$,} \\
        128(1-2x) & \text{ if $x\in [3/8,5/8]$,} \\
        64(4x-3) & \text{ if $x\in [5/8,3/4]$.} \\
    \end{cases}
    \]
    Then, the function $u$, with $u(0) = u'(0) = 0$ is a $\mathscr{C}^2$, non-decreasing function on $\R$, equal to 0 on $(-\infty, 0]$ and to 1 on $[1,+\infty)$. 
    We let $u^{+}(s) = u(s)$ and $u^{-}(s) = 1-u(s)$; moreover, we define
    \[
    w^{+}(s) = u^{-}(s) u^{+}(\lambda s), \qquad \text{and} \qquad w^{-}(s) = u^{+}(s+1) u^{-}(\lambda s+1). 
    \]

Let us define $w$ as follows: we define $w(s)$ to be $1$ for all $s \in [\lambda^{-1}, 1-\lambda^{-1}]$, to be $u(\lambda s)$ for $s \in [0, \lambda^{-1}]$, and to equal $u^{-}(\lambda (s-1) + 1)$ for $s \in [1-\lambda^{-1}, 1]$.
\begin{figure}[h!]
\centering
\includegraphics[width=0.7\textwidth]{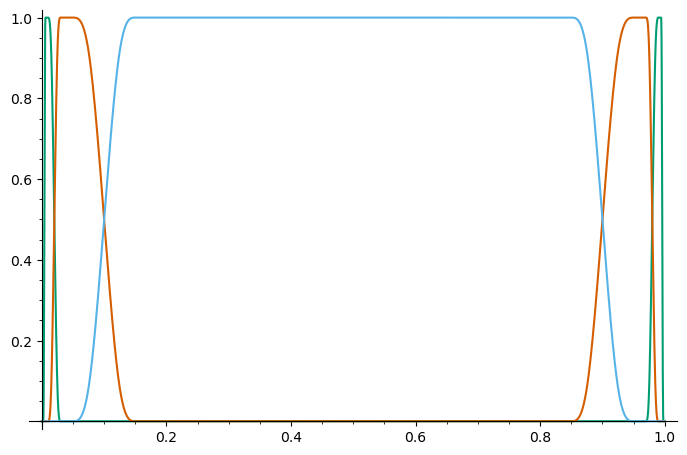}
\caption{The functions $w$ (in blue), $w^+(\lambda s) + w^-(\lambda(s-1))$ (in red) and $w^+(\lambda^2 s) + w^-(\lambda^2(s-1))$ (in green), plotted for $\lambda=5$.}
\end{figure}

    By construction, we have $\one_{[0,1]}(s) - w(s) = u^{-}(\lambda s) + u^{+}(\lambda (s-1) + 1)$; let us focus on the first term $u^{-}(\lambda s) $. We have
    \[
    \begin{split}
    u^{-}(\lambda s) &= u^{-}(\lambda s) \big( u^{+}(\lambda^2 s) + u^{-}(\lambda^2s)\big) = w^{+}(\lambda s) + u^{-}(\lambda s)u^{-}(\lambda^2s) \\
    & = w^{+}(\lambda s) + u^{-}(\lambda^2s), 
    \end{split}
    \]
    where we used the fact that $\lambda>4$.
    Inductively,
    \[
    \begin{split}
    u^{-}(\lambda s) &= w^{+}(\lambda s) + u^{-}(\lambda^2s) \\
    &=  w^{+}(\lambda s) + w^{+}(\lambda^2s) + \cdots + w^{+}(\lambda^{n-1} s) + u^{-}(\lambda^ns).
    \end{split}
    \]
    We reason similarly for the second term, namely
    \[
    \begin{split}
    u^{+}(\lambda (s-1) + 1) & = u^{+}(\lambda (s-1)  + 1) \big( u^{+}(\lambda^2(s-1)  + 1) + u^{-}(\lambda^2(s-1) + 1)\big) \\
    &= w^{-}(\lambda (s-1) ) + u^{+}(\lambda^2(s-1)  + 1), 
    \end{split}
    \]
    since $u^{+}(\lambda (s-1)  + 1) \cdot u^{+}(\lambda^2(s-1)  + 1)=u^{+}(\lambda^2(s-1) + 1)$.
    We conclude inductively as before.
\end{proof}
By applying \Cref{lem:smoothen_integral} to $\one_{[0,t]}(s) = \one_{[0,1]}(t^{-1}s)$, we immediately obtain the following corollary.

\begin{corollary}\label{cor:smoothen_integral}
    Let $\tx=\tx^{+}$ and $\phi_{t}(\tx) = \tx^{-}$, then
    \begin{multline*}
    \left\lvert \int_0^{t} f\circ \phi_s(\tx)\diff s - \Bigg( \int_0^{\lambda^{n}} f\circ \phi_s(\tx) w(t^{-1}s)\diff s + \sum_{j=1}^{n-1} \int_0^{\lambda^{j}} f\circ \phi_{\pm s}(\tx^{\pm}) w^{\pm}(\pm t^{-1} \lambda^{n-j}s)\diff s \Bigg)
    \right\rvert \\ = O(\|f\|_{\infty}).
    \end{multline*}
\end{corollary}
We will now study each of the smoothened integrals in \Cref{cor:smoothen_integral}. All of them are treated in the same way, so we now consider 
\[
\int_0^{\lambda^{j}} f\circ \phi_{s}(\tx) w(\lambda^{-j}s)\diff s,
\]
where, by a little abuse of notation, $w(s)$ above can be either $w(t^{-1}\lambda^{n}s)$ or any of the $w^{\pm}(\pm  t^{-1} \lambda^{n} s)$ and $\tx$ can be either $\tx^{+}$ or $\tx^{-}$; we remark that the difference in the orientation of the integral will not play a role. 

For $\theta \in \left(-\frac12,\frac12\right)^d$, we write $z=r+2\pi i\theta\in \R^d \oplus 2\pi i \left(-\frac12,\frac12\right)^d$.  
By \Cref{lem:proj_cl} and \Cref{lem:proj_Xi}, we have
\[
f = \int_{\left(-\frac12,\frac12\right)^d} \Xi_{-z}(g_{z}) \diff \theta, \qquad \text{where} \qquad g_{z} = \pi_{0} \circ \Xi_{z}(f) \in \mathscr{C}^{1}(S).
\]
We remark that, in the definition of $\Xi_{z}(f) = e^{z\cdot \xi}f$, the base point $\tx_0 \in \cF$ was fixed at the beginning of this section.

\begin{lemma}\label{lem:estimates_on_gz}
    We have
    \[
    \|g_z\|_{\mathscr{C}^1} 
    = O \big( \|f\|_{\mathscr{C}^1}\big).
    \]
\end{lemma}
\begin{proof}
    Let $\tx \in \cF$. Then, by assumption, for all $m \in \Z^d$ with $\|m\|_{\infty}> K$, we have $f\circ D_m(\tx) = 0$.
    Thus, it follows that 
    \[
    |g_z(x)| \leq \sum_{D_m\in \Deck} e^{|\Re z \cdot \xi(D_m\tx)|} |f| \circ D_m(\tx) = O \left(  K^d e^{C_{\omega} \, \|r\|_{\infty} \, K} \|f\|_{\infty} \right) = O\big( \|f\|_{\mathscr{C}^0}\big).
    \]
    Moreover, for $V\in \{X,Y\}$, we have
    \[
    |V g_z(x)| \leq \sum_{D_m\in \Deck} \Big(\|z\|_{\infty} \, \|\omega(V)\|_{\infty} |f| \circ D_m(\tx) + |V f| \circ D_m(\tx) \Big) e^{|\Re z \cdot \xi(D_m\tx)|} = O\big( \|f\|_{\mathscr{C}^1}\big),
    \]
    where the last equality follows exactly as before. 
    The proof is complete.
\end{proof}

Using \Cref{lem:invar},  we have
\begin{equation}\label{eq:ergod_int_2}
	\begin{split}
		& \int_0^{\lambda^{j}}f\circ \phi_s(\tx) w(\lambda^{-j}s)\diff s = \int_{\left(-\frac12,\frac12\right)^d} \int_0^{\lambda^{j}}\Xi_{-z}(g_{z})\circ \phi_s(\tx) w(\lambda^{-j}s)\diff s \diff \theta \\
		&\qquad \qquad =\int_{\left(-\frac12,\frac12\right)^d} \int_0^{\lambda^{j}}\Xi_{-z}(g_{z})\circ \tT^{-j} \circ\phi_{\lambda^{-j}s}(\tT^{j}\tx) w(\lambda^{-j}s)\diff s \diff \theta \\
		&\qquad \qquad =\int_{\left(-\frac12,\frac12\right)^d} \lambda^{j}\int_0^{1}[U^{j} \circ \Xi_{-z}(g_{z})] \circ\phi_{s}(\tT^{j}\tx) w(s)\diff s \diff \theta \\
		&\qquad \qquad =\int_{\left(-\frac12,\frac12\right)^d} \lambda^{j}\int_0^{1}[\Xi_{-z} \circ \cL_{z,F}^{j}(g_{z})] \circ\phi_{s}(\tT^{j}\tx) w(s)\diff s \diff \theta.
	\end{split}
\end{equation}
Note that, by the properties in \Cref{lem:smoothen_integral}, we have $\|w(s)\|_{\mathscr{C}^2} = \|w(t^{-1}\lambda^{n}s)\|_{\mathscr{C}^2}= O(1)$.

By definition, $\Xi_{-z} \circ \cL_{z,F}^{j}(g_z) =  e^{-z \cdot \xi} \cdot \cL_{z,F}^{j}(g_z)$. We rewrite
\[
e^{-z \cdot \xi\circ \phi_{s}(\tT^{j}\tx)} = e^{-z \cdot \xi(\tT^{j}\tx)} \, \exp\left(-z\cdot \int_{\tT^{j}\tx}^{\phi_{s}(\tT^{j}\tx)} \mathcal{p}^{\ast} \omega \right) = e^{-z \cdot \xi(\tT^{j}\tx)} \exp(\delta_{-z}(\tT^{j}\tx,s)).
\]
Substituting back into \eqref{eq:ergod_int_2}, we get
\begin{equation}\label{eq:ergod_int_4}
	\begin{split}
		& \int_0^{\lambda^{j}}f\circ \phi_s(\tx) w(\lambda^{-j}s)\diff s \\
		&\qquad \qquad =\int_{\left(-\frac12,\frac12\right)^d} \lambda^{j}  e^{-z \cdot \xi(\tT^{j}\tx)} \int_0^{1}\cL_{z,F}^{j}g_{z}  \circ\phi_{s}(T^{j}x) \exp(\delta_{-z}(T^{j}x,s)) w(s)\diff s  \diff \theta \\
		&\qquad \qquad =\int_{\left(-\frac12,\frac12\right)^d} \lambda^{j}  e^{-z \cdot \xi(\tT^{j}\tx)}  \langle \cL_{z,F}^{j}g_{z}(T^jx), \exp(\delta_{-z}(T^jx)) \, w\rangle \diff \theta.
	\end{split}
\end{equation}
We need the following estimates on the function $e^{\delta_{-z}(T^jx)} \, w \colon s \mapsto \exp(\delta_{-z}(T^jx,s)) \, w(s)$, with $s\in [0,1]$.

\begin{lemma}\label{lem:estimates_on_Gzju}
    We have $\|e^{\delta_{-z}(T^jx)} \cdot w\|_{\infty} = O(1)$ and $\| (e^{\delta_{-z}(T^jx)} \cdot w)' \|_{\infty} = O(1)$.
    Moreover
    \[
    \left\| \frac{\diff}{\diff \theta}e^{\delta_{-z}(T^jx)} \cdot w \right\|_{\mathscr{C}^1} \leq C_{\omega}(\pi+\|r\|_\infty) e^{C_{\omega} \|r\|_\infty} = O(1).
    \]
\end{lemma}
\begin{proof}
    From its definition, it follows immediately that $|e^{\delta_{-z}(T^jx,\cdot)}| \leq e^{C \|r\|_{\infty}}$ on $[0,1]$. Since $|w(s)| \leq 1$, the first claim is proved. 

    Secondly, we have $(e^{\delta_{-z}(T^jx)} )'(s) = -z \cdot \mathcal{p}^{\ast}\omega(X)\circ \phi_s(\tT^j\tx) \, e^{\delta_{-z}(\tT^j\tx,s)}$, from which we deduce $|(e^{\delta_{-z}(T^jx)}) '(s)| \leq C \|z\|_{\infty} e^{C \|r\|_{\infty}}$. Since $\|\Im z\|_{\infty} \leq \pi $, we can replace $\|z\|_{\infty} \leq C (\pi+\|r\|_{\infty})$. In turn, this implies the second claim since $|w'(s)|= O(1)$.

    Finally, we compute
    \[
    \frac{\diff}{\diff \theta}e^{\delta_{-z}(T^jx)}(s) = -2\pi i e^{\delta_{-z}(T^jx)}(s) \, \theta \cdot \int_{\tT^j\tx}^{\phi_s(\tT^j\tx)} \mathcal{p}^{\ast} \omega.
    \]
    From here, the last estimate can then be obtained exactly as before.
\end{proof}

By \Cref{prop:pert}, there exist $\sigma \in (0,\rho(r))$, $\varepsilon \in (0,1)$, a constant $C$, and a negative definite symmetric matrix $\Sigma$ so that 
\[
\|\cL_{z,F}^{n}\|_{1,1} \leq C \sigma^{n},
\]
for all $\theta = \frac{1}{2\pi}\Im z \in \left(-\frac12,\frac12\right)^d \setminus B(0,\varepsilon)$,
and
\[
\begin{split}
&\log \rho(z) = \log \rho(r) + 2\pi i\nu_T(F)\cdot \theta +2\pi^2 \theta \cdot \Sigma \theta + E_1(\theta), \qquad \text{where\ } | E_1(\theta)|\leq \|\theta\|_{2}^3,\\
&\text{and} \qquad \|Q_z^{n}\|_{1,1} \leq C \sigma^{n},
\end{split}
\]
for all $\theta = \frac{1}{2\pi}\Im z \in B(0,\varepsilon)$. In particular, for all $\theta \in B(0,\varepsilon)$, we have 
\[
\rho(z)^{n} = \rho(r)^{n} \exp \left(n(2\pi i\nu_T(F)\cdot \theta +2\pi^2 \theta \cdot \Sigma \theta +  E_1(\theta))\right).
\]
Plugging this into \eqref{eq:ergod_int_4}, by \Cref{lem:estimates_on_gz}, we deduce
\begin{equation}\label{eq:ergod_int_5}
	\begin{split}
		&\int_0^{\lambda^{j}}f\circ \phi_s(\tx) w(\lambda^{-j}s)\diff s = \lambda^{j} e^{-r\cdot\xi(\tT^j\tx)} \Bigg( O\Big(\sigma^j \|f\|_{\mathscr{C}^1}\Big)  \\
        & \quad  + \rho(r)^j \int_{B(0,\varepsilon)}  e^{2 \pi i[-\xi(\tT^j\tx) + j \nu_T(F)]\cdot \theta +2 \pi^2 j \theta \cdot \Sigma \theta +j E_1(\theta)} \langle \Pi_z g_{z}(T^jx), e^{\delta_{-z}(T^jx)} \, w\rangle \diff \theta  \Bigg),
	\end{split}
\end{equation}
where 
we used \Cref{lem:estimates_on_Gzju} and the fact that $\|g_z\|_{1,1} \leq \|g_z\|_{\mathscr{C}^1}$. 
We notice that $\xi(\tT^j\tx) = \xi(\tx) + S_jF(x)$, since the forms $\omega_i$ are closed.
Thus, if we call $F_0 = F - \nu_T(F)$ the centered Frobenius function, from \eqref{eq:ergod_int_5} we get
\begin{equation}\label{eq:ergod_int_6}
	\begin{split}
		&\int_0^{\lambda^{j}}f\circ \phi_s(\tx) w(\lambda^{-j}s)\diff s = (\lambda \rho(r))^j e^{-r\cdot\xi(\tT^j\tx)} \Bigg( O\Big(\|f\|_{\mathscr{C}^1} \frac{\sigma^j}{\rho(r)^j} \Big)  \\ 
        & \qquad + \int_{B(0,\varepsilon)}  e^{-2 \pi i[\xi(\tx) + S_jF_0(x)]\cdot \theta +2 \pi^2 j \theta \cdot \Sigma \theta +j E_1(\theta) } \langle \Pi_z g_{z}(T^jx), e^{\delta_{-z}(T^jx)} \, w\rangle \diff \theta \Bigg).
	\end{split}
\end{equation}
Let us call 
\[
P_j(\theta) =  \langle \Pi_z g_{z}(T^jx), e^{\delta_{-z}(T^jx)} \, w\rangle,
\]
and note that, by \Cref{lem:estimates_on_gz} and \Cref{lem:estimates_on_Gzju}, we have $|P_j(\theta)| = O( \|f\|_{\mathscr{C}^1})$. 

Using \Cref{lem:estimates_on_gz} and \Cref{lem:estimates_on_Gzju} again, we also have
\begin{equation*}
\begin{split}
|P_j(\theta) - P_j(0)|  \leq &|\langle \Pi_z g_z, (e^{\delta_{-z}(T^jx)}-e^{\delta_{-r}(T^jx)}) \, w \rangle |  \\
& + |\langle \Pi_z (g_z-g_r), e^{\delta_{-r}(T^jx)} \, w \rangle | +
  |\langle (\Pi_z - \Pi_r) g_r, e^{\delta_{-r}(T^jx)} \, w \rangle |\\
\leq& \|g_z\|_{1,1} \|\theta\|_2 C_\omega(\pi+\|r\|_\infty) e^{C_\omega \|r\|_\infty} + C \|\theta\|_2 K^d \|f\|_{\mathscr{C}^0} e^{C_\omega \|r\|_\infty} + \|\theta\|_{2} \|g_r\|_{1,1} e^{C_\omega \|r\|_\infty}\\
 = & O(\|\theta\|_2 \cdot \|f\|_{\mathscr{C}^1}).
\end{split}
\end{equation*}
We can thus rewrite \eqref{eq:ergod_int_6} as 
\begin{equation}\label{eq:ergod_int_7}
	\begin{split}
		&\int_0^{\lambda^{j}}f\circ \phi_s(\tx) w(\lambda^{-j}s)\diff s \\
        & \qquad = (\lambda \rho(r))^j e^{-r\cdot\xi(\tT^j\tx)}  \left(P_j(0) \int_{B(0,\varepsilon)}  e^{-2 \pi i[\xi(\tx) + S_jF_0(x)]\cdot \theta +2\pi^2 j \theta \cdot \Sigma \theta} \diff \theta + O\big( \|f\|_{\mathscr{C}^1} j^{-\frac{d+1}{2}} \big) \right),
	\end{split}
\end{equation} 
where we also used the fact that $| E_1(\theta)|\leq \|\theta\|_{2}^3$ and that $\int_{B(0,\varepsilon)}e^{j \theta \cdot \Sigma \theta} (1-e^{j \|\theta\|_{2}^3})\diff \theta = O( j^{-\frac{d+1}{2}})$.
 
\sloppy Since $\Pi_r g_r = \Pi_r (\pi_0\circ \Xi_r f)$, from the definition of $\mu_r$ in \eqref{eq:Maharam_measure} it follows that  $P_j(0) = \mu_{r}(f) \langle \ell_{+,r}(T^jx), e^{\delta_{-r}(T^jx)} \, w\rangle$.

The same proof as in \Cref{lem:self_similarity_of_Delta} gives us the following lemma.
\begin{lemma}\label{lem:dimension_of_l+}
    Let $a\in\mathscr{C}^1_c(0,1)$ and $y\in S_{\reg}$. 
    For any $k \in \N$, we have
    \[
    \langle \ell_{+}(y), a(s) \rangle = (\lambda \rho)^{k} \langle \ell_{+}( T^ky), e^{r\cdot S_kF}\circ \phi_{\lambda^ks}(y) \, a(\lambda^ks) \rangle.
    \]
\end{lemma}
Applying the previous lemma to $\langle \ell_{+,r}(T^{j}x), e^{\delta_{-r}(T^jx)} \, w\rangle$ with $y=T^{j}x$ and $k=n-j$, we get
\[
\begin{split}
    &\langle \ell_{+,r}(T^{j}x), e^{\delta_{-r}(T^{j}x)} \, w_j\rangle \\
    &\qquad = (\lambda \rho)^{n-j} \langle \ell_{+,r}(T^{n}x), e^{-r\cdot S_{n-j}F}\circ \phi_{\lambda^{n-j}s}(T^{j}x) \, e^{\delta_{-r}(T^{j}x)}(\lambda^{n-j}s) \, w(\lambda^{n-j}s)\rangle.
\end{split}
\]
We notice that
\[
\begin{split}
    &\exp \Big[-r \cdot  S_{n-j}F\circ \phi_{\lambda^{n-j}s}(T^{j}x)\Big] \\
    &\qquad = \exp \Big[-r \cdot  S_{n-j}F(T^{j}x)\Big] \cdot \exp \Big[-r \cdot  \Big( S_{n-j}F\circ \phi_{\lambda^{n-j}s}(T^jx) -  S_{n-j}F(T^jx)\Big)\Big] \\
    &\qquad = \exp \Big[-r \cdot  S_{n-j}F(T^{j}x)\Big] \cdot \exp \Bigg[-r \cdot  \Bigg( \int_{\phi_{\lambda^{n-j}s}(\tT^j\tx)}^{\tT^{n-j}\circ \phi_{\lambda^{n-j}s}(\tT^j\tx)} \mathcal{p}^{\ast}\omega - \int_{\tT^j\tx}^{\tT^n\tx} \mathcal{p}^{\ast}\omega\Bigg)\Bigg]\\
    &\qquad = \exp \Big[-r \cdot  S_{n-j}F(T^{j}x)\Big] \cdot \exp \Bigg[-r \cdot  \Bigg( \int_{\tT^n\tx}^{\phi_{s}(\tT^n\tx)} \mathcal{p}^{\ast}\omega - \int_{\tT^j\tx}^{\phi_{\lambda^{n-j}s}(\tT^j\tx)} \mathcal{p}^{\ast}\omega\Bigg)\Bigg],
\end{split}
\]
where, in the last equality, we used Stokes' Theorem. Therefore, we obtained
\[
\exp \Big[-r\cdot S_{n-j}F\circ \phi_{\lambda^{n-j}s}(T^{j}x)\Big] = \exp \Big[-r\cdot S_{n-j}F(T^{j}x)\Big] \cdot e^{\delta_{-r}(T^{n}x)}(s) \cdot [e^{\delta_{-r}(T^{j}x)}(\lambda^{n-j}s)]^{-1},
\]
which yields
\[
\langle \ell_{+,r}(T^{j}x), e^{\delta_{-r}(T^{j}x)} \, w_j\rangle = (\lambda \rho)^{n-j} e^{-r\cdot S_{n-j}F(T^{j}x)}\langle \ell_{+,r}(T^{n}x), e^{\delta_{-r}(T^{n}x)}(s) \, w(\lambda^{n-j}s)\rangle.
\]
In conclusion, let
\[
P_j = \langle \ell_{+,r}(T^{n}x), e^{\delta_{-r}(T^{n}x)}(s) \, w(\lambda^{n-j}s)\rangle;
\]
then, $P_j(0) = \mu_{r}(f)(\lambda \rho)^{n-j} e^{-r\cdot S_{n-j}F(T^jx)} P_j$.
Since  $(\lambda \rho)^{-(n-j)}e^{(n-j)\|r\|_{\infty} \cdot \|F\|_{\infty}} = O(\eta^{n-j})$ by \eqref{eq:constraint_on_r}, 
substituting back into \eqref{eq:ergod_int_7}, we obtain
\begin{equation}\label{eq:ergod_int_8}
	\begin{split}
		&\int_0^{\lambda^{j}}f\circ \phi_s(\tx) w(\lambda^{-j}s)\diff s= (\lambda \rho(r))^n e^{-r\cdot \xi(\tT^n\tx)} \\
        &\qquad \qquad \times \Big( \mu_{r}(f) P_j   \int_{B(0,\varepsilon)}  e^{-2 \pi i[\xi(\tx) + S_jF_0(x)]\cdot \theta +4 \pi^2 j \theta \cdot \Sigma \theta} \diff \theta + O\big( \|f\|_{\mathscr{C}^1}\eta^{n-j} j^{-\frac{d+1}{2}} \big) \Big).
	\end{split}
\end{equation}

Let us define the following norm in $\R^d$: given a vector $v \in \R^d$, we set
\[
\|v\|^2_{\Sigma} \coloneqq- v\cdot \Sigma^{-1} v.
\]
Applying a standard stationary-phase type estimate to \eqref{eq:ergod_int_8}, and since $P_j =  O(\eta^{n-j})$, gives us
\begin{equation}\label{eq:ergod_int_9}
	\begin{split}
		&\int_0^{\lambda^{j}}f\circ \phi_s(\tx) w(\lambda^{-j}s)\diff s \\
        & \quad =  (\lambda \rho(r))^n e^{-r\cdot\xi(\tT^n\tx)} \left(  \frac{ \mu_{r}(f) P_j  }{(2\pi j)^{d/2} \sqrt{\det(-\Sigma)}} e^{-\frac{1}{2j} \|\xi(\tx) + S_jF_0(x)\|^2_{\Sigma}} + O\big( \|f\|_{\mathscr{C}^1}\eta^{n-j} j^{-\frac{d+1}{2}} \big) \right).
	\end{split}
\end{equation}
In conclusion, we have shown that
\begin{equation}\label{eq:ergod_int_10}
		\int_0^{\lambda^{j}}f\circ \phi_s(\tx) w(\lambda^{-j}s)\diff s =  \frac{(\lambda \rho(r))^n e^{-r\cdot\xi(\tT^n\tx)}}{(2\pi n)^{d/2} \sqrt{\det(-\Sigma)}} \mu_{ r}(f) \Big(  P_j \Big( \frac{n}{j} \Big)^{d/2} e^{Z_j} + E_j\Big),
\end{equation}
where 
\begin{equation}\label{eq:definition_Z_j}
E_j =  O(\|f\|_{\mathscr{C}^1} n^{d/2} \eta^{n-j}j^{-\frac{d+1}{2}}), \qquad \text{and} \qquad  Z_j = Z_j(\tx) =-\frac{1}{2j} \|\xi(\tx) + S_jF_0(x)\|^2_{\Sigma}.
\end{equation}

We now sum the contributions for all the different values of $j$.
\begin{lemma}
    We have
    \[
    |E_n(x)| + \sum_{j=1}^{n-1} | E_j(x^{\pm})| =O(\|f\|_{\mathscr{C}^1} n^{-\frac{1}{2}}).
    \]
\end{lemma}
\begin{proof}
    By the estimate on $E_j$ above, we have $|E_n(x)| =O( \|f\|_{\mathscr{C}^1} n^{-\frac{1}{2}}) $ and, denoting $n_0 = \lfloor n/2\rfloor$, 
    \[
    \begin{split}
    \sum_{j=1}^{n-1} | E_j(x^{\pm})| &\ll  \|f\|_{\mathscr{C}^1} n^{d/2} \sum_{j=1}^{n-1} \eta^{n-j}j^{-\frac{d+1}{2}} \ll \|f\|_{\mathscr{C}^1} n^{d/2} \Big(\sum_{j=1}^{n_0} \eta^{n-j} + n_0^{-\frac{d+1}{2}} \sum_{j=n_0}^{n-1} \eta^{n-j}\Big) \\
    &=O( \|f\|_{\mathscr{C}^1}  n^{d/2} \eta^{n_0} ) + O(\|f\|_{\mathscr{C}^1} n^{-\frac{1}{2}}) = O(\|f\|_{\mathscr{C}^1} n^{-\frac{1}{2}}),
    \end{split}
    \]
    for all $t$ sufficiently large. The result is proven.
\end{proof}

\begin{lemma}
    Let $x \in S_{\reg}$ be such that $\frac{S_nF_0(x)}{n} \to 0$. Then
    we have
    \[
    \Bigg\lvert \sum_{j=1}^{n-1} P_j \Big(\frac{n}{j} \Big)^{\frac{d}{2}} e^{Z_j} -  e^{Z_n}\sum_{j=1}^{n-1} P_j\Bigg\rvert  \to 0. 
    \]
\end{lemma}
\begin{proof}
    Let us first observe that $|e^{Z_j}| \leq 1$. 
    We first claim that 
    \[
    \Bigg\lvert \sum_{j=1}^{n-1} P_j \Big(\frac{n}{j} \Big)^{\frac{d}{2}} e^{Z_j} -  \sum_{j=1}^{n-1} P_j e^{Z_j} \Bigg\rvert \to 0.
    \]
    Indeed, for any $k \leq n-1$, we have
    \[
    \begin{split}
    \Bigg\lvert \sum_{j=1}^{n-1} P_j \Big(\frac{n}{j} \Big)^{\frac{d}{2}} e^{Z_j} - \sum_{j=1}^{n-1} P_j e^{Z_j} \Bigg\rvert &\ll \sum_{j=1}^{n-1} P_j \Big(\frac{n-j}{j}\Big)^{\frac{d}{2}} \ll  \sum_{j=1}^{n-k-1} \eta^{n-j}n^{\frac{d}{2}} + \sum_{j=n-k}^{n-1} \eta^{n-j} \Big(\frac{k}{n-k}\Big)^{\frac{d}{2}}\\
    &\ll \eta^kn^{\frac{d}{2}} + \Big(\frac{k}{n-k}\Big)^{\frac{d}{2}},
    \end{split}
    \]  
    and the last term can be made arbitrarily small by choosing $k$ appropriately (e.g., $k = O( \sqrt{n})$). This proves our claim.

    It remains to show that 
    \[
    \Bigg\lvert \sum_{j=1}^{n-1} P_j e^{Z_j} - e^{Z_n} \sum_{j=1}^{n-1} P_j \Bigg\rvert \to 0.
    \]
    Fix $\varepsilon >0$, and let $k \in \N$ be such that $\eta^k<\varepsilon$. Let now $N\in \N$ be such that $|S_nF_0(x)/n| < \varepsilon/k$ for all $n\geq N$. For $n \geq N+k$, for any $j\leq n-1$ such that $n-j \leq k$, using the cocycle relation $S_nF_0(x) = S_jF_0(x) + S_{n-j}F_0(T^jx) = S_jF_0(x) + O(\|F\|_{\infty} (n-j)) $, we have 
    \[
    \begin{split}
    \Big|Z_n - Z_j\Big| &\ll \left( \frac{1}{2j} - \frac{1}{2n}\right) \|\xi(\tx) + S_jF_0(x)\|_{\Sigma}^2 + O \left(\frac{k |S_jF_0(x)|}{n}\right) + O \left(\frac{k^2}{n}\right)= O(\varepsilon).
    \end{split}
    \]
    Therefore,
    \[
    \begin{split}
    \Bigg\lvert \sum_{j=1}^{n-1} P_j e^{Z_j} -  e^{Z_n}\sum_{j=1}^{n-1} P_j \Bigg\rvert &\ll  \sum_{j=1}^{n-k-1} P_j+ \sum_{j=n-k}^{n-1} P_j |e^{Z_j} - e^{Z_n}| \\
    &\ll \eta^k + \sum_{j=n-k}^{n-1} \eta^{n-j} O(\varepsilon) = O(\varepsilon).
    \end{split}
    \]
    This completes the proof.
\end{proof}

Since we also have that 
\[
|e^{Z_n(\tx^{+})} - e^{Z_n(\tx^{-})}| \to 0
\]
if $\tx=\tx^{+}$ is such that $\frac{S_nF_0(x)}{n} \to 0$, we can sum all the contributions in \eqref{eq:ergod_int_10} for $j = 1, \dots, n-1$ and for both $\tx=\tx^{+}$ and for $\phi_t(\tx) =\tx^{-}$. We conclude that
\[
\int_0^{t} f\circ \phi_s(\tx)\diff s  =  \frac{(\lambda \rho(r))^n e^{-r\cdot\xi(\tT^n\tx)}}{(2\pi n)^{d/2} \sqrt{\det(-\Sigma)}} \mu_{ r}(f) \left( e^{Z_n(x)} \Big(  P_n(x) + \sum_{j=1}^{n-1}  P_j(x^{\pm}) \Big) +o(1)\right).
\]
Let $W_n(s)$ be the $\mathscr{C}^1$ function compactly supported on $(0,1)$ defined by
\[
W_n(s) =  w(s) + \sum_{j=1}^{n-1}\big(w^{+}(\lambda^{j}s) + w^{-}(\lambda^{j}(s-1))\big), 
\]
as in \Cref{lem:smoothen_integral}; then, 
\[
 P_n(x) + \sum_{j=1}^{n-1}  P_j(x^{\pm}) =  \langle \ell_{+,r}(T^{n}x), e^{\delta_{-r}(T^{n}x)}(s) \, W_n(t^{-1}\lambda^ns)\rangle.
\]
The last step is to replace $W_n$ with $\one_{[0,t^{-1}\lambda^n]}$. Note that the difference is a function whose support is contained in a union of two intervals, each of diameter $\lambda^{-n}$. The following lemma comes to aid.
\begin{lemma}
    Let $a$ be a $\mathscr{C}^1$ function supported on an interval of diameter $\lambda^{-n}$, and let $x\in S_{\reg}$. Then $\langle \ell_{+}(x), a \rangle = O(\eta^{n}) $.
\end{lemma}
\begin{proof}
    By \Cref{lem:dimension_of_l+} , we have that $\langle \ell_{+}(x), a(s) \rangle = (\lambda \rho)^{-n} \langle \ell_{+}(T^{-n}x), e^{r\cdot S_{-n}F} \circ \phi_{\lambda^{-n}s}a(\lambda^{-n}s) \rangle $. Since $\ell_{+}(T^{-n}x)$ is a measure, the right-hand side is $O((\lambda \rho)^{-n} e^{\|r\|_{\infty}n \|F\|_{\infty}}) = O(\eta^n)$, by our assumption on $r$. The proof is complete.
\end{proof}

From the definition \eqref{eq:definition_D_r}, we conclude
\begin{equation*}
    \int_0^{t} f\circ \phi_s(\tx)\diff s = \frac{(\lambda \rho(r))^n e^{-r\cdot \xi(\tT^n\tx)}}{(2\pi n)^{d/2} \sqrt{\det(-\Sigma)}} \mu_{ r}(f) \Big( \Delta_r(T^nx,t\lambda^{-n} )e^{Z_n(x)} + o(1) \Big).
\end{equation*}
Finally, we have the convergence of $Z_n$ in distribution by \Cref{thm:CLT},
which completes the proof of \Cref{thm:ergodic_integrals_asymptotics}.


\section{Hausdorff dimension of Maharam measures}\label{sec:Hausdorff}
In this section we compute the Hausdorff dimension of the measure $\nu_r$ for $r\in \R^d$, which is also the Hausdorff dimension of the corresponding Maharam measure. This was recently done by Berk, Fr\k{a}czek, Kotlewski and Trujillo in~\cite[Theorem 1.5]{BFKT} and we recover their result with our method. The main difference in the proofs is that while in~\cite{BFKT} the authors obtain good control of the measure of the intervals of certain dynamical partitions arising from renormalisation, we can, using our transfer operators, control the measure of arbitrary intervals. This significantly simplifies the proof.

Let us fix $r\in \R^d$ and let $\nu, \nu_T$ and $\vartheta$ be the corresponding
measures, as defined in \Cref{sec:weighted_transfer_operators}. First we observe
that by \Cref{lem:expression_for_nu} it suffices to compute the Hausdorff
dimension of $\vartheta$, as $\dim_H \nu = 1+\dim_H \vartheta$. To compute the
dimension of $\vartheta$ we apply the following classical result, known as
Frostman's lemma~\cite{PrzytyckiUrbanski}:
\begin{lemma}\label{lem:Frostman}
Suppose that $\mu$ is a measure on $\R$, and there exist $\delta_1 \leq \delta_2$ such that for $\mu$-a.e. point $x$, 
\[\delta_1 \leq \lim_{\epsilon \to 0} \frac{\log \mu(x-\epsilon,x+\epsilon)}{\log \epsilon} \leq \delta_2.\]
Then $\delta_1 \leq \dim_H \mu \leq \delta_2$.
\end{lemma}

Firstly let us identify the $\vartheta$ -a.e. points for which we will compute the limit in the statement of Frostman's lemma.

\begin{lemma}\label{lem:generic_points}
Let $I$ be a vertical interval of length 1 on $S_0$ and let $\vartheta$ be the measure on it. Then for $\vartheta$-a.e. points in $I$, 
\[\lim_{n \to \infty} \frac{1}{n} S_n F(x) = \nu_T(F).\]
\end{lemma}
\begin{proof}
By the Birkhoff ergodic theorem for $\nu_T$, $\lim_{n \to \infty} \frac{1}{n} S_n F(x) = \nu_T(F)$ for $\nu_T$-almost every point $x$. 
Since $\lim_{n \to \infty} \frac{1}{n} S_n F(x)$ is constant on horizontal leaves we see that the limit holds for $\vartheta$-almost every point $x\in I$. 
\end{proof}

\begin{theorem}\label{thm:Hausdorff_dim_nu}
For $x\in S_\mathrm{reg}$ satisfying $\lim_{n \to \infty} \frac{1}{n} S_n F(x) = \nu_T(F)$, the following limit holds:
\[\lim_{\epsilon \to 0} \frac{\log \vartheta( \phi^\perp_{(-\epsilon,\epsilon)}(x))}{\log \epsilon} = \frac{\log \lambda + \log \rho(r) - r\cdot \nu_T(F)}{\log \lambda}.\]
Thus 
\[\dim_H \nu = 1+ \frac{\log \lambda + \log \rho(r) - r\cdot \nu_T(F)}{\log \lambda}.\]
\end{theorem}
\begin{proof}
Let us fix a point $x$ satisfying the assumption. There exists an $N\in \N$ such that the rectangle $\phi_{[-\frac12,\frac12]} \circ \phi^\perp_{[-\lambda^{-N},\lambda^{-N}]}(x)$ is contained in $S_0$. Define for all $n \geq N$ the test function $f_n \in \mathscr{C}^{\infty, \bdd}(S)$, such that $f_n$ is supported in the rectangle $\phi_{[-\frac12,\frac12]} \circ \phi^\perp_{[-\lambda^{-n},\lambda^{-n}]}(x)$ and is equal to 1 on the rectangle $\phi_{[-\frac14,\frac14]} \circ \phi^\perp_{[-\lambda^{-(n+1)},\lambda^{-(n+1)}]}(x)$.

Let $G^-$ and $G^+$ be respectively the minimum and maximum of $G_x(s)$ taken over $x\in S_\mathrm{reg}$ and $s\in [-\frac12,\frac12]$. Then observe that if $\lambda^{-(n+1)} \leq \epsilon \leq \lambda^{-n}$, then 
\begin{equation}\label{eq:nu_estimate_f}
 \frac{\nu(f_{n})}{G^+} \leq \vartheta(\phi^\perp_{(-\epsilon,\epsilon)}(x)) \leq \frac{2 \nu(f_{n-1})}{G^-}.
\end{equation}
Hence we focus on computing $\nu(f_n)$.\\

By the definition of $\nu$, $\nu = \rho^{-n} (\cL'_{r,F})^n \nu$, hence 
\[\nu(f_n) = \rho^{-n} \nu(f_n \circ T^{-n} e^{r\cdot S_n F} \circ T^{-n}).\]

The function $f_n \circ T^{-n}$ is supported on a rectangle of width $2\lambda^{-n}$ and of height 1, centred at $T^n(x)$. For any point $y$ in this rectangle, 
\[\|S_n F \circ T^{-n} (y) - S_n F \circ T^{-n} (T^n x)\| < C \|F\|_{\mathscr{C}^1},\]
where the constant $C$ is independent of $x$ or $n$. Hence up to a uniform multiplicative constant, 
$\nu(f_n \circ T^{-n} e^{r\cdot S_n F} \circ T^{-n}) = \nu(f_n\circ T^{-n}) e^{r\cdot S_n F(x)}$.\\

It remains to bound $\nu(f_n\circ T^{-n})$. By \Cref{lem:nu_is_measure_1}, an upper bound is $C \lambda^{-n}$. For the lower bound, let us show that 
$\liminf_{n\to \infty} \lambda^n \nu(f_n\circ T^{-n}) > 0$. Indeed, suppose for contradiction that there exists a sequence $(n_k)_k$ such that 
$\lambda^{n_k} \nu(f_{n_k} \circ T^{-n_k}) \to 0$. 

By compactness of $S$ the sequence of points $T^{n_k}(x)$ accumulates to some $x_\infty$. If $x_\infty$ happens to be a singularity, we can pass to a subsequence so that the accumulation happens in a half-plane either to the left or right of $x_\infty$. Assume without loss of generality the case of the right half-plane, and let $m$ be large enough such that the rectangle 
$R = \phi_{[0,\lambda^{-m}]} \circ \phi^\perp_{[-\frac15,\frac15]} (x_\infty)$ does not contain any singularities. Observe that for large enough $k$, $\nu(R) \leq \frac{G^+}{G^-} \lambda^{n_k-m} \nu(f_{n_k}\circ T^{n_k})$. Hence $\nu(R) = 0$, and hence $\vartheta$ on any vertical interval inside $R$ is identically 0.

Using the fact that $\diff \vartheta (\phi_s(x)) = G_x(s) \diff \vartheta(x)$ (see the proof of \Cref{lem:expression_for_nu}) and the minimality of the horizontal flow on $S$, we deduce that $\vartheta$ is identically equal to 0, and hence also $\nu=0$, which is a contradiction.

Thus we deduce that
\[\lim_{n\to \infty} \frac{\log \nu(f_n)}{-n\log \lambda} = \lim_{n\to \infty} \frac{-n \log \rho +n r\cdot \nu_T(F) -n \log \lambda}{-n\log \lambda}.\]
Combining with \eqref{eq:nu_estimate_f}, we deduce the first statement of the theorem, and the second follows by \Cref{lem:Frostman} and \Cref{lem:generic_points}.
\end{proof}

Observe that if the statement of \Cref{rmk:nu_T_is_Gibbs} holds, then for any real $r$, $\log \rho(r) = h_{\nu_{T,r}}(T) + r \cdot \nu_{T,r}(F) - \log \lambda$. Hence the Hausdorff dimension can also be written as 
\[\dim_H \vartheta = \frac{h_{\nu_T}(T)}{\log \lambda}.\]


\appendix

\section{Examples}\label{sec:examples}
In this appendix we present some examples of $\Z^d$-covers of translation surfaces which admit a pseudo-Anosov map $T$, and consider the twisted $\tT$-action on Maharam cohomology in these cases.

\subsection{Example: the $3\times1$ staircase, a $\Z$-cover}

Consider the translation surface $S$ shown in \Cref{fig:3staircaseS}, which is a square-tiled surface of genus 2 in the stratum $\mathcal{H}(2)$.

\begin{figure}[h]
\centering
\begin{tikzpicture}[scale=1.6,square/.style={regular polygon,regular polygon sides=4,inner sep=1.6,fill}]
\draw (0,0) -- (3,0) -- (3,1) -- (0,1) -- cycle;
\draw[dotted] (1,0) -- (1,1) (2,0) -- (2,1);
\node[skyblue,draw=none] at (0.5,1) {>};
\node[skyblue,draw=none] at (2.5,0) {>};
\node[orange,draw=none] at (1.5,0) {\large /};
\node[orange,draw=none] at (1.5,1) {\large /};
\node[reddishpurple,draw=none,font=\huge] at (0,0.5) {--};
\node[reddishpurple,draw=none,font=\huge] at (3,0.5) {--};
\node[bluishgreen,draw=none,font=\large] at (0.5,0) {$\wr$};
\node[bluishgreen,draw=none,font=\large] at (2.5,1) {$\wr$};

\foreach \i in {0,...,3}{
\node[square,vermillion] at (\i,0) {};
\node[square,vermillion] at (\i,1) {};

}
\end{tikzpicture}
\caption{The compact surface $S$ with its one singularity.}
\label{fig:3staircaseS}
\end{figure}
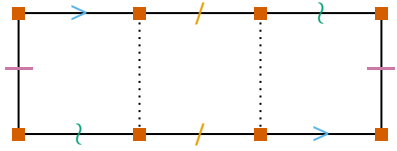

Consider the $\Z$-cover $\tS$ shown in \Cref{fig:3staircase}.

\begin{figure}[h]
\begin{tikzpicture}[square/.style={regular polygon,regular polygon sides=4,inner sep=1.6,fill}, every node/.style={draw}]
\draw (1,0) -- (3,0) -- (3,1) --  (5,1) -- (5,2) -- (7,2) -- (7,3) -- (9,3) -- (9,4); 
\draw (0,0) -- (0,1) -- (2,1) --  (2,2) -- (4,2) -- (4,3) -- (6,3) -- (6,4) -- (8,4); 
\foreach \i in {0,...,3}
{
	\draw[dotted] ({2*\i},\i) -- ({2*\i+1},\i);
	\foreach \j in {1,2,3}
		{
        	\draw[loosely dotted] ({\j+2*\i},\i) -- ({\j+2*\i},\i+1);
        }
}
\draw[dotted] (8,4) -- (9,4);

\foreach \i in {1,...,3}
{
	\pgfmathtruncatemacro{\ii}{{\i+1}};

	\node[orange,draw=none] at (1.5+2*\i,\i+1) {\large \rep{\i}{/}};
	\node[orange,draw=none] at (1.5+2*\i,\i) {\large \rep{\i}{/}};
	
	\node[skyblue,draw=none] at ({0.5+2*\i},\i-1) { \rep{\i}{>}};
	\node[skyblue,draw=none] at ({0.5+2*\i},\i+1) { \rep{\i}{>}};
}

\node[reddishpurple,draw=none,font=\huge] at (2,1.5) {--};
\node[reddishpurple,draw=none,font=\huge] at (5,1.5) {--};

\node[reddishpurple,draw=none,font=\huge] at (4,2.45) {--};
\node[reddishpurple,draw=none,font=\huge] at (4,2.55) {--};
\node[reddishpurple,draw=none,font=\huge] at (7,2.45) {--};
\node[reddishpurple,draw=none,font=\huge] at (7,2.55) {--};

\node[reddishpurple,draw=none,font=\huge] at (6,3.4) {--};
\node[reddishpurple,draw=none,font=\huge] at (6,3.5) {--};
\node[reddishpurple,draw=none,font=\huge] at (6,3.6) {--};
\node[reddishpurple,draw=none,font=\huge] at (9,3.4) {--};
\node[reddishpurple,draw=none,font=\huge] at (9,3.5) {--};
\node[reddishpurple,draw=none,font=\huge] at (9,3.6) {--};

\node[square,vermillion] at (2,0) {};
\node[square,vermillion] at (2,1) {};
\node[square,vermillion] at (2,2) {};
\node[square,vermillion] at (5,3) {};
\node[square,vermillion] at (5,1) {};
\node[square,vermillion] at (5,2) {};

\end{tikzpicture}
\caption{The cover $\tS$ with one of its singularities.}
\label{fig:3staircase}
\end{figure}
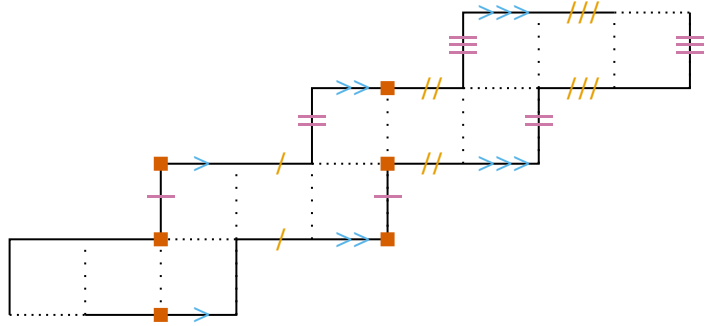

The surface $\tS$ admits cylinder decompositions in the vertical and horizontal directions, of modulus $\frac12$ and $\frac13$ respectively. Hence the Dehn twists in these cylinders, which can be represented by the matrices  $A = \begin{pmatrix} 1 & 0 \\ 2 & 1 \end{pmatrix}$ and $B = \begin{pmatrix} 1 & 3 \\ 0 & 1 \end{pmatrix}$ respectively, give automorphisms of $\tS$. Composing them, for example as $T = BA = \begin{pmatrix} 7 & 3 \\ 2 & 1 \end{pmatrix}$ gives a pseudo-Anosov map on $\tS$.

\subsubsection{The twisted $\tT$-action on Maharam cohomology of $\tS$}
The homology $H_1(\tS_0,\Z)$ is generated by the classes of three families of curves $(a_i)_{i\in \Z}, (b_i)_{i\in \Z}, (c_i)_{i\in \Z}$, as shown in \Cref{fig:staircase_homology}.

\begin{figure}[h]
\begin{tikzpicture}[square/.style={regular polygon,regular polygon sides=4,inner sep=2,fill}, every node/.style={draw}]
\draw (1,0) -- (3,0) -- (3,1) --  (5,1) -- (5,2) -- (7,2) -- (7,3) -- (9,3) -- (9,4); 
\draw (0,0) -- (0,1) -- (2,1) --  (2,2) -- (4,2) -- (4,3) -- (6,3) -- (6,4) -- (8,4); 
\foreach \i in {0,...,3}
{
	\draw[dotted] ({2*\i},\i) -- ({2*\i+1},\i);
	\foreach \j in {1,2,3}
		{
        	\draw[loosely dotted] ({\j+2*\i},\i) -- ({\j+2*\i},\i+1);
        }
}
\draw[dotted] (8,4) -- (9,4);

\foreach \i in {1,...,3}
{
	\pgfmathtruncatemacro{\ii}{{\i+1}};

	\draw[orange,->,line width=1] (1.5+2*\i,\i) -- (1.5+2*\i,\i+1);
	\node[orange,draw=none] at (1.5+2*\i,-0.3+\i) {$b_{\i}$};
	\draw[skyblue,->,line width=1] (0.5+2*\i,\i-1) -- (0.5+2*\i,\i+1);
	\node[skyblue,draw=none] at (0.5+2*\i,0.2+\i+1) {$a_{\i}$};
	\draw[reddishpurple,->,line width=1] (2*\i,0.5+\i) -- (3+2*\i,0.5+\i);
	\node[reddishpurple,draw=none] at (-0.25+2*\i,0.5+\i) {$c_{\i}$};
}

\end{tikzpicture}
\caption{The three families of generators of $H_1(\tS_0,\Z)$.}
\label{fig:staircase_homology}
\end{figure}
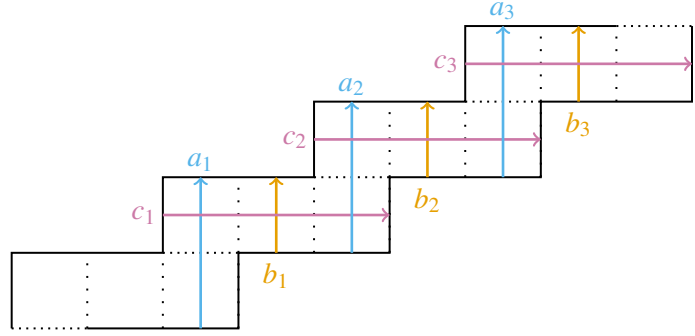

This means that if we know that a 1-form $\tilde \alpha$ on $\tS$ is
$z$-Maharam, then it is determined only by its integrals over $a_0,b_0,c_0$,
since $\int_{a_k} \tilde \alpha = e^{kz} \int_{a_0} \tilde \alpha$, and
similarly for $b_k,c_k$. This explains why the space $H^1(\tS_0,z,\C)$ is
3-dimensional, with a basis being $(\alpha,\beta,\gamma)$, where \[\alpha
([a_0]) = \beta ([b_0]) = \gamma ([c_0]) = 1\] and \[\alpha ([b_0]) = \alpha
([c_0]) =  \beta ([a_0]) = \beta ([c_0]) = \gamma ([a_0]) = \gamma ([b_0]) =
0.\]

To determine how $\tT_*$ acts on $H_1(\tS_0,z)$, we
first compute its action on the three families of curves $a_i,b_i,c_i$.

Let us first compute the action of $A = \begin{pmatrix} 1 & 0 \\ 2 & 1
\end{pmatrix}$. In \Cref{fig:homology_action} we draw the image of a section of $\tS$
containing $a_1,b_1,c_1$ under the matrix $A$, and the image of the curves after
cutting and pasting the surface back into its original form.

\begin{figure}[h]
\begin{tikzpicture}[square/.style={regular polygon,regular polygon sides=4,inner sep=2,fill}, every node/.style={draw}, scale=1.1]
\draw (2,0) -- (3,0) -- (3,1) --  (5,1) -- (5,2) ; 
\draw (2,1) --  (2,2) -- (4,2) -- (4,3) -- (5,3) ; 
\foreach \i in {1}
{
	\draw[dotted] ({2*\i},\i) -- ({2*\i+1},\i);
	\foreach \j in {1,2,3}
		{
        	\draw[loosely dotted] ({\j+2*\i},\i) -- ({\j+2*\i},\i+1);
        }
}
\draw[dotted] (5,2) -- (4,2);
\draw[loosely dotted] (5,2) -- (5,3);
\draw[loosely dotted] (2,0) -- (2,1);

\foreach \i in {1}
{
	\pgfmathtruncatemacro{\ii}{{\i+1}};

	\draw[orange,->,line width=1] (1.5+2*\i,\i) -- (1.5+2*\i,\i+1);
	\node[orange,draw=none] at (1.5+2*\i,-0.3+\i) {$b_{\i}$};
	\draw[skyblue,->,line width=1] (0.5+2*\i,\i-1) -- (0.5+2*\i,\i+1);
	\node[skyblue,draw=none] at (0.5+2*\i,0.2+\i+1) {$a_{\i}$};
	\draw[reddishpurple,->,line width=1] (2*\i,0.5+\i) -- (3+2*\i,0.5+\i);
	\node[reddishpurple,draw=none] at (-0.25+2*\i,0.5+\i) {$c_{\i}$};
}

\foreach \x in {7,...,10}{
\draw[gray!30] (\x,-3) -- (\x,6);
}

\foreach \y in {-3,...,6}{
\draw[gray!30] (7,\y) -- (10,\y);
}
\begin{scope}[cm={1,2,0,1,(5,-7)}]

\draw (2,0) -- (3,0) -- (3,1) --  (5,1) -- (5,2) ; 
\draw (2,1) --  (2,2) -- (4,2) -- (4,3) -- (5,3) ; 
\foreach \i in {1}
{
	\draw[dotted] ({2*\i},\i) -- ({2*\i+1},\i);
	\foreach \j in {1,2,3}
		{
        	\draw[loosely dotted] ({\j+2*\i},\i) -- ({\j+2*\i},\i+1);
        }
}
\draw[dotted] (5,2) -- (4,2);
\draw[loosely dotted] (5,2) -- (5,3);
\draw[loosely dotted] (2,0) -- (2,1);

\foreach \i in {1}
{
	\pgfmathtruncatemacro{\ii}{{\i+1}};

	\draw[orange,->,line width=1] (1.5+2*\i,\i) -- (1.5+2*\i,\i+1);
	\node[orange,draw=none] at (1.9+2*\i,-1+\i) {$A_* b_{\i}$};
	\draw[skyblue,->,line width=1] (0.5+2*\i,\i-1) -- (0.5+2*\i,\i+1);
	\node[skyblue,draw=none] at (-0.1+2*\i,1.1+\i+1) {$A_* a_{\i}$};
	\draw[reddishpurple,->,line width=1] (2*\i,0.5+\i) -- (3+2*\i,0.5+\i);
	\node[reddishpurple,draw=none] at (-0.4+2*\i,1.2+\i) {$A_* c_{\i}$};
}
\end{scope}

\begin{scope}[shift={(10,0)}]

\draw (2,0) -- (3,0) -- (3,1) --  (5,1) -- (5,2) ; 
\draw (2,1) --  (2,2) -- (4,2) -- (4,3) -- (5,3) ; 
\foreach \i in {1}
{
	\draw[dotted] ({2*\i},\i) -- ({2*\i+1},\i);
	\foreach \j in {1,2,3}
		{
        	\draw[loosely dotted] ({\j+2*\i},\i) -- ({\j+2*\i},\i+1);
        }
}
\draw[dotted] (5,2) -- (4,2);
\draw[loosely dotted] (5,2) -- (5,3);
\draw[loosely dotted] (2,0) -- (2,1);

\foreach \i in {1}
{
	\pgfmathtruncatemacro{\ii}{{\i+1}};

	\draw[orange,->,line width=1] (1.5+2*\i,\i) -- (1.5+2*\i,\i+1);
	\node[orange,draw=none] at (1.5+2*\i,-0.3+\i) {$A_* b_{\i}$};
	\draw[skyblue,->,line width=1] (0.5+2*\i,\i-1) -- (0.5+2*\i,\i+1);
	\node[skyblue,draw=none] at (0.5+2*\i,0.3+\i+1) {$A_* a_{\i}$};
	\draw[reddishpurple,->,line width=1] (2,1.5) -- (2.25,2);
	\draw[reddishpurple,->,line width=1] (2.25,0) -- (3.25,2);
	\draw[reddishpurple,->,line width=1] (3.25,1) -- (3.75,2);
	\draw[reddishpurple,->,line width=1] (3.75,1) -- (4.75,3);
	\draw[reddishpurple,->,line width=1] (4.75,1) -- (5,1.5);
	\node[reddishpurple,draw=none] at (-0.45+2*\i,0.5+\i) {$A_* c_{\i}$};
}
\end{scope}

\draw[->] (4.5,3.6) arc(120:90:4.5);
\node[draw=none] at (5.7,4.35) {$A$};

\draw[->] (10.5,4.5) arc(80:50:4.5);
\node[draw=none] at (12.5,4.5) {cut and paste};

\end{tikzpicture}
\caption{The images of $a_1,b_1,c_1$ under the action of $A$.}
\label{fig:homology_action}
\end{figure}
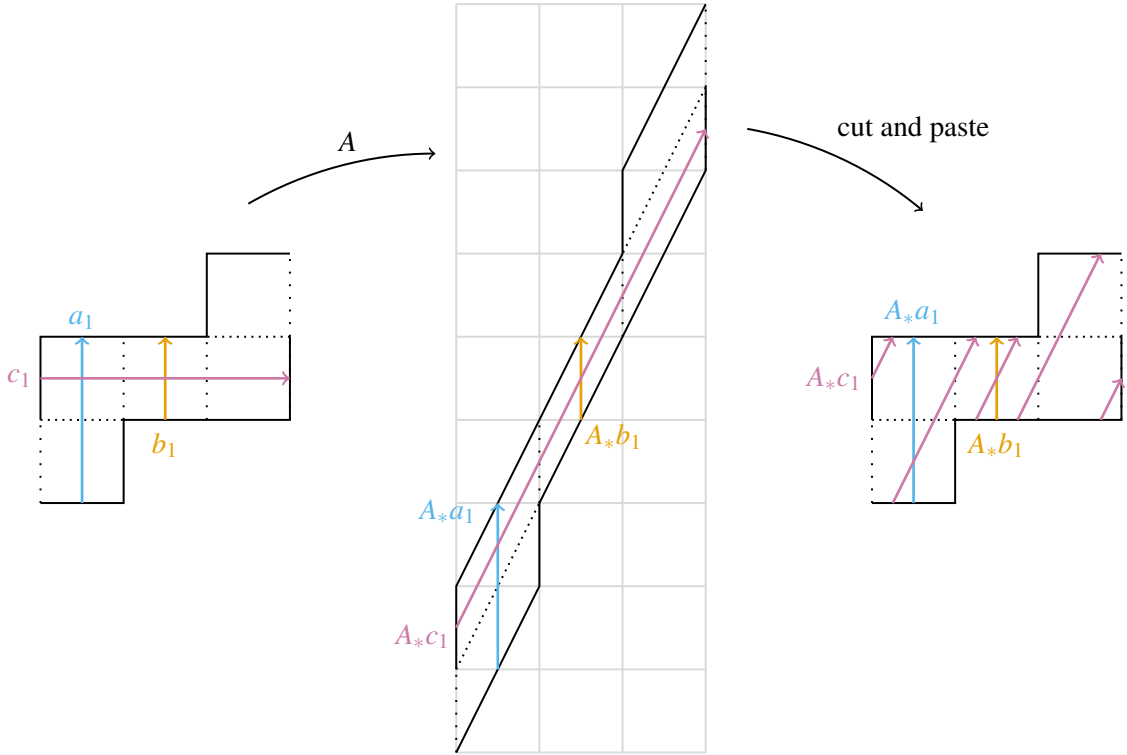

We observe that $A_*[a_i] = [a_i]$, $A_* [b_i] = [b_i]$ and $A_*[c_i] = [a_i] + [a_{i+1}] + 2[b_i] + [c_i]$. Thus if we introduce a formal variable $t$, which acts by Deck transformations, $t \cdot [a_i] = [a_{i+1}]$ etc., we can represent the action of $A_*$ on the homology $H_1(\tS_0,\Z)$ in our chosen basis by the polynomial-valued matrix
\[A_* = \begin{pmatrix}  1 & 0 & 1+t \\ 0 & 1 & 2 \\ 0 & 0 & 1  \end{pmatrix}.\]

Similarly for $B = \begin{pmatrix} 1 & 3 \\ 0 & 1 \end{pmatrix}$ we obtain
\[B_* = \begin{pmatrix}  1 & 0 & 0 \\ 0 & 1 & 0 \\ t^{-1}+1 & 1 & 1  \end{pmatrix},\]

and then the action of $T_*$ on $H_1(\tS_0,\Z)$ is given by the product
\begin{equation}
\label{eq:T_action_staircase}
T_* = B_* A_* = \begin{pmatrix}  1 & 0 & 1+t \\ 0 & 1 & 2 \\ t^{-1}+1 & 1 & t^{-1}+5+t  \end{pmatrix}.
\end{equation}

\begin{remark}
If instead of taking the curves $a_0,b_0,c_0$ above one would take the curves corresponding to the zippered rectangles for the flow in the stable direction of $T$, the matrix for $T_*$ would be precisely the level-counting cocycle as defined in~\cite{Tum}.
\end{remark}

The action of $\cL_z^\#$ on $H^1(\tS_0,z,\C)$ in the basis $(\alpha,\beta,\gamma)$ is then given by the matrix for the $T_*$-action on homology obtained in \eqref{eq:T_action_staircase}, with $t = e^z$. Indeed, let us prove the following general statement.

\begin{lemma}\label{lem:twisted_action}
Let $\{c_i^a : 1\leq i \leq k=2g-2+|\Sigma|,\, a \in \Z^d\}$ be a basis for $H_1(\tS_0,\Z)$, such that for $D \in \Deck$, 
\[D(c_i^a) = c_i^{a+\omega(D)}.\]
For $i=1,\dots,k$, let $\alpha_i \in H^1(\tS,z,\C)$ be such that 
\[\alpha_i(c_j^a) = \delta_{ij} e^{z\cdot a}.\] 

Define for $i,j \in \{1,\dots,k\}$ the Laurent polynomial in the variables $t = (t_1,\dots,t_d)$, 
$M_{ij}(t)~=~\sum_{a\in \Z^d} m_{ij,a}\,t^a$, so that
\[\tT_*(c_j^0) = \sum_{a\in \Z^d} \sum_{i=1}^k m_{ij,a}\, c_i^a.\]

For $z\in \C^d$, define the matrix $M_z$ by $(M_z)_{ij} = M_{ij}(e^z).$
Then the $\cL_z^\#$ action on $H^1(\tS,z,\C)$ is given by the matrix $M_z$, in the sense that 
\[\alpha_i \circ \tT = \sum_{j=1}^k (M_z)_{ij}  \alpha_j.\]
\end{lemma}
\begin{proof}
We compute for any $j=1,\dots,k$, and $b\in \Z^d$,
\begin{align*}
\alpha_i \circ \tT (c_j^b) &= \alpha_i \left(\sum_{a\in \Z^d} m_{ij,a} c_i^{a+b}\right)
= \sum_{a\in \Z^d} m_{ij,a} e^{z\cdot(a+b)}
= e^{z\cdot b} M_{ij}(e^{z}).
\end{align*}

\end{proof}



\subsubsection{Other staircases}
One can consider similar staircase constructions with different parameters, for example a $4\times 1$ staircase starting from a genus 2 surface in the stratum $\mathcal{H}(1,1)$, as shown in \Cref{fig:4staircaseS}.

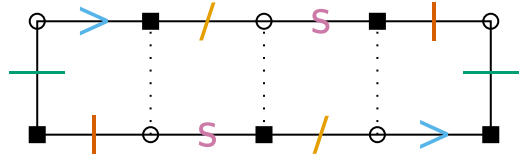
\begin{figure}[h]
\begin{tikzpicture}[square/.style={regular polygon,regular polygon sides=4,inner sep=2,fill}, every node/.style={draw}, 
		scale=1.5]
\draw (0,0) -- (0,1) -- (4,1) --  (4,0) -- cycle;

\foreach \j in {1,2,3}
	{
		\draw[loosely dotted] ({\j},0) -- ({\j},+1);
	}
\node[square] at (0,0) {};
\node[square] at (2,0) {};
\node[square] at (4,0) {};
\node[square] at (1,1) {};
\node[square] at (3,1) {};

\node[circle,inner sep=2] at (0,1) {};
\node[circle,inner sep=2] at (2,1) {};
\node[circle,inner sep=2] at (4,1) {};
\node[circle,inner sep=2] at (1,0) {};
\node[circle,inner sep=2] at (3,0) {};

\node[vermillion,draw=none] at (0.5,0) {\huge |};
\node[vermillion,draw=none] at (3.5,1) {\huge|};

\node[reddishpurple,draw=none] at (1.5,0) {\huge \sffamily s};
\node[reddishpurple,draw=none] at (2.5,1) {\huge\sffamily s};

\node[orange,draw=none] at (2.5,0) {\huge /};
\node[orange,draw=none] at (1.5,1) {\huge /};

\node[skyblue,draw=none] at (3.5,0) {\huge >};
\node[skyblue,draw=none] at (0.5,1) {\huge >};

\node[bluishgreen,draw=none] at (0,0.5) {\huge ---};
\node[bluishgreen,draw=none] at (4,0.5) {\huge ---};

\end{tikzpicture}
\caption{The genus 2 surface $S$.}
\label{fig:4staircaseS}
\end{figure}

Consider the cover $\tS$ as in \Cref{fig:4staircase}. In the figure pairs of squares with
the same label form vertical cylinders of modulus $\frac12$: the top edge of one
square is glued to the bottom edge of the other one and vice versa.

\begin{figure}[h]
\begin{tikzpicture}[square/.style={regular polygon,regular polygon sides=4,inner sep=2,fill}, every node/.style={draw}]
\draw (1,0) -- (4,0) -- (4,1) --  (7,1) -- (7,2) -- (10,2) -- (10,3) -- (13,3) -- (13,4); 
\draw (0,0) -- (0,1) -- (3,1) --  (3,2) -- (6,2) -- (6,3) -- (9,3) -- (9,4) -- (12,4); 
\foreach \i in {0,...,3}
{
	\draw[dotted] ({3*\i},\i) -- ({3*\i+1},\i);
	\foreach \j in {1,2,3}
		{
        	\draw[loosely dotted] ({\j+3*\i},\i) -- ({\j+3*\i},\i+1);
        }
    \node[draw=none] at ({3*\i+1.5},{\i+0.5}) {$\i$};
    \pgfmathtruncatemacro{\ii}{{\i+1}};
    \node[draw=none] at ({3*\i+2.5},{\i+0.5}) {$\ii$};
}
\draw[dotted] (12,4) -- (13,4);

\foreach \i in {1,...,2}
{
	\pgfmathtruncatemacro{\ii}{{\i+1}};
    \node[reddishpurple,draw=none] at ({1.5+3*\i},\i) { \sffamily \rep{\i}{s\hspace{-1pt}}};
	\node[reddishpurple,draw=none] at ({2.5+3*\i},{\i+1}) {\sffamily \rep{\ii}{s\hspace{-1pt}}};
	
	\node[orange,draw=none] at (2.5+3*\i,\i) {\large \rep{\ii}{/}};
	\node[orange,draw=none] at (1.5+3*\i,\i+1) {\large \rep{\i}{/}};
	
	\node[skyblue,draw=none] at ({3.5+3*\i},\i) { \rep{\ii}{>}};
	\node[skyblue,draw=none] at ({0.5+3*\i},\i+1) { \rep{\i}{>}};
}

\node[reddishpurple,draw=none] at (2.5,1) {\sffamily s};
\node[orange,draw=none] at (2.5,0) {/};
\node[skyblue,draw=none] at (3.5,0) {>};

\node[reddishpurple,draw=none] at (1.5+9,3) {\sffamily \rep{3}{s\hspace{-1pt}}};
\node[orange,draw=none] at (1.5+9,4) {\rep{3}{/}};
\node[skyblue,draw=none] at (0.5+9,4) {\rep{3}{>}};

\node[bluishgreen,draw=none] at (3,1.5) {--};
\node[bluishgreen,draw=none] at (7,1.5) {--};
\node[bluishgreen,draw=none] at (6,2.5) {=};
\node[bluishgreen,draw=none] at (10,2.5) {=};
\node[bluishgreen,draw=none] at (9,3.5) {$\equiv$};
\node[bluishgreen,draw=none] at (13,3.5) {$\equiv$};

\node[square,vermillion] at (3,1) {};
\node[square,vermillion] at (5,1) {};
\node[square,vermillion] at (7,1) {};
\node[square,vermillion] at (7,3) {};

\node[square,yellow] at (6,2) {};
\node[square,yellow] at (8,2) {};
\node[square,yellow] at (10,2) {};
\node[square,yellow] at (10,4) {};

\node[circle,inner sep=3,blue] at (3,2) {};
\node[circle,inner sep=3,blue] at (3,0) {};
\node[circle,inner sep=3,blue] at (5,2) {};
\node[circle,inner sep=3,blue] at (7,2) {};

\node[circle,inner sep=3,black] at (6,3) {};
\node[circle,inner sep=3,black] at (6,1) {};
\node[circle,inner sep=3,black] at (8,3) {};
\node[circle,inner sep=3,black] at (10,3) {};

\end{tikzpicture}
\caption{The cover $\tS$ with a few of its singularities.}
\label{fig:4staircase}
\end{figure}
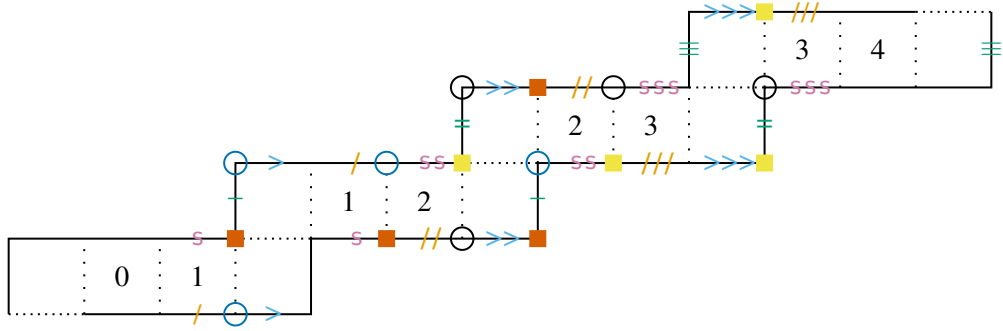

The cover $\tS$ has two cylinder decompositions, with horizontal
cylinders of modulus $\frac14$ and vertical cylinders of modulus $\frac12$, with corresponding Dehn twists $\begin{pmatrix}
1 & 4 \\ 0 & 1\end{pmatrix}$ and $\begin{pmatrix} 1 & 0 \\ 2 & 1\end{pmatrix}$.

\subsection{Example: the wind-tree model, a $\Z^2$ cover}

The Ehrenfest wind-tree model is a billiard in the plane with rectangular obstacles. First introduced by P. Ehrenfest and T. Ehrenfest in 1911 \cite{Ehrenfest} to model gas diffusion, the periodic version as first studied by Hardy and Weber \cite{HardyWeber} has attracted a lot of attention in the last decade and a half, see for example \cite{DHL,FU}. The billiard table for parameters $0<a,b<1$ is defined as $\R^2$ with removed obstacles, which are translates of the rectangle $[0,a] \times [0,b]$, centered at each element of $\Z^2$.

The unfolding of the wind-tree model is a $\Z^2$-cover  $\tilde X$ of the
genus 5 compact translation surface $X$, obtained by gluing 4 copies of $[0,1]^2 \setminus (\frac{1-a}{2},\frac{1+a}{2}) \times (\frac{1-b}{2},\frac{1+b}{2})$. 

While the results of \cite{FU} show that for typical directions on any wind-tree model, the translation flow is not ergodic, there are special parameters, including all rational $a,b$, for which the cover $\tilde X$ does admit pseudo-Anosovs and hence fits our setting. We will consider the case of $a = b  = 1/2$, as shown in \Cref{fig:X}.

\begin{figure}[h]
\centering
\begin{tikzpicture}[scale=2.5,line width=1]
\begin{scope}[decoration={markings, mark=at position 0.6 with {\arrow{>}}}]

\fill[gray!15] (0,0) -- (1,0) -- (1,1) -- (0,1) -- (0,0) -- (0.25,0.25) -- (0.25,0.75) -- (0.75,0.75) -- (0.75,0.25) -- (0.25,0.25) -- cycle;
\draw[skyblue, postaction={decorate}] (0,0) -- (1,0);
\draw[skyblue, postaction={decorate}] (0,1) -- (1,1);
\draw (0.25,0.25) -- (0.75,0.25);
\draw (0.25,0.75) -- (0.75,0.75);
\draw[orange, postaction={decorate}] (0,0) -- (0,1);
\draw[orange, postaction={decorate}] (1,0) -- (1,1);
\draw (0.25,0.25) -- (0.25,0.75);
\draw (0.75,0.25) -- (0.75,0.75);

\node[circle,fill,inner sep=0.9mm,vermillion] at (0.25,0.25) {};
\node[rectangle,fill,inner sep=1mm,reddishpurple] at (0.75,0.25) {};
\node[diamond,fill,inner sep=0.75mm,bluishgreen] at (0.25,0.75) {};
\node[star,fill,inner sep=0.75mm] at (0.75,0.75) {};

\node[skyblue] at (0.6,-0.1) {$h_{00}$};
\node[skyblue] at (0.6,0.9) {$h_{00}$};
\node[orange] at (-0.1,0.15) {$v_{00}$};
\node[orange] at (0.9,0.15) {$v_{00}$};

\draw[bluishgreen, postaction={decorate}] (0.4,0) -- (0.4,0.25);
\draw[bluishgreen] (0.4,0.75) -- (0.4,1);
\node[bluishgreen] at (0.255,0.15) {$c_{v,0}$}; 

\draw[blue, postaction={decorate}] (0,0.5) -- (0.25,0.5);
\draw[blue,] (0.75,0.5) -- (1,0.5);
\node[blue] at (0.1,0.4) {$c_{h,0}$}; 

\begin{scope}[shift={(1.3,0)}]
\fill[gray!15] (0,0) -- (1,0) -- (1,1) -- (0,1) -- (0,0) -- (0.25,0.25) -- (0.25,0.75) -- (0.75,0.75) -- (0.75,0.25) -- (0.25,0.25) -- cycle;
\draw[skyblue, postaction={decorate}] (0,0) -- (1,0);
\draw[skyblue, postaction={decorate}] (0,1) -- (1,1);
\draw (0.25,0.25) -- (0.75,0.25);
\draw (0.25,0.75) -- (0.75,0.75);
\draw[orange, postaction={decorate}] (0,0) -- (0,1);
\draw[orange, postaction={decorate}] (1,0) -- (1,1);
\draw (0.25,0.25) -- (0.25,0.75);
\draw (0.75,0.25) -- (0.75,0.75);

\node[circle,fill,inner sep=0.9mm,vermillion] at (0.75,0.25) {};
\node[rectangle,fill,inner sep=1mm,reddishpurple] at (0.25,0.25) {};
\node[diamond,fill,inner sep=0.75mm,bluishgreen] at (0.75,0.75) {};
\node[star,fill,inner sep=0.75mm] at (0.25,0.75) {};

\node[skyblue] at (0.6,-0.1) {$h_{10}$};
\node[skyblue] at (0.6,0.9) {$h_{10}$};
\node[orange] at (0.1,0.15) {$v_{10}$};
\node[orange] at (1.1,0.15) {$v_{10}$};

\draw[bluishgreen, postaction={decorate}] (0.4,0) -- (0.4,0.25);
\draw[bluishgreen] (0.4,0.75) -- (0.4,1);
\node[bluishgreen] at (0.55,0.15) {$c_{v,1}$}; 

\draw[blue,] (0,0.5) -- (0.25,0.5);
\draw[blue,] (0.75,0.5) -- (1,0.5);
\end{scope}

\begin{scope}[shift={(1.3,1.3)}]
\fill[gray!15] (0,0) -- (1,0) -- (1,1) -- (0,1) -- (0,0) -- (0.25,0.25) -- (0.25,0.75) -- (0.75,0.75) -- (0.75,0.25) -- (0.25,0.25) -- cycle;
\draw[skyblue, postaction={decorate}] (0,0) -- (1,0);
\draw[skyblue, postaction={decorate}] (0,1) -- (1,1);
\draw (0.25,0.25) -- (0.75,0.25);
\draw (0.25,0.75) -- (0.75,0.75);
\draw[orange, postaction={decorate}] (0,0) -- (0,1);
\draw[orange, postaction={decorate}] (1,0) -- (1,1);
\draw (0.25,0.25) -- (0.25,0.75);
\draw (0.75,0.25) -- (0.75,0.75);

\node[circle,fill,inner sep=0.9mm,vermillion] at (0.75,0.75) {};
\node[rectangle,fill,inner sep=1mm,reddishpurple] at (0.25,0.75) {};
\node[diamond,fill,inner sep=0.75mm,bluishgreen] at (0.75,0.25) {};
\node[star,fill,inner sep=0.75mm] at (0.25,0.25) {};

\node[skyblue] at (0.6,0.1) {$h_{11}$};
\node[skyblue] at (0.6,1.1) {$h_{11}$};
\node[orange] at (0.1,0.15) {$v_{11}$};
\node[orange] at (1.1,0.15) {$v_{11}$};

\draw[bluishgreen] (0.4,0) -- (0.4,0.25);
\draw[bluishgreen] (0.4,0.75) -- (0.4,1);

\draw[blue,] (0,0.5) -- (0.25,0.5);
\draw[blue,] (0.75,0.5) -- (1,0.5);
\end{scope}

\begin{scope}[shift={(0,1.3)}]
\fill[gray!15] (0,0) -- (1,0) -- (1,1) -- (0,1) -- (0,0) -- (0.25,0.25) -- (0.25,0.75) -- (0.75,0.75) -- (0.75,0.25) -- (0.25,0.25) -- cycle;
\draw[skyblue, postaction={decorate}] (0,0) -- (1,0);
\draw[skyblue, postaction={decorate}] (0,1) -- (1,1);
\draw (0.25,0.25) -- (0.75,0.25);
\draw (0.25,0.75) -- (0.75,0.75);
\draw[orange, postaction={decorate}] (0,0) -- (0,1);
\draw[orange, postaction={decorate}] (1,0) -- (1,1);
\draw (0.25,0.25) -- (0.25,0.75);
\draw (0.75,0.25) -- (0.75,0.75);

\node[circle,fill,inner sep=0.9mm,vermillion] at (0.25,0.75) {};
\node[rectangle,fill,inner sep=1mm,reddishpurple] at (0.75,0.75) {};
\node[diamond,fill,inner sep=0.75mm,bluishgreen] at (0.25,0.25) {};
\node[star,fill,inner sep=0.75mm] at (0.75,0.25) {};

\node[skyblue] at (0.6,0.1) {$h_{01}$};
\node[skyblue] at (0.6,1.1) {$h_{01}$};
\node[orange] at (-0.1,0.15) {$v_{01}$};
\node[orange] at (0.9,0.15) {$v_{01}$};

\draw[bluishgreen] (0.4,0) -- (0.4,0.25);
\draw[bluishgreen] (0.4,0.75) -- (0.4,1);

\draw[blue, postaction={decorate}] (0,0.5) -- (0.25,0.5);
\draw[blue,] (0.75,0.5) -- (1,0.5);
\node[blue] at (0.1,0.4) {$c_{h,1}$}; 
\end{scope}

\end{scope}

\end{tikzpicture}

\caption{The surface $X$ with its four singularities and the classes generating $H_1(X,\Z)$.}
\label{fig:X}
\end{figure}
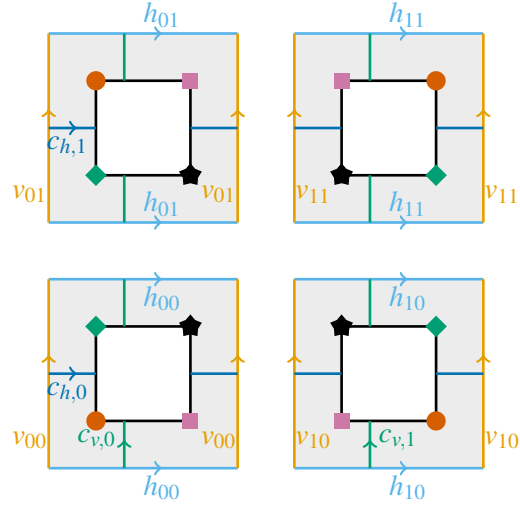

Define the homology classes $h_{xy}, v_{xy} \in H_1(X,\Z)$ for $x,y\in \{0,1\}$ as in \Cref{fig:X}.
Define 
\begin{align*}
\gamma_h &\coloneqq -v_{00}+v_{10}-v_{01}+v_{11},\\
\gamma_v &\coloneqq h_{00}+h_{10}-h_{01}-h_{11},
\end{align*}
and let $\tilde X$ be the $\Z^2$-cover associated to $\Gamma^{\mathrm{ab}} = H_1(X,\Z)/ \langle\gamma_h,\gamma_v\rangle$, as shown in \Cref{fig:Xtilde}.

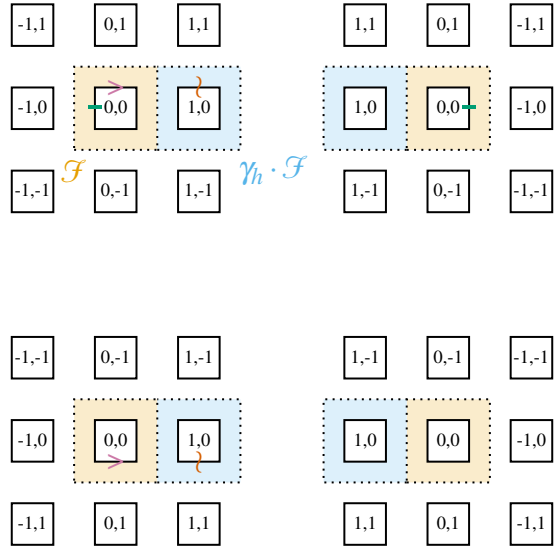
\begin{figure}[h]
\centering

\begin{tikzpicture}[scale=1.1]
\begin{scope}[decoration={markings, mark=at position 0.6 with {\arrow{>}}}]

\fill[orange!20] (-0.25,-0.25) -- (0.75,-0.25) -- (0.75,0.75) -- (-0.25,0.75)  -- (-0.25,-0.25) -- (0,0) -- (0,0.5) -- (0.5,0.5) -- (0.5,0) -- (0,0) -- cycle;
\fill[skyblue!20] (0.75,-0.25) -- (1.75,-0.25) -- (1.75,0.75) -- (0.75,0.75) -- (0.75,-0.25) -- (1,0) -- (1,0.5) -- (1.5,0.5) -- (1.5,0) -- (1,0) -- cycle;

\foreach \x in {-1,...,1}{
	\foreach \y in {-1,...,1}{
		\draw (\x,\y) -- ++(0.5,0) -- ++ (0,0.5) -- ++ (-0.5,0) -- cycle;
		\node at (\x+0.25,\y+0.25) {\tiny \x,\y};	
	}
}

\draw[dotted] (-0.25,-0.25) -- (0.75,-0.25) -- (0.75,0.75) -- (-0.25,0.75) -- cycle;
\draw[dotted] (0.75,-0.25) -- (1.75,-0.25) -- (1.75,0.75) -- (0.75,0.75);

\node[orange,font=\large] at (-0.25,-0.5) {$\mathcal{F}$};
\node[skyblue,font=\large] at (2.15,-0.5) {$\gamma_h \cdot \mathcal{F}$};

\node[reddishpurple] at (0.25,0.5) {>};
\node[bluishgreen,font=\huge] at (0,0.2) {-}; 
\node[vermillion,font=\large] at (1.25,0.5) {$\wr$}; 

\begin{scope}[shift={(4,0)}]
\fill[orange!20] (-0.25,-0.25) -- (0.75,-0.25) -- (0.75,0.75) -- (-0.25,0.75)  -- (-0.25,-0.25) -- (0,0) -- (0,0.5) -- (0.5,0.5) -- (0.5,0) -- (0,0) -- cycle;
\begin{scope}[shift={(-2,0)}]
\fill[skyblue!20] (0.75,-0.25) -- (1.75,-0.25) -- (1.75,0.75) -- (0.75,0.75) -- (0.75,-0.25) -- (1,0) -- (1,0.5) -- (1.5,0.5) -- (1.5,0) -- (1,0) -- cycle;
\end{scope}
\foreach \x in {-1,...,1}{
	\foreach \y in {-1,...,1}{
		\draw (\x,\y) -- ++(0.5,0) -- ++ (0,0.5) -- ++ (-0.5,0) -- cycle;
		\node at (-\x+0.25,\y+0.25) {\tiny \x,\y};	
	}
}
\draw[dotted] (-0.25,-0.25) -- (0.75,-0.25) -- (0.75,0.75) -- (-0.25,0.75) -- cycle;
\draw[dotted] (-0.25,-0.25) -- (-1.25,-0.25) -- (-1.25,0.75) -- (-0.25,0.75);

\node[bluishgreen,font=\huge] at (0.5,0.2) {-}; 
\end{scope}

\begin{scope}[shift={(0,-4)}]
\fill[orange!20] (-0.25,-0.25) -- (0.75,-0.25) -- (0.75,0.75) -- (-0.25,0.75)  -- (-0.25,-0.25) -- (0,0) -- (0,0.5) -- (0.5,0.5) -- (0.5,0) -- (0,0) -- cycle;
\fill[skyblue!20] (0.75,-0.25) -- (1.75,-0.25) -- (1.75,0.75) -- (0.75,0.75) -- (0.75,-0.25) -- (1,0) -- (1,0.5) -- (1.5,0.5) -- (1.5,0) -- (1,0) -- cycle;

\foreach \x in {-1,...,1}{
	\foreach \y in {-1,...,1}{
		\draw (\x,\y) -- ++(0.5,0) -- ++ (0,0.5) -- ++ (-0.5,0) -- cycle;
		\node at (\x+0.25,-\y+0.25) {\tiny \x,\y};	
	}
}
\draw[dotted] (-0.25,-0.25) -- (0.75,-0.25) -- (0.75,0.75) -- (-0.25,0.75) -- cycle;
\draw[dotted] (0.75,-0.25) -- (1.75,-0.25) -- (1.75,0.75) -- (0.75,0.75);
\node[reddishpurple] at (0.25,0) {>};
\node[vermillion,font=\large] at (1.25,0) {$\wr$}; 
\end{scope}

\begin{scope}[shift={(4,-4)}]
\fill[orange!20] (-0.25,-0.25) -- (0.75,-0.25) -- (0.75,0.75) -- (-0.25,0.75)  -- (-0.25,-0.25) -- (0,0) -- (0,0.5) -- (0.5,0.5) -- (0.5,0) -- (0,0) -- cycle;
\begin{scope}[shift={(-2,0)}]
\fill[skyblue!20] (0.75,-0.25) -- (1.75,-0.25) -- (1.75,0.75) -- (0.75,0.75) -- (0.75,-0.25) -- (1,0) -- (1,0.5) -- (1.5,0.5) -- (1.5,0) -- (1,0) -- cycle;
\end{scope}
\foreach \x in {-1,...,1}{
	\foreach \y in {-1,...,1}{
		\draw (\x,\y) -- ++(0.5,0) -- ++ (0,0.5) -- ++ (-0.5,0) -- cycle;
		\node at (-\x+0.25,-\y+0.25) {\tiny \x,\y};	
	}
}
\draw[dotted] (-0.25,-0.25) -- (0.75,-0.25) -- (0.75,0.75) -- (-0.25,0.75) -- cycle;
\draw[dotted] (-0.25,-0.25) -- (-1.25,-0.25) -- (-1.25,0.75) -- (-0.25,0.75);
\end{scope}

\end{scope}
\end{tikzpicture}

\caption{A part of the cover $\tilde X$, a fundamental domain $\mathcal{F}$ for the $\Z^2$ Deck group action and its image under the action of the Deck group element $\gamma_h \Gamma^{\mathrm{ab}}$. The labeled squares represent `holes' which are not part of the surface, the edges of the square holes with equal labels are glued to one other.}
\label{fig:Xtilde}
\end{figure}

As $X$ is of genus 5 and has 4 singularities, for $z\neq 0$, $H^1(\tilde X_0, -z, \C) \cong \C^{12}$.
The homology $H_1(\tilde X_0,\Z)$ is generated by the Deck group translates of the 12 classes $\{[a_h],[a_v],[b_h],[b_v],[c]$, $ [d_1], [d_2], [d_3], [e_1], [e_2], [e_3], [e_4]\}$, whose representatives are depicted in \Cref{fig:windtree_curves}.

For a curve $\gamma$ on $\tilde X$, we denote by $\gamma^{(a,b)}$ with $(a,b) \in \Z^2$ the image of $\gamma$ by the Deck transformation $D_{(a,b)}$, i.e., $\gamma^{(a,b)} = (D_{(a,b)})(\gamma)$.

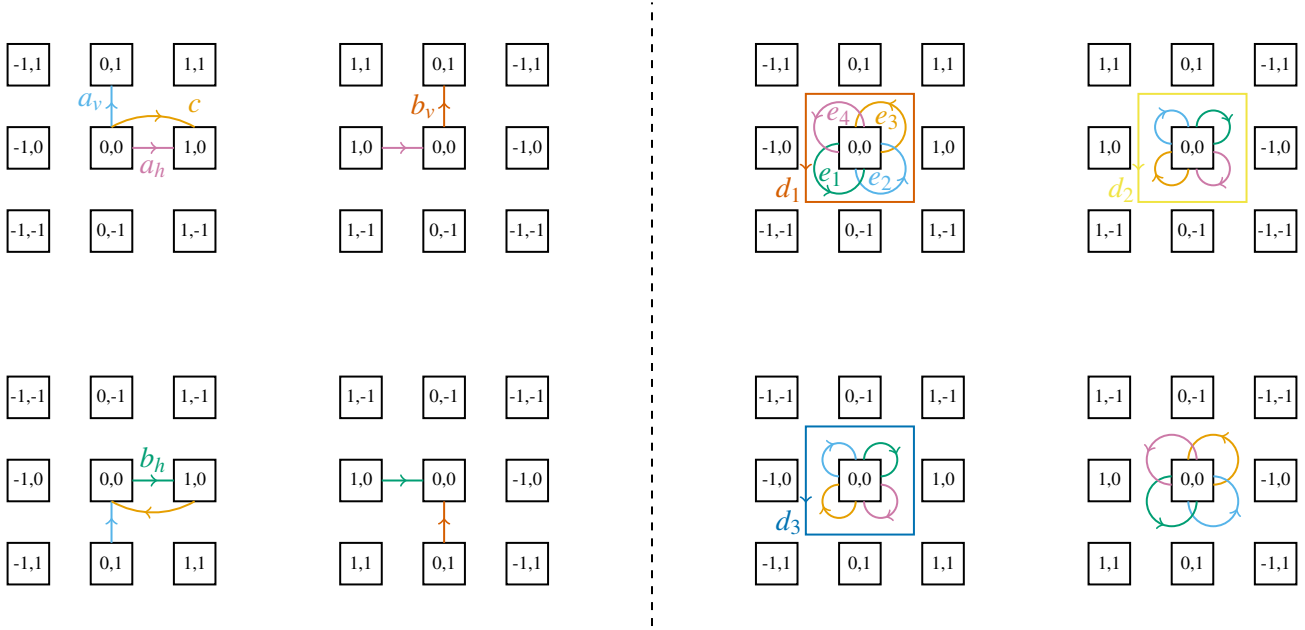
\begin{figure}[h]
\makebox[\textwidth][c]{
\begin{tikzpicture}[scale=1.1]
\begin{scope}[decoration={markings, mark=at position 0.6 with {\arrow{>}}}]

\foreach \x in {-1,...,1}{
	\foreach \y in {-1,...,1}{
		\draw (\x,\y) -- ++(0.5,0) -- ++ (0,0.5) -- ++ (-0.5,0) -- cycle;
		\node at (\x+0.25,\y+0.25) {\tiny \x,\y};	
	}
}

\begin{scope}[shift={(4,0)}]
\foreach \x in {-1,...,1}{
	\foreach \y in {-1,...,1}{
		\draw (\x,\y) -- ++(0.5,0) -- ++ (0,0.5) -- ++ (-0.5,0) -- cycle;
		\node at (-\x+0.25,\y+0.25) {\tiny \x,\y};	
	}
}
\end{scope}

\begin{scope}[shift={(0,-4)}]
\foreach \x in {-1,...,1}{
	\foreach \y in {-1,...,1}{
		\draw (\x,\y) -- ++(0.5,0) -- ++ (0,0.5) -- ++ (-0.5,0) -- cycle;
		\node at (\x+0.25,-\y+0.25) {\tiny \x,\y};	
	}
}
\end{scope}

\begin{scope}[shift={(4,-4)}]
\foreach \x in {-1,...,1}{
	\foreach \y in {-1,...,1}{
		\draw (\x,\y) -- ++(0.5,0) -- ++ (0,0.5) -- ++ (-0.5,0) -- cycle;
		\node at (-\x+0.25,-\y+0.25) {\tiny \x,\y};	
	}
}
\end{scope}

\draw[skyblue,postaction=decorate] (0.25,0.5) -- ++ (0,0.5);
\draw[skyblue,postaction=decorate] (0.25,-4.5) -- ++ (0,0.5);
\node[skyblue] at (0,0.8) {$a_v$};

\draw[vermillion,postaction=decorate] (4.25,0.5) -- ++ (0,0.5);
\draw[vermillion,postaction=decorate] (4.25,-4.5) -- ++ (0,0.5);
\node[vermillion] at (4,0.75) {$b_v$};

\draw[reddishpurple,postaction=decorate] (0.5,0.25) -- ++ (0.5,0);
\draw[reddishpurple,postaction=decorate] (3.5,0.25) -- ++ (0.5,0);
\node[reddishpurple] at (0.75,0) {$a_h$};

\draw[bluishgreen,postaction=decorate] (0.5,-3.75) -- ++ (0.5,0);
\draw[bluishgreen,postaction=decorate] (3.5,-3.75) -- ++ (0.5,0);
\node[bluishgreen] at (0.75,-3.5) {$b_h$};

\draw[orange,postaction=decorate] (0.25,0.5) arc (120:60:1);
\draw[orange,postaction=decorate] (1.25,-4) arc (-60:-120:1);
\node[orange] at (1.25,0.75) {$c$};

\draw[dashed] (6.75,-5.5) -- (6.75,2);

\begin{scope}[shift={(9,0)}]

\foreach \x in {-1,...,1}{
	\foreach \y in {-1,...,1}{
		\draw (\x,\y) -- ++(0.5,0) -- ++ (0,0.5) -- ++ (-0.5,0) -- cycle;
		\node at (\x+0.25,\y+0.25) {\tiny \x,\y};	
	}
}

\draw[orange,postaction=decorate] (0.5,0.2) arc (-90:180:0.3);
\draw[skyblue,postaction=decorate] (0.2,0) arc (-180:90:0.3);
\draw[bluishgreen,postaction=decorate] (0,0.3) arc (90:360:0.3);
\draw[reddishpurple,postaction=decorate] (0.3,0.5) arc (0:270:0.3); 

\draw[vermillion,postaction=decorate] (0.9,0) -- (0.9,0.9) -- (-0.4,0.9) -- (-0.4,-0.4) -- (0.9,-0.4) -- cycle;

\node[orange] at (0.58,0.58) {$e_3$}; 
\node[bluishgreen] at (-0.1,-0.1) {$e_1$};
\node[skyblue] at (0.5,-0.15) {$e_2$}; 
\node[reddishpurple] at (0,0.63) {$e_4$};  
\node[vermillion] at (-0.6,-0.25) {$d_1$};

\begin{scope}[shift={(4,0)}]
\foreach \x in {-1,...,1}{
	\foreach \y in {-1,...,1}{
		\draw (\x,\y) -- ++(0.5,0) -- ++ (0,0.5) -- ++ (-0.5,0) -- cycle;
		\node at (-\x+0.25,\y+0.25) {\tiny \x,\y};	
	}
}

\draw[bluishgreen,postaction=decorate] (0.3,0.5) arc (180:-90:0.2);
\draw[skyblue,postaction=decorate] (0,0.3) arc (270:0:0.2);
\draw[orange,postaction=decorate] (0.2,0) arc (360:90:0.2);
\draw[reddishpurple,postaction=decorate] (0.5,0.2) arc (90:-180:0.2);

\draw[yellow,postaction=decorate] (0.9,0) -- (0.9,0.9) -- (-0.4,0.9) -- (-0.4,-0.4) -- (0.9,-0.4) -- cycle;
\node[yellow] at (-0.6,-0.25) {$d_2$};

\end{scope}

\begin{scope}[shift={(0,-4)}]
\foreach \x in {-1,...,1}{
	\foreach \y in {-1,...,1}{
		\draw (\x,\y) -- ++(0.5,0) -- ++ (0,0.5) -- ++ (-0.5,0) -- cycle;
		\node at (\x+0.25,-\y+0.25) {\tiny \x,\y};	
	}
}

\draw[bluishgreen,postaction=decorate] (0.3,0.5) arc (180:-90:0.2);
\draw[skyblue,postaction=decorate] (0,0.3) arc (270:0:0.2);
\draw[orange,postaction=decorate] (0.2,0) arc (360:90:0.2);
\draw[reddishpurple,postaction=decorate] (0.5,0.2) arc (90:-180:0.2);

\draw[blue,postaction=decorate] (0.9,0) -- (0.9,0.9) -- (-0.4,0.9) -- (-0.4,-0.4) -- (0.9,-0.4) -- cycle;
\node[blue] at (-0.6,-0.25) {$d_3$};
\end{scope}

\begin{scope}[shift={(4,-4)}]
\foreach \x in {-1,...,1}{
	\foreach \y in {-1,...,1}{
		\draw (\x,\y) -- ++(0.5,0) -- ++ (0,0.5) -- ++ (-0.5,0) -- cycle;
		\node at (-\x+0.25,-\y+0.25) {\tiny \x,\y};	
	}
}
\draw[orange,postaction=decorate] (0.5,0.2) arc (-90:180:0.3);
\draw[skyblue,postaction=decorate] (0.2,0) arc (-180:90:0.3);
\draw[bluishgreen,postaction=decorate] (0,0.3) arc (90:360:0.3);
\draw[reddishpurple,postaction=decorate] (0.3,0.5) arc (0:270:0.3); 

\end{scope}

\end{scope}

\end{scope}
\end{tikzpicture}
}
\caption{The 12 classes generating $H_1(\tilde X_0,\Z)$ up to Deck transformations.}
\label{fig:windtree_curves}
\end{figure}

The compact surface $X$ can be decomposed into cylinders in the north-east and north-west directions, see \Cref{fig:X_cylinders}.  Let $A$ denote the Dehn twist in both north-east cylinders, and $B$ the Dehn twist in both north-west cylinders. Then $T = B^{-1}A$ is a pseudo-Anosov map on $X$. Since the cylinder decompositions lift to $\tilde X$, $A$ and $B$ both lift to $\tilde X$, corresponding to Dehn twists in all north-east (respectively north-west) cylinders on $\tilde X$. Thus also $T$ lifts to $\tilde X$.

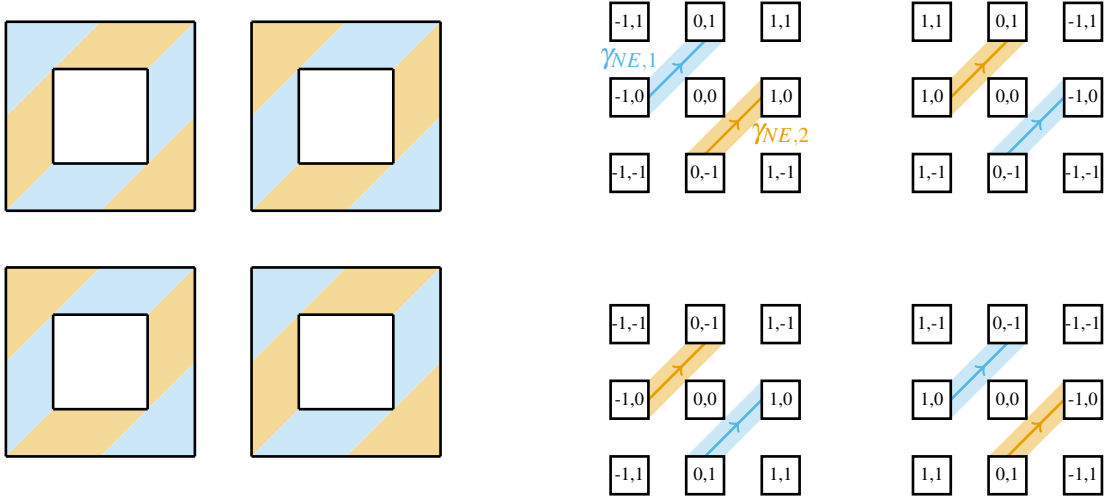
\begin{figure}[h]
\centering
\begin{tikzpicture}[line width=1]
\begin{scope}[scale=2.5, decoration={markings, mark=at position 0.6 with {\arrow{>}}}]

\fill[orange!40] (0,0) -- (0.25,0.25) -- (0.75,0.25) -- (0.5,0) -- cycle;
\fill[orange!40] (0.75,0.25) -- (1,0.5) -- (1,1) -- (0.75,0.75) -- cycle;
\fill[orange!40] (0,0.5) -- (0,1) -- (0.5,1) -- cycle;
\fill[skyblue!30] (0.25,0.75) -- (0.5,1) -- (1,1) -- (0.75,0.75) -- cycle;
\fill[skyblue!30] (0.5,0) -- (1,0.5) -- (1,0) -- cycle;
\fill[skyblue!30] (0,0) -- (0.25,0.25) -- (0.25,0.75) -- (0,0.5) -- cycle;

\draw (0,0) -- (1,0);
\draw (0,1) -- (1,1);
\draw (0.25,0.25) -- (0.75,0.25);
\draw (0.25,0.75) -- (0.75,0.75);
\draw (0,0) -- (0,1);
\draw (1,0) -- (1,1);
\draw (0.25,0.25) -- (0.25,0.75);
\draw (0.75,0.25) -- (0.75,0.75);

\begin{scope}[shift={(1.3,0)}]
\fill[orange!40] (0.25,0.75) -- (0.5,1) -- (1,1) -- (0.75,0.75) -- cycle;
\fill[orange!40] (0.5,0) -- (1,0.5) -- (1,0) -- cycle;
\fill[orange!40] (0,0) -- (0.25,0.25) -- (0.25,0.75) -- (0,0.5) -- cycle;
\fill[skyblue!30] (0,0) -- (0.25,0.25) -- (0.75,0.25) -- (0.5,0) -- cycle;
\fill[skyblue!30] (0.75,0.25) -- (1,0.5) -- (1,1) -- (0.75,0.75) -- cycle;
\fill[skyblue!30] (0,0.5) -- (0,1) -- (0.5,1) -- cycle;

\draw (0,0) -- (1,0);
\draw (0,1) -- (1,1);
\draw (0.25,0.25) -- (0.75,0.25);
\draw (0.25,0.75) -- (0.75,0.75);
\draw (0,0) -- (0,1);
\draw (1,0) -- (1,1);
\draw (0.25,0.25) -- (0.25,0.75);
\draw (0.75,0.25) -- (0.75,0.75);

\end{scope}
\begin{scope}[shift={(1.3,1.3)}]
\fill[orange!40] (0,0) -- (0.25,0.25) -- (0.75,0.25) -- (0.5,0) -- cycle;
\fill[orange!40] (0.75,0.25) -- (1,0.5) -- (1,1) -- (0.75,0.75) -- cycle;
\fill[orange!40] (0,0.5) -- (0,1) -- (0.5,1) -- cycle;
\fill[skyblue!30] (0.25,0.75) -- (0.5,1) -- (1,1) -- (0.75,0.75) -- cycle;
\fill[skyblue!30] (0.5,0) -- (1,0.5) -- (1,0) -- cycle;
\fill[skyblue!30] (0,0) -- (0.25,0.25) -- (0.25,0.75) -- (0,0.5) -- cycle;

\draw (0,0) -- (1,0);
\draw (0,1) -- (1,1);
\draw (0.25,0.25) -- (0.75,0.25);
\draw (0.25,0.75) -- (0.75,0.75);
\draw (0,0) -- (0,1);
\draw (1,0) -- (1,1);
\draw (0.25,0.25) -- (0.25,0.75);
\draw (0.75,0.25) -- (0.75,0.75);
\end{scope}
\begin{scope}[shift={(0,1.3)}]
\fill[orange!40] (0.25,0.75) -- (0.5,1) -- (1,1) -- (0.75,0.75) -- cycle;
\fill[orange!40] (0.5,0) -- (1,0.5) -- (1,0) -- cycle;
\fill[orange!40] (0,0) -- (0.25,0.25) -- (0.25,0.75) -- (0,0.5) -- cycle;
\fill[skyblue!30] (0,0) -- (0.25,0.25) -- (0.75,0.25) -- (0.5,0) -- cycle;
\fill[skyblue!30] (0.75,0.25) -- (1,0.5) -- (1,1) -- (0.75,0.75) -- cycle;
\fill[skyblue!30] (0,0.5) -- (0,1) -- (0.5,1) -- cycle;

\draw (0,0) -- (1,0);
\draw (0,1) -- (1,1);
\draw (0.25,0.25) -- (0.75,0.25);
\draw (0.25,0.75) -- (0.75,0.75);
\draw (0,0) -- (0,1);
\draw (1,0) -- (1,1);
\draw (0.25,0.25) -- (0.25,0.75);
\draw (0.75,0.25) -- (0.75,0.75);

\end{scope}

\begin{scope}[scale=0.4, shift={(9,4.5)}]

\fill[skyblue!30] (-0.5,0) -- (-0.5,0.5) -- (0,1) -- (0.5,1) -- cycle;
\fill[orange!40] (0,-0.5) -- (0.5,-0.5) -- (1,0) -- (1,0.5) -- cycle;

\draw[orange,postaction=decorate] (0.25,-0.5) -- (1,0.25);
\draw[skyblue,postaction=decorate] (-0.5,0.25) -- (0.25,1);

\node[skyblue] at (-0.75,0.75) {$\gamma_{NE,1}$};
\node[orange] at (1.25,-0.25) {$\gamma_{NE,2}$};

\foreach \x in {-1,...,1}{
	\foreach \y in {-1,...,1}{
		\draw (\x,\y) -- ++(0.5,0) -- ++ (0,0.5) -- ++ (-0.5,0) -- cycle;
		\node at (\x+0.25,\y+0.25) {\tiny \x,\y};	
	}
}

\begin{scope}[shift={(4,0)}]

\fill[skyblue!30] (0,-0.5) -- (0.5,-0.5) -- (1,0) -- (1,0.5) -- cycle;
\fill[orange!40] (-0.5,0) -- (-0.5,0.5) -- (0,1) -- (0.5,1) -- cycle;

\draw[skyblue,postaction=decorate] (0.25,-0.5) -- (1,0.25);
\draw[orange,postaction=decorate] (-0.5,0.25) -- (0.25,1);

\foreach \x in {-1,...,1}{
	\foreach \y in {-1,...,1}{
		\draw (\x,\y) -- ++(0.5,0) -- ++ (0,0.5) -- ++ (-0.5,0) -- cycle;
		\node at (-\x+0.25,\y+0.25) {\tiny \x,\y};	
	}
}
\end{scope}

\begin{scope}[shift={(0,-4)}]

\fill[skyblue!30] (0,-0.5) -- (0.5,-0.5) -- (1,0) -- (1,0.5) -- cycle;
\fill[orange!40] (-0.5,0) -- (-0.5,0.5) -- (0,1) -- (0.5,1) -- cycle;

\draw[skyblue,postaction=decorate] (0.25,-0.5) -- (1,0.25);
\draw[orange,postaction=decorate] (-0.5,0.25) -- (0.25,1);

\foreach \x in {-1,...,1}{
	\foreach \y in {-1,...,1}{
		\draw (\x,\y) -- ++(0.5,0) -- ++ (0,0.5) -- ++ (-0.5,0) -- cycle;
		\node at (\x+0.25,-\y+0.25) {\tiny \x,\y};	
	}
}
\end{scope}

\begin{scope}[shift={(4,-4)}]

\fill[skyblue!30] (-0.5,0) -- (-0.5,0.5) -- (0,1) -- (0.5,1) -- cycle;
\fill[orange!40] (0,-0.5) -- (0.5,-0.5) -- (1,0) -- (1,0.5) -- cycle;

\draw[orange,postaction=decorate] (0.25,-0.5) -- (1,0.25);
\draw[skyblue,postaction=decorate] (-0.5,0.25) -- (0.25,1);

\foreach \x in {-1,...,1}{
	\foreach \y in {-1,...,1}{
		\draw (\x,\y) -- ++(0.5,0) -- ++ (0,0.5) -- ++ (-0.5,0) -- cycle;
		\node at (-\x+0.25,-\y+0.25) {\tiny \x,\y};	
	}
}
\end{scope}

\end{scope}

\end{scope}
\end{tikzpicture}
\caption{The cylinder decomposition of $X$ in the north-east direction, and the lifts of two cylinders to $\tilde X$, with core curves $\gamma_{NE,1}$ and $\gamma_{NE,2}$.}
\label{fig:X_cylinders}
\end{figure}

We can use the basis shown in \Cref{fig:windtree_curves} to record classes in $H_1(\tilde X_0, \Z)$ as vectors of Laurent polynomials with integer coefficients, i.e., as elements of the ring $\Z[x,x^{-1},y,y^{-1}] ^{12}$, where $x^ay^b$ in the $\gamma$-coordinate represents $\gamma^{(a,b)}$. In this notation the Deck group action by $D_{(a,b)}$ is represented by multiplication by $x^ay^b$.

Denote by $\gamma_{NE,1}$ and $\gamma_{NE,2}$ the two curves depicted in \Cref{fig:X_cylinders}, which are the core curves of two north-east cylinders on $\tilde X$ that project to cover all of $X$. Define $\gamma_{NW,1}$ and $\gamma_{NW,2}$ symmetrically for the north-west cylinders.

For $s \in \{(NE,1),(NE,2),(NW,1),(NW,2)\}$, let $v_s$ be the representation of $[\gamma_s]$ in our basis.
One checks (e.g. see \Cref{fig:curve_NE1} for $\gamma_{NE,1}$) that the following holds:

\[
\begin{array}{llccccccccccr}
v_{NE,1} = & (x^{-1} & 1 & 1 & y^{-1}& 0 & 0 & 0 & 0 & 0 & 0 & 0 & -1),\\
v_{NE,2} = & (1 & y^{-1} & x^{-1} & 1 & 0 & 0 & 0 & 0 & 0 & 1 & 0 & 0),\\
v_{NW,1} = & (-1 & 1 & -x^{-1} & y^{-1} & 0 & 0 & 0 & 0 & 0 & 0 & 1 & 0),\\
v_{NW,2} = & (-x^{-1} & y^{-1} & -1 & 1 & 0 & 0 & 0 & 0 & -1 & 0 & 0 & 0).
\end{array}
\]

\tikzmath{
	\a = 25;
	\r = 0.25/(sin(\a)+cos(\a));
}
\begin{figure}[h]
\centering
\begin{tikzpicture}
\begin{scope}[scale=1.2, decoration={markings, mark=at position 0.75 with {\arrow{>}}}]

\foreach \x in {-1,...,1}{
	\foreach \y in {-1,...,1}{
		\draw (\x,\y) -- ++(0.5,0) -- ++ (0,0.5) -- ++ (-0.5,0) -- cycle;
		\node at (\x+0.25,\y+0.25) {\tiny \x,\y};	
	}
}

\draw[skyblue,postaction=decorate] (-0.5,0.25) -- (0.25,1);
\draw[bluishgreen,postaction=decorate] (-0.5,0.25) -- (0,0.25);
\draw[brown,postaction=decorate] (0.25,0.5) -- (0.25,1);

\draw[reddishpurple,postaction=decorate] (0,0.25) arc (270-\a:\a:\r);

\node[skyblue] at (-0.5,0.75) {$\gamma_{NE,1}$}; 
\node[bluishgreen] at (-0.1,-0.3) {$a_h^{(-1,0)}$};
\node[brown] at (0.7,0.75) {$a_v^{(0,0)}$};

\node[reddishpurple] at (2.2,0.75) {$-e_4^{(0,0)}$};

\begin{scope}[shift={(4,0)}]
\foreach \x in {-1,...,1}{
	\foreach \y in {-1,...,1}{
		\draw (\x,\y) -- ++(0.5,0) -- ++ (0,0.5) -- ++ (-0.5,0) -- cycle;
		\node at (-\x+0.25,\y+0.25) {\tiny \x,\y};	
	}
}

\draw[skyblue,postaction=decorate] (0.25,-0.5) -- (1,0.25);

\draw[bluishgreen,postaction=decorate] (0.5,0.25) -- (1,0.25);
\draw[brown,postaction=decorate] (0.25,-0.5) -- (0.25,0);

\draw[reddishpurple,postaction=decorate] (0.25,0) arc (-180+\a:90-\a:\r);

\node[brown] at (-0.25,-0.25) {$b_v^{(0,-1)}$};

\end{scope}

\begin{scope}[shift={(0,-4)}]
\foreach \x in {-1,...,1}{
	\foreach \y in {-1,...,1}{
		\draw (\x,\y) -- ++(0.5,0) -- ++ (0,0.5) -- ++ (-0.5,0) -- cycle;
		\node at (\x+0.25,-\y+0.25) {\tiny \x,\y};	
	}
}

\draw[skyblue,postaction=decorate] (0.25,-0.5) -- (1,0.25);

\draw[bluishgreen,postaction=decorate] (0.5,0.25) -- (1,0.25);
\draw[brown,postaction=decorate] (0.25,-0.5) -- (0.25,0);

\draw[reddishpurple,postaction=decorate] (0.25,0) arc (-180+\a:90-\a:\r);

\node[bluishgreen] at (0.9,0.6) {$b_h^{(0,0)}$};

\end{scope}

\begin{scope}[shift={(4,-4)}]
\foreach \x in {-1,...,1}{
	\foreach \y in {-1,...,1}{
		\draw (\x,\y) -- ++(0.5,0) -- ++ (0,0.5) -- ++ (-0.5,0) -- cycle;
		\node at (-\x+0.25,-\y+0.25) {\tiny \x,\y};	
	}
}
\draw[skyblue,postaction=decorate] (-0.5,0.25) -- (0.25,1);
\draw[bluishgreen,postaction=decorate] (-0.5,0.25) -- (0,0.25);
\draw[brown,postaction=decorate] (0.25,0.5) -- (0.25,1);

\draw[reddishpurple,postaction=decorate] (0,0.25) arc (270-\a:\a:\r);
\end{scope}

\end{scope}

\end{tikzpicture}
\caption{The class $[\gamma_{NE,1}]$ can be written as $[a_h^{(-1,0)}] + [a_v^{(0,0)}] +[b_h^{(0,0)}] + [b_v^{(0,-1)}] - [e_4^{(0,0)}]$.} 
\label{fig:curve_NE1}
\end{figure}
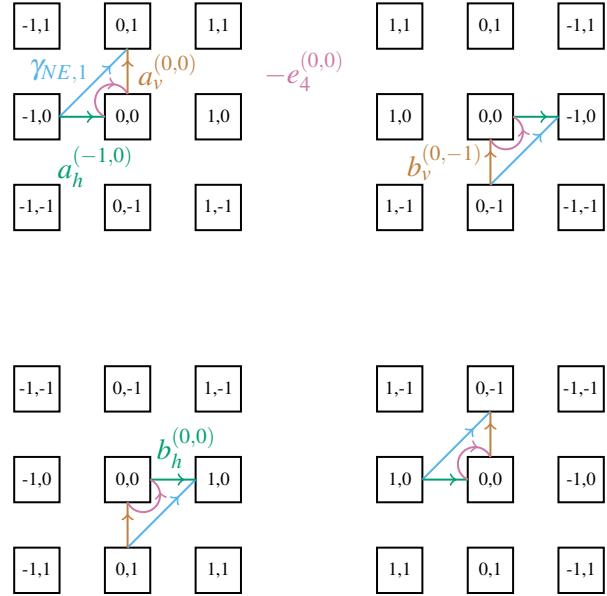

Further, for $s \in \{(NE,1),(NE,2),(NW,1),(NW,2)\}$ one computes the intersection numbers of $\gamma_s$ with the the basis curves. For the basis curve $\delta_i = \delta_i^{(0,0)}$, let 
\[(\sigma_s)_i = \sum_{(a,b)\in \Z^2} x^ay^b \langle \delta_i, \gamma_s^{(a,b)} \rangle,\]

\newgeometry{bottom=2cm} 

so that if $T_s$ is the simultaneous Dehn twist in the union of cylinders with core curves $\gamma_s^{(a,b)}$ for $(a,b)\in \Z^2$, then 
\[T_s(\delta_i) = \delta_i + v_s (\sigma_s)_i.\]
Thus the matrix $M_s$ for the action of $(T_s)_*$ on $H_1(\tilde X_0,\Z)$ is given by 
\[ M_s = I + v_s^T \cdot \sigma_s.\]

One computes that, if $\chi = x+x^{-1}, \upsilon = y+y^{-1}$, then

\[
\hspace{-2cm}
\begin{array}{llllllllllllr}
\sigma_{NE,1} = & (x & -1 & 1 & -y &  -x+y^{-1} & x-y & x^{-1}-y^{-1} &  -1+x &  \upsilon & -\chi+\upsilon & -\chi & 0),\\
\sigma_{NE,2} = & (1 & -y & x & -1 &  x^{-1}-y & -x^{-1}+y^{-1} & -x+y &  0 &  \chi & 0 & -\upsilon & \chi-\upsilon),\\
\sigma_{NW,1} = & (1 & 1 & x & y &  -x^{-1}+y^{-1} & x^{-1}-y & x-y^{-1} &  1-x &  -\chi+\upsilon & \upsilon & 0 & -\chi),\\
\sigma_{NW,2} = & (x & y & 1 & 1 &  x-y & -x+y^{-1} & -x^{-1}+y &  0 &  0 & \chi & \chi-\upsilon & -\upsilon).
\end{array}
\]
Finally we can compute the linear action on the twisted cohomology.

Let $M = (M_{NW,1} M_{NW,2})^{-1} M_{NE,1} M_{NE,2} $ be the matrix representing the lift to $\tilde X$ of the pseudo-Anosov with derivative
\[\begin{pmatrix} 7 & 6 \\ -6 & -5 \end{pmatrix} \begin{pmatrix} 7 & -6 \\ 6 & -5 \end{pmatrix} ,\]
which is the composition of a positive Dehn twist in the NE cylinders followed by a negative Dehn twist in the NW cylinders.

By \Cref{lem:twisted_action}, for $z = (z_1,z_2) \in \R^2$, to obtain the action of this pseudo-Anosov on $H^1_{-z}(X_0,\C) \cong H^1(\tilde X_0,z,\C)$, we substitute $x = e^{z_1}, y = e^{z_2}$ into the matrix $M$.

The top eigenvalue of $M (e^{z_1},e^{z_2})$ is $\lambda \rho((z_1,z_2))$ where the stretch factor of the pseudo-Anosov is $\lambda = 73+12\sqrt{37}$.

\begin{figure}[h!]
\centering
\includegraphics[width=0.6\textwidth]{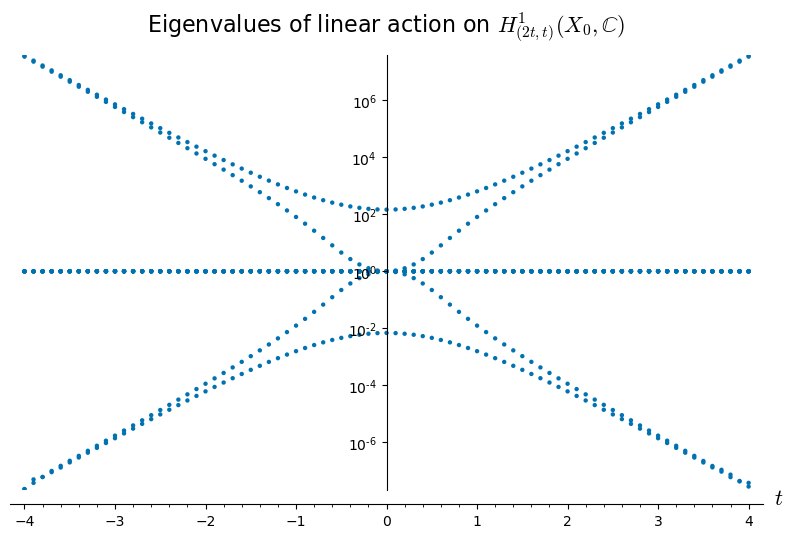}
\captionsetup{width=\textwidth}
\caption{The eigenvalues of the linear action on $H^1_z(X_0,\C)$, for the one-parameter family $\{z = (2t,t):\ t\in \R\}$. On the $x$-axis is the parameter $t$, on the $y$-axis, the blue dots represent the eigenvalues.
One observes that for non-zero values of $z$, the linear action has two eigenvalues greater than 1, two less than 1 and the eigenvalue 1 with multiplicity 8 \big(in $H^1(\tilde X_0,z,\C)$ the 1-eigenspace is the codimension-4 subspace of cohomology classes vanishing on the curves $\gamma^{(0,0)}_{NE,1}, \gamma^{(0,0)}_{NE,2}, \gamma^{(0,0)}_{NW,1}, \gamma^{(0,0)}_{NW,2}$.\big) Thus in this case the eigenvalues (without multiplicity) of the linear action on $H^1_z(X_0,\C)$ and on $H^1_z(X,\C)$ are the same. }
\label{fig:eigenvalues}
\end{figure}
\restoregeometry

\section{Proof of \Cref{prop:leafwise_measures}}
\label{sec:leafwise_measures}
The result of \cite{GoLi2} does not apply directly to our setting since $S_0$ is non-compact and the horizontal foliation is not defined at the singularities. However the proof is not hard to adapt by using the compactness of $S$.

\begin{proof}
Let $f = \frac{\diff\widehat{\mathcal{l}_2}}{\diff\widehat{\ell_1}}$ be the leafwise Radon-Nikodym derivative, which is defined $\widehat{\ell_1}$-almost everywhere on horizontal leaves. Since $|\widehat{\mathcal{l}_2}| \leq \widehat{\ell_1}$, we have $|f|\leq C$. Further, as $\widehat{\ell_1} = g T_* \widehat{\ell_1}$ and $\widehat{\mathcal{l}_2} = \gamma g T_* \widehat{\mathcal{l}_2}$, we have for almost all $x\in S_\mathrm{reg}$, $f(x) = \gamma f(T^{-1}x)$.

Let $I$ be a horizontal interval of length 1 in $S_0$. Consider the sequence of nested partitions $\mathcal{F}_n$ of $I$, where $\mathcal{F}_n$ is the partition into $2^n$ intervals of length $2^{-n}$ each. For $x\in I$, denote by $F_n(x)$ the element of $\mathcal{F}_n$ containing $x$, which is defined for $\widehat{\ell_1}$-almost every $x$. Let $i(n)$ be such that $\lambda^{-1} \leq \lambda^{i(n)} 2^{-(n+1)} < 1$. \\

{ 
Define the finite measure $\mu$ on $I$ by $\mu(A):=\widehat{\ell_1}(A)$ for measurable $A\subset I$ and normalize it to a probability measure $
\bar\mu := \mu(I)^{-1}\mu.$
For each $n$ let $\mathcal{A}_n$ be the sub-$\sigma$-algebra of measurable subsets of $I$ generated by the partition $\mathcal{F}_n$. Then this 
is an increasing sequence of $\sigma$-algebras.

Define the $\mathcal{A}_n$-conditional expectation of $f$ (with respect to $\bar\mu$) by the usual averaging over the element containing $x$:
\[
E_{\bar\mu}[f\mid \mathcal{A}_n](x)
= \frac{1}{\bar\mu(F_n(x))}\int_{F_n(x)} f\,\diff\bar\mu
= \frac{\int_{F_n(x)} f\,\diff\widehat{\ell_1}}{\widehat{\ell_1}(F_n(x))}.
\]
Because $f\in L^\infty(\widehat{\ell_1})$ (indeed $|f|\le C$), we have $f\in L^1(\bar\mu)$ and these conditional expectations are well defined.

\medskip

We now apply the Martingale convergence theorem in our context.
\begin{theorem}[Martingale convergence theorem, {\cite[Theorem 11.5]{Williams}}]
Let $(I,\bar\mu)$ be a probability space and let $\mathcal{A}_n$ be an increasing sequence of sub-$\sigma$-algebras generated by measurable partitions $\mathcal{F}_n$ with $\operatorname{diam}(F_n(x))\to0$ for $\bar\mu$-a.e.\ $x$. If $f\in L^1(\bar\mu)$ then
\[
E_{\bar\mu}[f\mid \mathcal{A}_n](x)\to f(x)
\qquad\text{for }\bar\mu\text{-almost every }x.
\]
Equivalently, for every $\varepsilon>0$,
\[
\frac{\bar\mu\{y\in F_n(x): |f(y)-f(x)|>\varepsilon\}}{\bar\mu(F_n(x))}\to0
\quad\text{for }\bar\mu\text{-a.e. }x.
\]
\end{theorem}

Because $\mathcal{F}_n$’s elements shrink to points a.e., the hypotheses hold, and hence the displayed convergence holds for $\widehat{\ell_1}$-almost every $x\in I$. Concretely, we have

}
 for $\widehat{\ell_1}$-almost every $x\in I$, for any $\epsilon >0$,
\begin{equation}
\label{eq:RN_oscillation}
\frac{\widehat{\ell_1}\{y\in F_n(x): |f(y)-f(x)| > \epsilon\}}{\widehat{\ell_1}(F_n(x))} \to 0  \text{ as } n \to +\infty.
\end{equation}

Fix an $x$ for which the above holds. Let $\bar x_n$ be the center of $F_n(x)$, and let $x_n = T^{-i(n)}(\bar x_n)$. Then $B(x_n,\lambda^{-1}) \subset T^{-i(n)}(F_n(x)) \subset B(x_n,1)$, where $B(x,\delta)$ denotes the one-dimensional ball inside the horizontal leaf. \\

Since $g$ is bounded below by some $g_- > 0$ on $S$, and is Hölder on horizontal leaves, there exists a constant $C > 0$, such that for each $n$, and any $y,z \in F_n(x)$, 
\[\prod_{j=0}^{i(n)-1}g(T^{-j}y) \leq C  \prod_{j=0}^{i(n)-1}g(T^{-j}z).\]
Since $\widehat{\ell_1} = g T_* \widehat{\ell_1}$, it follows that for any set $A \subset F_n(x)$, 
\[\frac1{C} \leq \frac{\widehat{\ell_1} (T^{-i(n)}(A))}{\widehat{\ell_1}(A)} \leq C.\]

Hence by \eqref{eq:RN_oscillation},
\[\frac{\widehat{\ell_1}\big(\{y\in T^{-i(n)}(F_n(x)): |f(T^{i(n)} y)-f(x)| > \epsilon\}\big)}{\widehat{\ell_1}\big(T^{-i(n)}(F_n(x))\big)} \to 0  \text{ as } n \to +\infty.\]

As $\widehat{\ell_1} (T^{-i(n)}F_n(x)) \leq \widehat{\ell_1}(B(x_n,1))$ is uniformly bounded, the numerator of the expression above converges to 0. Since $f(T^{i(n)} y) = \gamma^{\, i(n)} f(y)$, and $T^{-i(n)}(F_n(x)) \supset B(x_n,\lambda^{-1})$, it follows that

\begin{equation}
\label{eq:RN_oscillation2}
\widehat{\ell_1}\big(\{y\in B(x_n,\lambda^{-1}): |\gamma^{\, i(n)}f(y)-f(x)| > \epsilon\}\big) \to 0.
\end{equation}

Take a subsequence along which $x_n$ converges to some $x'$ and $\gamma^{-i(n)}$ converges to some $\gamma'$ with $|\gamma'| = 1$. In the case that $x'$ is a singularity, assume without loss of generality that all $x_n$ are in one half-plane above $x'$. In this case we denote by $B(x',\delta)$ the one-dimensional ball around $x'$ contained in the boundary of that half-plane.

Fix some $u \in \mathscr{C}^{0}(S)$. 
As \eqref{eq:RN_oscillation2} holds for any $\epsilon$ and the subsequence $(x_n)_n$ does not depend on $\epsilon$, it follows that
\[\int_{B(x_n,\lambda^{-1})} u \diff \widehat{\mathcal{l}_2} - f(x)\gamma'\int_{B(x_n,\lambda^{-1})} u \diff \widehat{\ell_1} = 
\int_{B(x_n,\lambda^{-1})} u\, (f - f(x)\gamma') \diff \widehat{\ell_1} \to 0.\]

By the continuity of the measures $\widehat{\ell_1}$ and $\widehat{\mathcal{l}_2}$, we deduce that 
\[\int_{B(x',\lambda^{-2})} u \diff \widehat{\mathcal{l}_2} = f(x)\gamma' \int_{B(x',\lambda^{-2})} u \diff \widehat{\ell_1},\]
and thus on $B(x',\lambda^{-2})$, $\widehat{\mathcal{l}_2} = c_x \widehat{\ell_1}$ for some constant $c_x$.\\

Now we use that $T^{-1}$ is topologically mixing, and contracting on vertical leaves. This means that any interval $I$ contained in a horizontal leaf contains a point $x$, whose orbit under $T^{-1}$ is dense. Pick such a point $y\in B(x',\lambda^{-2})$, and let $r>0$ be such that $B(y,r) \subset B(x',\lambda^{-2})$. Then we see that on $T^{-n}(B(y,r))$, $\widehat{\mathcal{l}_2} = c_x \gamma^{-n} \widehat{\ell_1}$.

By the same argument as above, for any horizontal interval $B(z,r)$, taking a sequence $n_k$ such that $T^{-n_k}(y) \to z$ and $\gamma^{-n_k}\to \gamma_z$ for some $\gamma_z$, we see that on $B(z,r)$, $\widehat{\mathcal{l}_2} = c_x \gamma_z \widehat{\ell_1}$. Thus, letting $f(z) = c_x\gamma_z$, we see that for any $z\in S$, $\widehat{\mathcal{l}_2} = f(z)\widehat{\ell_1}$ on $B(z,r)$. By continuity of the measures, $f$ must be continuous on $S$, and since $f\circ T^{-1} = \gamma^{-1} f$, it follows that $\gamma=1$ and $f$ is constant.
\end{proof}

\subsection*{Acknowledgements}
The project started during the conference ``New Frontiers in Parabolic Dynamics and Renormalization" in Bologna. The authors thank the organisers for making this collaboration possible and, in particular, Corinna Ulcigrai for early encouraging comments.

D.R.~thanks Gabriel Rivière for suggesting that the methods in \cite{DaRi} can be applied to the setting of the present paper, and for patiently explaining their work during several discussions.

This research is part of R.C.’s activity within the INdAM (Istituto Nazionale di Alta Matematica) group GNFM,
and R.C. and D.R.’s activity within the UMI Group “DinAmicI”.

Y.T. was supported by the Swiss National Science Foundation through Grant 213663.

\bibliography{ACRT_mah}
\bibliographystyle{amsplain}

\end{document}